\documentclass[a4paper, reqno]{amsart}
\usepackage[utf8]{inputenc}
\usepackage[T1]{fontenc}
\usepackage{lmodern}
\usepackage{amssymb,amsxtra}
\usepackage[all]{xy}
\usepackage{nicefrac,mathtools,enumitem}
\usepackage{mdwlist}
\setlist[enumerate,1]{label=\textup{(\arabic*)}}
\usepackage{lmodern,textcomp}
\usepackage{a4wide}

\usepackage{dsfont}

\usepackage{graphicx}
\usepackage{mathtools} 
\usepackage[all]{xy}

\usepackage{tikz}
\usetikzlibrary{matrix}
\tikzset{cd/.style=matrix of math nodes,row sep=2em,column sep=2em, text height=1.5ex, text depth=0.5ex}
\tikzset{cdar/.style=->,auto}

\usepackage{microtype}

\usepackage[pdftitle={Quantum isometry groups of the crossed products},
pdfauthor={Jyotishman Bhowmick, Arnab Mandal, Sutanu Roy},
pdfsubject={Mathematics}
]{hyperref}

\usepackage[lite]{amsrefs}

\newcommand*{\arxiv}[1]{\href{http://www.arxiv.org/abs/#1}{arXiv: #1}}

\renewcommand{\PrintDOI}[1]{\href{http://dx.doi.org/\detokenize{#1}}{doi: \detokenize{#1}}}

\BibSpec{book}{%
  +{}  {\PrintPrimary}                {transition}
  +{,} { \textit}                     {title}
  +{.} { }                            {part}
  +{:} { \textit}                     {subtitle}
  +{,} { \PrintEdition}               {edition}
  +{}  { \PrintEditorsB}              {editor}
  +{,} { \PrintTranslatorsC}          {translator}
  +{,} { \PrintContributions}         {contribution}
  +{,} { }                            {series}
  +{,} { \voltext}                    {volume}
  +{,} { }                            {publisher}
  +{,} { }                            {organization}
  +{,} { }                            {address}
  +{,} { \PrintDateB}                 {date}
  +{,} { }                            {status}
  +{}  { \parenthesize}               {language}
  +{}  { \PrintTranslation}           {translation}
  +{;} { \PrintReprint}               {reprint}
  +{.} { }                            {note}
  +{.} {}                             {transition}
  +{} { \PrintDOI}                   {doi}
  +{} { available at \url}            {eprint}
  +{}  {\SentenceSpace \PrintReviews} {review}
  }
\numberwithin{equation}{section}

\theoremstyle{plain}
\newtheorem{theorem}{Theorem}[section]
\newtheorem{lemma}[theorem]{Lemma}
\newtheorem{proposition}[theorem]{Proposition}

\newtheorem{corollary}[theorem]{Corollary}

\theoremstyle{definition}
\newtheorem{definition}[theorem]{Definition}

\theoremstyle{remark}
\newtheorem{remark}[theorem]{Remark}

\newtheorem{example}[theorem]{Example}

\newcommand*{\clc}{{\mathcal C}}

\newcommand*{\Pol}{{\textup{Pol}}}

\newcommand*{\Jnd}{{\mathcal{J}}}

\newcommand*{\bc}{{\mathbb C}}

\newcommand*{\fatlam}{{\boldsymbol{\lambda}}}
\newcommand*{\fatmu}{{\boldsymbol{\mu}}}

\newcommand*{\corep}[1]{\textup{#1}}          
\newcommand*{\Ducorep}[1]{\hat{\corep{#1}}}   

\newcommand*{\GDrin}[1]{\mathfrak{D}_{#1}}
\newcommand*{\GCodb}[1]{\widehat{\mathfrak{D}}_{#1}}

\newcommand*{\Bialg}[1]{(#1,\Comult[#1])}

\newcommand*{\nb}{\nobreakdash}
\newcommand*{\Star}{$^*$\nb-}

\newcommand*{\C}{\textup C}
\newcommand*{\Z}{\mathbb Z}
\newcommand*{\R}{\mathbb R}
\newcommand*{\N}{\mathbb N}

\newcommand*{\T}{\mathbb T}

\newcommand*{\Ltwo}{\textup{L}^2}

\global\long\def\Ww{\mathds{W}}
\global\long\def\wW{\text{\reflectbox{\ensuremath{\Ww}}}\:\!}

\global\long\def\WW{{\mathds{V}\!\!\text{\reflectbox{\ensuremath{\mathds{V}}}}}}

\newcommand*{\QISO}{\textup{QISO}}
\newcommand*{\CLS}{\mathrm{CLS}}
\newcommand*{\G}[1][G]{#1}
\newcommand*{\DuG}[1][G]{\hat{#1}}
\newcommand*{\GH}{H}
\newcommand*{\DuGH}{\hat{H}}

\newcommand*{\Comult}[1][]{\Delta_{#1}}
\newcommand*{\DuComult}[1][]{\hat{\Delta}_{#1}}

\newcommand*{\Bound}{\mathbb B}
\newcommand*{\Comp}{\mathbb K}

\newcommand*{\univ}{\textup u}
\newcommand*{\red}{\textup r}

\newcommand*{\Contvin}{\textup C_0}
\newcommand*{\Cont}{\textup C}

\newcommand*{\Mor}{\textup{Mor}}
\newcommand*{\Id}{\textup{id}}

\newcommand*{\multunit}[1][]{\textup W^{#1}}
\newcommand*{\Dumultunit}[1][]{\widehat{\textup W}^{#1}}

\newcommand*{\bichar}{\textup V}
\newcommand*{\Dubichar}{\widehat{\textup{V}}}

\newcommand*{\Afilt}{\widetilde{A}}
\newcommand*{\Bfilt}{\widetilde{B}}
\newcommand*{\Cfilt}{\widetilde{C}}
\newcommand*{\ABfilt}{\widetilde{A} \boxtimes_\bichar \widetilde{B}}
\newcommand*{\ABCfilt}{\widetilde{A} \boxtimes_{\bichar_1} \widetilde{B}}
\newcommand*{\crfilt}{\widetilde{A} \rtimes_{\beta} \widetilde{B}}

\newcommand*{\Flip}{\Sigma}
\newcommand*{\flip}{\sigma}
\newcommand*{\Cst}{\textup C^*}
\newcommand*{\Cred}{\textup C^*_\textup r}

\newcommand*{\Cstcat}{\mathfrak{C^*alg}}

\newcommand*{\Hils}[1][H]{\mathcal{#1}}
\newcommand*{\Mult}{\mathcal M}
\newcommand*{\U}{\mathcal U}

\newcommand*{\defeq}{\mathrel{\vcentcolon=}}

\begin{document}
\title[Quantum symmetries  of the twisted tensor products]{Quantum symmetries  of the twisted tensor products of $\Cst$-algebras}

\author{Jyotishman Bhowmick}
\email{jyotishmanb@gmail.com}
\address{Statistics and Mathematics Unit\\
 Indian Statistical Institute\\
 203, B. T. Road, Kolkata 700108\\
 India}

\author{Arnab Mandal}
\email{arnab@niser.ac.in}
\address{School of Mathematical Sciences\\
 National Institute of Science Education and Research  Bhubaneswar, HBNI\\
 Jatni, 752050\\
 India}

\author{Sutanu Roy}
\email{sutanu@niser.ac.in}
\address{School of Mathematical Sciences\\
 National Institute of Science Education and Research  Bhubaneswar, HBNI\\
 Jatni, 752050\\
 India}
 
 \author{Adam Skalski}
\email{a.skalski@impan.pl}
\address{Institute of Mathematics of the Polish Academy of Sciences\\
ul. Sniadeckich 8\\
00-656 Warsaw\\
 Poland}


\begin{abstract}
We consider the construction of twisted tensor products in the category of  $\Cst$-algebras equipped with orthogonal filtrations and under certain  assumptions on the form of the twist compute the corresponding quantum symmetry group, which turns out to be the generalized Drinfeld double of the quantum symmetry groups of the original filtrations. We show how these results apply to a wide class of crossed products of $\Cst$-algebras by actions of discrete groups. We also discuss an example where the hypothesis of our main theorem is not satisfied and the quantum symmetry group is not a generalized Drinfeld double. 
\end{abstract}

\maketitle

\section{Introduction}
 \label{sec:Intro} 
The study of quantum symmetry groups (in the framework of compact quantum groups of Woronowicz, \cite{Woronowicz:CQG}) has started from the seminal paper of Wang (\cite{wang}), who studied quantum permutation groups and quantum symmetry groups of finite-dimensional $\Cst$\nb-algebras equipped with reference states. Soon after this, the theory of quantum symmetries was  extended to finite metric spaces and finite graphs by Banica, Bichon and their collaborators (see \cites{ban1 , Banica:Qaut , bichon}, and more recently \cite{moritzgraph} and \cite{soumalyagraph}), who uncovered several interesting connections to combinatorics, representation theory and free probability (\cite{weberraum}, \cite{speicherweber}, \cite{banica_new_expository} and the references therein). The next breakthrough came through the work of Goswami and his coauthors (\cites{goswami, goswami2}), who introduced the concept of quantum isometry groups associated to a given spectral triple \'a la Connes, viewed as a noncommutative differential manifold (for a general description of Goswami's theory we refer to a recent book \cite{goswamibook2}, another introduction to the subject of quantum symmetry groups may be found in the lecture notes \cite{Aubrun-Skalski-Speicher:QSym}). Among examples fitting in the Goswami's framework were the spectral triples associated with the group $\Cst$\nb-algebras of discrete groups, whose quantum isometry groups were first studied in \cite{adamjyotishgroup}, and later analyzed for example in \cite{twoparameterbanicaskalski}, \cite{cyclicgroupsbanicaskalski} and \cite{arnabsingleauthor}.

Historically, the main source of examples of quantum groups was the deformation theory related to the quantum method of the inverse problem and the desire to study a quantum version of the Yang-Baxter equation (\cite{drinfeld}); and in fact already this early work gave rise to the construction of what is now called the Drinfeld double, which plays an important role in this paper. Later, when the theory became to be viewed as one of the instances of the noncommutative mathematics \`a la Connes (\cite{connes}),  there was a hope that by analogy with the classical situation the quantum groups  might arise as quantum symmetries of physical objects appearing in the quantum field theory. It is worth mentioning that   Goswami's theory was in particular applied to compute the  quantum isometry group of the finite  spectral triple corresponding to the standard model in particle physics (\cite{connes}, \cite{connes2}), for which we refer to Chapter 9 of \cite{goswamibook2}.

In fact the examples related to  group $\Cst$\nb-algebras of discrete groups motivated Banica and Skalski to introduce in \cite{MR3066746} a new framework of quantum symmetry groups based on  orthogonal filtrations of unital $\Cst$\nb-algebras, which will be the main focus of our paper. Before we pass to a more specific description, we should mention that Thibault de Chanvalon generalized in \cite{MR3158722} this approach further to orthogonal filtrations of Hilbert $\Cst$\nb-modules. The concept of an orthogonal filtration of a given  unital $\Cst$\nb-algebra  $A$ with a reference state $\tau_A$ is essentially a family of mutually orthogonal (with respect to the scalar product coming from $\tau_A$) finite-dimensional subspaces spanning a dense subspace of $A$. The corresponding quantum symmetry group is the universal compact quantum group acting on $A$ in such a way that the individual subspaces are preserved. The article \cite{MR3066746} proves that such a universal action always exists and discusses several examples. The problems related to the study of quantum symmetry groups in this setup are two-fold: first we need to construct a natural filtration on a $\Cst$\nb-algebra, and then we want to compute the corresponding quantum symmetry group. 

The starting point for this work was an observation that if a $\Cst$\nb-algebra $A$ is equipped with an orthogonal filtration and we have an action of a discrete group $\Gamma$ on $A$, preserving the state $\tau_A$, then the corresponding crossed product admits a natural orthogonal filtration (Proposition \ref{prop15thmarch}), which we will denote $\crfilt$. Note here that the crossed product construction, generalizing that of a group $\Cst$-algebra, on one hand yields a very rich and intensely studied source of examples of operator algebras, and on the other was originally motivated by the desire to model inside the same $\Cst$-algebra both the initial system, and the group acting on it in a way compatible with the action -- in other words, making the action implemented by a unitary representation. The attempts to compute and analyze the resulting  quantum symmetry groups have led, perhaps unexpectedly, to discovering deep connections between the quantum symmetry group construction and the notion of a generalized \emph{Drinfeld double} of a pair of (locally) compact quantum groups linked through a bicharacter, as studied for example in \cite{Baaj-Skandalis:Unitaires, Roy:Codoubles}. This motivated us to extend the original question to the context of twisted tensor products of \cite{Meyer-Roy-Woronowicz:Twisted_tensor}.

Let $ A $ and $ B $ be $ \Cst $-algebras equipped with reduced actions $ \gamma_A $ and $ \gamma_B $ of locally compact quantum groups $G$ and $H$, respectively. Then the twisted tensor product of $ A $ and $ B $,  denoted  $ A \boxtimes_{\bichar_{1}} B $, is a $ \Cst $-algebra defined in terms of  the maps $ \gamma_A, \gamma_B $ and a bicharacter $ \bichar_{1} $ belonging to the unitary multiplier algebra of $\Contvin(\DuG)\otimes\Contvin(\DuG[H])$. For the trivial bicharacter $ \bichar_{1} = 1, $ the  $ \Cst $-algebra $ A \boxtimes_{\bichar_{1}} B $ is the minimal tensor product $ A \otimes B. $ When $ A, B $ are unital $ \Cst $-algebras, $ A \boxtimes_{\bichar_{1}} B $ is a unital  $ \Cst $-algebra. 

Our main result is the following: suppose that $ A $ and $ B $ are two unital $\Cst$\nb-algebras equipped with  orthogonal filtrations $ \Afilt $ and $ \Bfilt $, respectively, that $ \gamma_A, \gamma_B $ are filtration preserving reduced actions of compact quantum groups $G$ and $H$ on $ A $ and $ B, $ and that $ \bichar_{1} \in\U(\Contvin(\DuG)\otimes\Contvin(\DuG[H]))$ is a bicharacter.  Universality of~\(\QISO(\Afilt)\) and~\(\QISO(\Bfilt)\) gives a unique bicharacter $\bichar\in\U(\Contvin( \widehat{\QISO ( \Afilt )}) \otimes  \Contvin( \widehat{\QISO ( \Bfilt )}))$ lifting~\(\bichar_{1}\). 
Then one can always construct a natural orthogonal filtration $\Afilt\boxtimes_{\bichar_{1}}\Bfilt$ on $ A \boxtimes_{\bichar_{1}} B, $  and moreover, the resulting quantum symmetry group $\QISO(\Afilt\boxtimes_{\bichar_{1}}\Bfilt)$ is a generalized Drinfeld double of $\QISO(\Afilt)$ and $\QISO(\Bfilt)$ with respect to~\(\bichar\)  (Theorem \ref{maintheorem} and Theorem~\ref{cor:equiv-cross-iso}). 

Next, we try to apply Theorem \ref{cor:equiv-cross-iso} to compute the quantum symmetry group of $ A \rtimes_{\beta,\red} \Gamma $ in terms of the quantum symmetry groups of $\Afilt$ and $\Bfilt$, where the latter is a natural filtration of $ \Cst_\red ( \Gamma ). $ It turns out (Theorem \ref{maintheoremcrossed}) that if the action $ \beta $ of $ \Gamma $ on $ A $ factors through the action of the quantum symmetry group of $\Afilt $ then indeed, we can apply Theorem \ref{maintheorem} to prove that the quantum symmetry group of $ A \rtimes_{\beta,\red} \Gamma $ is a generalized Drinfeld double of the quantum symmetry groups of $\Afilt $ and ~ $ \Bfilt$. Our result can be applied to a wide class of crossed products, including noncommutative torus, the Bunce-Deddens algebra, crossed products of Cuntz algebras studied by Katsura as well as certain crossed products related to compact quantum groups and their homogeneous spaces.  We also exhibit an example where the hypothesis of Theorem \ref{maintheoremcrossed} does not hold and  the quantum symmetry group of the crossed product is not of the generalized Drinfeld double form.

In order to establish  Theorem \ref{cor:equiv-cross-iso}, we need a certain universal property of the universal $\Cst$\nb-algebra associated with a Drinfeld double. Although in the context of quantum symmetry groups it suffices to work with compact/discrete quantum groups, the property we mention remains true in the general locally compact setting, being a natural framework for studying twisted tensor products and Drinfeld doubles; thus we choose to consider this level of generality in the first few sections of the paper. For locally compact quantum groups, we refer to \cites{kustermans, Kustuniv} and \cite{Woronowicz:Mult_unit_to_Qgrp}; we will in fact only use the $\Cst$\nb-algebraic aspects of the theory.

The plan of the article is as follows: in Section \ref{sec:prelim}, after fixing the notations and conventions for locally compact quantum groups, we discuss the theories of twisted tensor product of $\Cst$\nb-algebras and generalized Drinfeld doubles developed in \cite{Meyer-Roy-Woronowicz:Twisted_tensor} and \cite{Roy:Codoubles}, respectively.  
 In Section \ref{sec:Univ-Drinf} we study the `universal $\Cst$\nb-algebra' associated to the generalized Drinfeld double, and prove two results about its action on twisted tensor products, namely Lemma \ref{lemm:Coact-univ-Drinf} and Theorem \ref{the:UnivProp-Drinf}. 
Section \ref{sec:orthfilt} is devoted to showing that if $A$ and $B$ are unital $\Cst$\nb-algebras equipped with orthogonal filtrations and $\bichar$ is a bicharacter in the unitary multiplier algebra of $ \Contvin( \widehat{\QISO ( \Afilt )}  ) \otimes  \Contvin( \widehat{\QISO ( \Bfilt )}  ), $ then the twisted tensor product $ A \boxtimes_\bichar B $ also admits a natural  orthogonal filtration $\ABfilt$ with respect to the twisted tensor product state.  In Section \ref{sec:sec5} we prove the main result of this article which says that $\QISO(\ABfilt)$ is canonically isomorphic to the generalized Drinfeld double of $ \QISO ( \Afilt )$ and $ \QISO ( \Bfilt )$ with respect to the bicharacter $ \bichar $. Finally in Section \ref{sec:sec6} we show that the above theorem applies to reduced crossed products $ A \rtimes_{\beta,\red} \Gamma $ for group actions of a specific form (Theorem \ref{maintheoremcrossed}). 
We also discuss several natural examples in which the assumptions of Theorem \ref{maintheoremcrossed} are satisfied. In Subsection \ref{counterexample}, we show that the conclusion of Theorem \ref{maintheoremcrossed} fails to hold for more general actions. Finally we discuss further possible extensions of such a framework to twisted crossed products and to crossed products by actions of discrete quantum groups.

\subsection*{Acknowledgements}
The second author was supported by the National Postdoctoral Fellowship given by SERB-DST, Government of India grant no. PDF/2017/001795. The third author was partially supported by an Early Career Research Award given by SERB-DST, Government of India grant no. ECR/2017/001354. The last author was partially supported by the National Science Centre
(NCN) grant no.~2014/14/E/ST1/00525. This work was started during the internship of the second author at IMPAN in 2016, funded by the Warsaw Center for Mathematical Sciences. A.M. thanks A.S. for his kind hospitality at IMPAN. We thank the referees for their thoughtful comments and suggestions.

\section{Preliminaries}
\label{sec:prelim}
Let us fix some notations and conventions. For a normed linear space $ A $, $ A' $ will denote the set of all bounded linear functionals on $ A. $ 
All Hilbert spaces and \(\Cst\)\nb-algebras (which are not explicitly multiplier algebras) 
are assumed to be separable. For a \(\Cst\)\nb-algebra~\(A\), let \(\Mult(A)\) be its multiplier
algebra and let \(\U(A)\) be the group of unitary multipliers
of~\(A\).  For two norm closed subsets \(X\) and~\(Y\) of a 
\(\Cst\)\nb-algebra~\(A\) and~\(T\in\Mult(A)\), let 
\[
XTY\defeq\{xTy\mid x\in X, y\in T\}^\CLS
\]
where CLS stands for the \emph{closed linear span}.

Let~\(\Cstcat\) be the category of \(\Cst\)\nb-algebras with
nondegenerate \Star{}homomorphisms \(\varphi\colon A\to\Mult(B)\) as
morphisms from $A $ to $ B$ (with the composition understood via strict extensions). Moreover, \(\Mor(A,B)\) will denote this set of morphisms. 

Let~\(\Hils\) be a Hilbert space.  A \emph{representation} of a
\(\Cst\)\nb-algebra~\(A\) on~\(\Hils\)  is a nondegenerate
\Star{}homomorphism \(A\to\Bound(\Hils)\).  Since
\(\Bound(\Hils)=\Mult(\Comp(\Hils))\) (where $\Comp(\Hils)$ denotes the algebra of compact operators on $\Hils$) and the nondegeneracy
conditions \(A \Comp(\Hils)=\Comp(\Hils)\) and
\(A \Hils=\Hils\) are equivalent, we have \(\pi\in\Mor(A,\Comp(\Hils))\). 

We write~\(\Flip\) for the tensor flip \(\Hils\otimes\Hils[K]\to
\Hils[K]\otimes\Hils\), \(x\otimes y\mapsto y\otimes x\), for two
Hilbert spaces \(\Hils\) and~\(\Hils[K]\).  We write~\(\flip\) for the
tensor flip isomorphism \(A\otimes B\to B\otimes A\) for two
\(\Cst\)\nb-algebras \(A\) and~\(B\). Further we use the standard `leg' notation for maps acting on tensor products.

\subsection{Quantum groups}
\label{sec:qgrp}
\begin{definition} 
	\label{def:Hopf-Cstar}
	A \emph{Hopf~\(\Cst\)\nb-algebra} is a pair~\(\Bialg{C}\) consisting of a 
	\(\Cst\)\nb-algebra~\(C\) and an element~\(\Comult\in\Mor(C,C\otimes C)\) 
	such that 
	\begin{enumerate}
		\item\label{cond:coasso} \(\Comult[C]\) is~\emph{coassociative}: \((\Comult[C]\otimes\Id_{C})\circ\Comult[C]
		=(\Id_{C}\otimes\Comult[C])\circ \Comult[C]\);
		\item\label{cond:cancellation} \(\Comult[C]\) satisfies the \emph{cancellation conditions}:
		\(\Comult[C](C) (1_C\otimes C) = C\otimes C = (C\otimes 1_C) \Comult[C](C)\).
	\end{enumerate}
	Let~\(\Bialg{D}\) be a Hopf~\(\Cst\)\nb-algebra. A \emph{Hopf~\Star{}homomorphism} from~\(C\) to~\(D\) is an 
	element~\(f\in\Mor(C,D)\) such that~\((f\otimes f)(\Comult[C](c))=\Comult[D](f(c))\) for all~\(c\in C\). 
	
\end{definition}

A \emph{compact quantum group} or CQG, in short, is described by a Hopf~\(\Cst\)\nb-algebra \(\Bialg{C}\) such 
that~\(C\) is unital (in which case $C$ is often called a \emph{Woronowicz algebra}). Hopf~\Star{}homomorphisms between CQGs are called \emph{CQG morphisms}.

\begin{example}
	\label{ex:CQG}
	Every compact group~\(G\) can be viewed as a compact quantum group by 
	setting~\(C=\Cont(G)\) and~\((\Comult[C]f)(g_1,g_2)\defeq f(g_1g_2)\) for all 
	\(f\in\Cont(G)\), \(g_{1},g_{2}\in G\). Also, every discrete group~\(\Gamma\) gives rise to a compact quantum group by 
	fixing~\(C=\Cred(\Gamma)\) and \(\Comult[C](\lambda_{g})\defeq \lambda_{g}\otimes\lambda_{g}\), where 
	\(\lambda\) is the regular representation of~\(\Gamma\) on~\(L^{2}(\Gamma)\) and $g \in \Gamma$ is arbitrary; we could as well choose here the full group $\Cst$\nb-algebra $\Cst(\Gamma)$, as will be discussed below.
\end{example} 

Nonunital Hopf~\(\Cst\)\nb-algebras are noncommutative analogue of locally compact semigroups satisfying the cancellation property. The question of how one should define locally compact quantum \emph{groups} was studied for many years, with the approach by multiplicative unitaries initiated by  Baaj and Skandalis~\cite{Baaj-Skandalis:Unitaires} and later developed   in~\cite{Woronowicz:Mult_unit_to_Qgrp} by Woronowicz, and a generally accepted notion based on von Neumann algebraic techniques proposed  by Kustermans and Vaes in \cite{kv} (see also a later paper \cite{MNW}). One should note here also an earlier von Neumann algebraic approach presented in \cite{MasudaNakagami}, and also the fact that both articles \cite{MNW} and \cite{kv} drew on the algebraic duality techniques of \cite{VDDuality}.

Thus a \emph{locally compact quantum group} $\G$ is a virtual object studied via its associated operator algebras, in particular the von Neumann algebra $\textup{L}^\infty (\G)$ equipped with a comultiplication $\Delta: \textup{L}^\infty (\G) \to \textup{L}^\infty (\G) \overline{\otimes} \textup{L}^\infty (\G)$ and the left and right \emph{Haar weights}. We assume that $\textup{L}^\infty (\G)$ is represented on $\Ltwo(\G)$, the GNS-Hilbert space of the right Haar weight. The key fact connecting the approach of Kustermans and Vaes with these of \cite{Baaj-Skandalis:Unitaires} and \cite{Woronowicz:Mult_unit_to_Qgrp} is the existence of a distinguished unitary $\multunit[\G]\in \U(\Ltwo(\G) \otimes \Ltwo(\G))$, whose properties we will now describe.

Let~\(\Hils\) be a Hilbert space. Recall that a  \emph{multiplicative unitary} is an element 
\(\multunit\in\U(\Hils\otimes\Hils)\) 
that satisfies the 
\emph{pentagon equation}
\begin{equation}
\label{eq:pentagon}
\multunit_{23}\multunit_{12}
= \multunit_{12}\multunit_{13}\multunit_{23}
\qquad
\text{in \(\U(\Hils\otimes\Hils\otimes\Hils).\)}
\end{equation}  
Define~\(C\defeq \{(\omega\otimes\Id_{\Hils})(\multunit)\mid
\omega\in\Bound(\Hils)_*\}^\CLS\). When~\(\multunit\) is \emph{manageable} (see~\cite{Woronowicz:Mult_unit_to_Qgrp}*{Theorem 1.5}), \(C\) is a separable nondegenerate \(\Cst\)\nb-subalgebra of~\(\Bound(\Hils)\).
Moreover, the formula \(\Comult[C](c)\defeq \multunit (c\otimes 1)\multunit[*]\) for~\(c\in C\) defines  an element of~\(\Mor(C,C\otimes C)\) 
such that~\(\Bialg{C}\) is a Hopf~\(\Cst\)\nb-algebra (see ~\cites{Soltan-Woronowicz:Multiplicative_unitaries, Woronowicz:Mult_unit_to_Qgrp}), 
which is said to be \emph{generated} by~\(\multunit\). 

In particular when we return to the Kustermans-Vaes setup, we denote the $\Cst$\nb-algebra $C$ associated to $\multunit[\G]$ (with $\Hils=\Ltwo(\G)$) as $\Contvin(\G)$. Note that $\Contvin(\G) \subset \textup{L}^{\infty}(\G)$ and the respective coproducts are compatible and will both be denoted by $\Comult[\G]$ (or $\Comult$ if the context is clear).

The \emph{dual} of a multiplicative unitary~\(\multunit\in\U(\Hils\otimes\Hils)\) is given by the formula \(\Dumultunit\defeq \Flip\multunit[*]\Flip\in\U(\Hils\otimes\Hils)\). Moreover, \(\Dumultunit\) is manageable whenever~\(\multunit\) is. It turns out that if we start from a locally compact quantum group $\G$ we can associate to it another  locally compact quantum group which we denote $\DuG$ and call the dual locally compact quantum group of $\G$, so that $\textup{W}^{\DuG} = \Dumultunit$. We naturally have  \(\Contvin(\DuG)\defeq \{(\omega\otimes\Id_{\Hils})(\Dumultunit)\mid \omega\in\Bound(\Hils)_{*}\}^\CLS\) and \(\DuComult(\hat{c})\defeq\Dumultunit(\hat{c}\otimes 1)\Dumultunit[*]\) for all~\(\hat{c}\in\Contvin(\DuG)\).  Since the dual of~\(\Dumultunit\) is equal to \(\multunit\), we obtain a canonical isomorphism $\hat{\DuG} \approx \G$.

\begin{example}
	\label{ex:classical}
	Let~\(G\) be a locally compact group and \(\mu\) be its right Haar measure. Then the operator \(\multunit[G]\in\U(\Ltwo(G,\mu)\otimes \Ltwo(G,\mu))\) defined 
	by~\((\multunit[G] f) ( g_1, g_2 ) = f( g_1 g_2, g_2 )\), $f\in \Ltwo(G,\mu), g_1,g_2 \in G$, is a (manageable) multiplicative unitary. The resulting locally compact quantum group is given simply by the algebra $\Contvin(G)$ with the comultiplication determined by the formula \((\Comult[G] f) (g_1, g_2)=f(g_1g_2)\) for all~\(f\in\Contvin(G)\), \(g_1, g_2\in G\). Further we have $\Contvin(\hat{G}) = \Cred(G)$, and the dual comultiplication is determined by the formula \(\DuComult[G](\lambda_g)=\lambda_g\otimes \lambda_g\) for all~\(g \in G\). 
\end{example}
Let then $\G$ be a locally compact quantum group in the sense described above. We then write \(1_{\G}\) for the identity element of~\(\Mult(\Contvin(\G))\). By virtue of~\cite{Woronowicz:CQG}*{Theorem 1.3} and~\cite{kustermans}*{Theorem 3.16 \& Proposition 3.18}, every compact quantum group can be also viewed as a locally compact quantum group; for a compact quantum group $\G$ we naturally write $\C(\G)$ for the (unital) $\Cst$\nb-algebra $\Contvin(\G)$. \emph{Discrete quantum groups} are duals of CQGs. 

\begin{definition}
	\label{def:corepresentation}
	Let~\(\Bialg{C}\) be a Hopf~\(\Cst\)\nb-algebra and~\(D\) be a~\(\Cst\)\nb-algebra. 
	\begin{enumerate} 
		\item An element~\(\corep{U}\in\U(D\otimes C)\) is said to be a 
		\emph{right corepresentation} of~\(C\) in~\(D\) if and only if~\((\Id_{D}\otimes\Comult[C])(\corep{U})=\corep{U}_{12}\corep{U}_{13}\) in~\(\U(D\otimes C\otimes C)\). 
		\item Similarly, a \emph{left corepresentation} of~\(C\) in~\(D\) is an element~\(\corep{U}\in\U(C\otimes D)\) satisfying 
		\((\Comult[C]\otimes\Id_{D})(\corep{U})=\corep{U}_{23}\corep{U}_{13}\) in~\(\U(C\otimes C\otimes D)\). 
		\item In particular, if~\(D=\Comp(\Hils[L])\) for some 
		Hilbert space~\(\Hils[L]\), then~\(\corep{U}\) is said to be (right or left) corepresentation of~\(C\) on~\(\Hils[L]\).
	\end{enumerate}
	An element~\(\corep{U}\in\U(D\otimes C)\) is a right corepresentation of~\(C\) in~\(D\) if and only if~\(\Ducorep{U}\defeq\flip(\corep{U}^{*})\in\U(C\otimes D)\) is a left corepresentation of~\(C\) in~\(D\). From now on we reserve the word ``corepresentation'' for right corepresentations. Moreover if $\G$ is a locally compact quantum group then we call right corepresentations of $\Contvin(\G)$ simply (\emph{unitary, strongly continuous) representations} of $\G$.
\end{definition}

Consider a locally compact quantum group $\G$ and the  manageable multiplicative unitary ~\(\multunit[\G]\). Then \(\multunit[\G]\in\U(\Contvin(\DuG)\otimes\Contvin(\G))\subset\U(\Ltwo(\G)\otimes\Ltwo(\G))\). Thus we can also view  ~\(\multunit[\G]\)  as a unitary element of the \emph{abstract} $\Cst$\nb-algebra $\Mult(\Contvin(\DuG)\otimes\Contvin(\G))$,  called the \emph{reduced bicharacter} of~\(\G\).
Indeed, it satisfies the following bicharacter conditions (see Definition \ref{def:bicharacter}):
\begin{alignat}{2}
\label{eq:Comult_W}
(\Id_{\Contvin(\DuG)}\otimes \Comult[\G])(\multunit[\G])
&= \multunit[\G]_{12}\multunit[\G]_{13}
&\qquad \text{in \(\U(\Contvin(\DuG)\otimes \Contvin(\G)\otimes \Contvin(\G))\);}\\
\label{eq:dual_Comult_W}
(\Comult[\DuG]\otimes\Id_{\Contvin(\G)})(\multunit[\G])
&=\multunit[\G]_{23}\multunit[\G]_{13} 
&\qquad\text{in \(\U(\Contvin(\DuG)\otimes\Contvin(\DuG)\otimes \Contvin(\G))\).}
\end{alignat}  
Therefore, \(\multunit[\G]\) is a representation of $\G$ in~\(\Contvin(\DuG)\) and also a
(left) representation of $\DuG$ in~\(\Contvin(\G)\).

Finally note that in fact the Haar weights will not play any significant role in this paper and we will concentrate on the $\Cst$\nb-algebraic setup.

\subsection{Universal algebras of locally compact quantum groups}
\label{subsec:univ}

Suppose $G$ is a locally compact group. If $G$ is not amenable (for example, $ G = \mathbb{F}_2 $), the convolution algebra  of compactly supported continuous functions on $G$ can have more than one $\textup{C}^{*}$-algebraic completion. By completing the convolution algebra of $G $ in the norm topology of $ \mathcal{B} ( L^2 ( G ) ), $ we get the reduced group $\textup{C}^{*}-$algebra $ \textup{C}^{*}_\red ( G ). $ On the other hand, the quantity
$$ \left\| f \right\|_{\univ}: = {\rm sup}_\pi \{ \left\| \pi ( f ) \right\| :~ \pi ~ {\rm is} ~ {\rm a} ~ \ast- ~ {\rm representation} ~ {\rm of} ~ \C_c ( G )      \} $$
defines a pre-$\textup{C}^{*}$-algebraic norm on the convolution algebra $ \C_c ( G ). $ The completion of $ \C_c ( G ) $ with respect to $ \left\| \cdot \right\|_\univ $ defines a $\textup{C}^{*}$ algebra known as the full group $\textup{C}^{*}$ algebra of $ G $ and denoted by $\C^\ast ( G ). $ There always exists a canonical surjective $\textup{C}^{*}$-homomorphism from $ \C^\ast ( G ) $ to $ \textup{C}^{*}_\red ( G ). $ When $ G $ is not amenable,  this homomorphism is not one-to-one. For more details, we refer to \cite{pedersen}. This explains the need to associate with any locally compact quantum group, apart from the `reduced algebra of functions', discussed in the last subsection, also its `universal' counterpart, which we describe in what follows.

Let $\G$ be a locally compact quantum group. By ~\cite{Kustuniv} (see also ~\cite{Soltan-Woronowicz:Multiplicative_unitaries}*{Proposition~22 \& Theorem 25}), the 
algebra $\Contvin(\G)$ admits also a \emph{universal version}, denoted $\Contvin^{\univ}(\G)$, characterized by the fact that there exists a canonical 1-1 correspondence between representations of $\DuG$ and $\Cst$\nb-algebraic representations of $\Contvin^{\univ}(\G)$, implemented by the `semi-universal' version of the multiplicative unitary $\textup{W}^{\G}$, denoted $\wW^{\G}$. We have $\wW^{\G} \in \Mult(\Contvin(\DuG) \otimes \Contvin^{\univ}(\G))$, and $\wW^{\G}$ is in fact a representation of $\DuG$ in $\Contvin^{\univ}(\G)$. The correspondence mentioned above is of the following form:  given any (left) representation~\(\corep{U}\in\U(\Contvin(\DuG) \otimes D)\) of~$\DuG$ in a~\(\Cst\)\nb-algebra~\(D\) there is a unique~\(\varphi\in\Mor(\Contvin^{\univ}(\G),D)\) with
\begin{equation}
\label{eq:univ_dumaxcorep}
(\Id \otimes \varphi)(\wW^{\G})=\corep{U}.
\end{equation}
The algebra $\Contvin^{\univ}(\G)$ is equipped with the comultiplication $\Delta^{\univ} \in\Mor(\Contvin^\univ(\G), \Contvin^\univ(\G)\otimes\Contvin^\univ(\G))$ (as well as a `bi-universal multiplicative unitary', see below)  making $\Contvin^{\univ}(\G)$ a Hopf~\(\Cst\)\nb-algebra (see~\cite{Soltan-Woronowicz:Multiplicative_unitaries}*{Proposition~31}). By certain abuse of notation we will write $\Id_{\G}$ for both $\Id_{\Contvin(\G)}$ and $\Id_{\Contvin^\univ(\G)}$, with the specific meaning clear from the context. For clarity, we will also sometimes write $\Delta^{\univ}_{\G}$ for $\Delta^{\univ}$.

Further note that the universality of~\(\wW^{\G}\) gives a unique \(\Lambda_{\G}\in\Mor(\Contvin^\univ(\G),\Contvin(\G))\), known as the \emph{reducing morphism}, satisfying the equations
\begin{equation}
\label{eq:dumaxcorep_to_multunit}
(\Id \otimes\Lambda_{\G})(\wW^{\G}) = \multunit[\G], 
\qquad 
(\Lambda_{\G}\otimes\Lambda_{\G})\circ \Comult^\univ
=\Comult \circ \Lambda_{\G}.
\end{equation}

Another application of the universal property yields the existence of the character $e: \Contvin^{\univ}(\G)\to \bc$, called the \emph{counit}, associated to the trivial 
representation of~\(\DuG\):
\begin{equation}
\label{eq:bounded_dual_counit} 
(\Id \otimes e)(\wW)= 1_{\Contvin(\DuG)},
\end{equation}
with the following property:
\begin{equation}
\label{eq:bdd_dual_counit_and_univ_dual_comult}
(e\otimes\Id)\circ\Comult^{\univ}
= (\Id\otimes e)\circ \Comult^{\univ}
= \Id_{\Contvin^{\univ}(\G)}.
\end{equation}

Naturally the construction described above can be applied also to the dual locally compact quantum group $\DuG$, yielding the (Hopf)-$\Cst$\nb-algebra $\Contvin^{\univ}(\DuG)$, the unitary $\Ww^{\G} \in \Mult(\Contvin^{\univ}(\DuG) \otimes \Contvin(\G))$, reducing morphism $\Lambda_{\DuG}$, etc. In fact Kustermans showed the existence of the `fully universal' multiplicative unitary $\WW^{\G} \in \Mult(\Contvin^{\univ}(\DuG) \otimes \Contvin^{\univ}(\G))$ such that 
\[ \multunit[\G] = (\Lambda_{\DuG} \otimes \Lambda_{\G})(\WW^{\G}).\]

We call $\G$ \emph{coamenable} if the reducing morphism $\Lambda_{\G}$ is an isomorphism. This is the case for $\G$ classical or discrete. On the other hand
the dual of a classical locally compact group  \(G\) is coamenable if and only if $G$ is amenable. Indeed, we have $\Contvin^{\univ}(\hat{G})= \textup{C}^{*}(G)$, with the comultiplication ~\(\Comult[\hat{G}]^{\univ}(u_g)=u_g\otimes u_g\) for all~\(g\in G\), where $u_g$ denotes the image of $g$ under the universal representation. 
By analogy with the notation introduced before, for a compact quantum group $\G$ we will write simply $\C^u(\G)$ for $\Contvin^{\univ}(\G)$. Each of the algebras $\C(\G)$ and $\C^u(\G)$ contain then a canonical dense Hopf *-subalgebra, $\textup{Pol} (\G)$, and in fact $\textup{Pol} (\G)$ may admit also other, so-called \emph{exotic completions} to Hopf $\Cst$\nb-algebras (see~\cite{MR2990124}).

\subsection{Coactions, bicharacters and quantum group morphisms}
In this short subsection we discuss quantum group actions and the notion of morphisms between locally compact quantum groups.
\begin{definition}
	\label{def:coaction}
	Let~\(\Bialg{C}\) be a Hopf~\(\Cst\)\nb-algebra. A \emph{\textup(right\textup) coaction} of~\(C\) on a \(\Cst\)\nb-algebra 
	\(A\) is an element~\(\gamma \in\Mor(A, A\otimes C)\) such that
	\begin{enumerate}
		\item \(\gamma \) is a comodule structure (in other words, it satisfies the action equation), that is,
		\begin{equation}
		\label{eq:right_action}
		(\Id_A\otimes\Comult[C])\circ \gamma
		= (\gamma\otimes\Id_{C})\circ \gamma;
		\end{equation}
		\item \(\gamma\) satisfies the \emph{Podle\'s condition}:
		\begin{equation}
		\label{eq:Podles_cond}
		\gamma(A)(1_A\otimes C)=A\otimes C.
		\end{equation}
	\end{enumerate}
	In the case where $\Bialg{C} = (\Contvin(\G), \Comult[\G])$ for a locally compact quantum group $\G$, we call $\gamma$ simply the (reduced) action of $\G$ on $A$, and  
	the pair \((A,\gamma )\) is called a~\(\G\)\nb-\(\Cst\)\nb-algebra. Sometimes one needs also to consider \emph{universal} actions of $\G$ (i.e.\ the coactions of the Hopf $\Cst$\nb-algebra $\Contvin^{\univ}(\G)$). 
\end{definition}

Let us  record  a  consequence of the Podle\'s condition for \Star{}homomorphisms.
\begin{lemma}
	\label{lemma16thjan3}
	Let~\(A, B, C\) be \(\Cst\)\nb-algebras with~\(A\subseteq\Mult(D)\) and let~\(\gamma\colon A\to\Mult(D\otimes C)\) be 
	a \Star{}homomorphism such that \(\gamma(A)(1_{D}\otimes C)=A\otimes C\). Then 
	\(\gamma\in\Mor(A,A\otimes C)\).
\end{lemma}
\begin{proof}
	 The Podle\'s condition 
	implies that ~\(\gamma(A)(A\otimes C)=\gamma(A)(1_{D}\otimes C)(A\otimes 1_{C})=(A\otimes C)(A\otimes 1_{C})=A\otimes C\). Applying the adjoint to the last equality we also get $(A\otimes C) \gamma(A) = A \otimes C$. Thus 	\(\gamma\in\Mor(A,A\otimes C)\).
\end{proof}

A \emph{covariant representation} of the coaction  \((A,\gamma)\) on a Hilbert
space~\(\Hils\) is a pair~\((\corep{U},\varphi)\) consisting of a 
corepresentation \(\corep{U}\in\U(\Comp(\Hils)\otimes C)\) of~\(C\) on~\(\Hils\)  
and a representation \(\varphi\in\Mor(A, \Comp(\Hils))\) that satisfy 
the covariance condition
\begin{equation}
\label{eq:covariant_corep}
(\varphi\otimes\Id_{C})(\gamma(a)) =
\corep{U}(\varphi(a)\otimes 1_{C})\corep{U}^*,
\qquad\qquad a\in A.
\end{equation}
A covariant representation is called
\emph{faithful} if~\(\varphi\) is faithful. Faithful covariant representations always exist whenever~\(\gamma\) is injective (see~
\cite{Meyer-Roy-Woronowicz:Twisted_tensor}*{Example~4.5}).

In this article we are mainly going to work with compact quantum groups~\(\G\) and unital \(\G\)\nb-\(\Cst\)\nb-algebras \((A,\gamma)\) with a faithful state~\(\tau\) on $A$ such that~\(\gamma\) \emph{preserves} \(\tau\), i.e.\
\[
(\tau\otimes\Id_{\G})(\gamma(a))=\tau(a)1_{\G}
\qquad\text{for all $a\in A$,}
\]
Let us note that  such actions are always injective.
\begin{lemma}
	\label{lemm:faithful-state}
	Let~\((A,\gamma)\) be a unital~\(\G\)\nb-\(\Cst\)\nb-algebra with a faithful state~\(\tau\) such that 
	\(\gamma\) preserves~\(\tau\). Then~\(\gamma\) is injective. 
\end{lemma}

We will now define bicharacters associated to two given Hopf $\Cst$\nb-algebras. This notion plays a fundamental role in describing morphisms between locally compact quantum groups.

\begin{definition}
	\label{def:bicharacter}
	Let~\(\Bialg{C}\) and~\(\Bialg{D}\) be Hopf~\(\Cst\)\nb-algebras. An element 
	\(\bichar\in\U(C\otimes D)\) is said to be a \emph{bicharacter} if it satisfies the 
	following properties
	\begin{alignat}{2}
	\label{eq:bichar_char_in_first_leg}
	(\Comult[C]\otimes\Id_{D})(\bichar)
	&=\bichar_{23}\bichar_{13}
	&\qquad &\text{in }
	\U(C\otimes C\otimes D),\\
	\label{eq:bichar_char_in_second_leg}
	(\Id_{C}\otimes\Comult[D])(\bichar)
	&=\bichar_{12}\bichar_{13}
	&\qquad &\text{in }
	\U(C\otimes D\otimes D).
	\end{alignat}
\end{definition}

The notion of bicharacter for quantum groups is a generalization of that  of bicharacters for groups. Indeed, let $G$ and be $H$ be locally compact abelian groups. An element $\bichar\in\U(\Contvin(\hat{G})\otimes\Contvin(\hat{H}))$ is simply a continuous map $\bichar\colon \hat{G}\times\hat{H}\to\T$ and the conditions~\eqref{eq:bichar_char_in_first_leg} and~\eqref{eq:bichar_char_in_second_leg} say that this map is a bicharacter in the classical sense.

Let $\G$ and $\G[H]$ be locally compact quantum groups.

\begin{example}
	\label{ex:bichar}
	A Hopf~\Star{}homomorphism~\(f\in\Mor(\Contvin(\G),\Contvin(\G[H]))\) \emph{induces} a 
	bicharacter~\(\bichar_{f}\in\U(\Contvin(\DuG)\otimes\Contvin(\G[H]))\) defined by 
	\(\bichar_{f}\defeq(\Id_{\DuG}\otimes f)(\multunit[\G])\). The `semi-universal' multiplicative unitaries
	\(\Ww\in\U(\Contvin^{\univ}(\DuG)\otimes\Contvin(\G))\) and \(\wW\in\U(\Contvin(\DuG)\otimes \Contvin^{\univ}(\G))\) 
	are both bicharacters.
\end{example}  

Bicharacters in \(\U(\Contvin(\DuG)\otimes\Contvin(\G[H]))\) are interpreted as quantum
group morphisms from~\(\G\) to~\(\G[H]\). The article~\cite{Meyer-Roy-Woronowicz:Homomorphisms} contains a detailed study of such morphisms and provides several equivalent pictures.   Since the notion of bicharacters is going to be crucially used in the article, let us recall some of the related definitions and results from that paper. For simplicity we will describe below quantum group morphisms from \(\G\) to~\(\DuG[H]\).

 A \emph{right quantum group homomorphism} from~\(\G\) to~\(\DuG[H]\) is
an element \(\Delta_R\in\Mor(\Contvin(\G),\Contvin(\G)\otimes\Contvin(\DuG[H]))\) with the following 
properties:
\begin{equation}
\label{eq:right_homomorphism}
(\Comult[\G]\otimes\Id_{\DuG[H]})\circ\Delta_{R} 
=  (\Id_{\G}\otimes\Delta_{R})\circ\Comult[\G]
\quad\text{and}\quad
(\Id_{\G}\otimes\Comult[\DuGH])\circ\Delta_{R}
=  (\Delta_{R}\otimes\Id_{\DuG[H]})\circ\Delta_{R}.
\end{equation}

\begin{theorem}[\cite{Meyer-Roy-Woronowicz:Homomorphisms}]
	\label{the:equiv_homs} 
	Let $\G$ and $\G[H]$ be locally compact quantum groups.
	There are natural bijections between the following sets:
	\begin{enumerate}
		\item bicharacters \(\bichar\in\U(\Contvin(\DuG)\otimes\Contvin(\DuG[H]))\)\textup{;}
		\item bicharacters \(\Dubichar\in\U(\Contvin(\DuG[H])\otimes\Contvin(\DuG))\)\textup{;}
		\item right quantum group homomorphisms \(\Delta_R \in\Mor(\Contvin(\G),\Contvin(\G)\otimes\Contvin(\DuG[H]))\)\textup{;}
		\item Hopf~\Star{}homomorphisms \(f\in\Mor(\Contvin^{\univ}(\G),\Contvin^{\univ}(\DuG[H]))\)\textup{.}
	\end{enumerate}
	The first bijection maps a bicharacter~\(\bichar\) to its dual 
	\(\Dubichar\in\U(\Contvin(\DuG[H])\otimes\Contvin(\DuG))\) defined by
	\begin{equation}
	\label{eq:dual_bicharacter}
	\Dubichar\defeq\flip(\bichar^*).
	\end{equation}
	A bicharacter~\(\bichar\) and a right quantum group
	homomorphism~\(\Delta_R\) determine each other uniquely via
	\begin{equation}
	\label{eq:def_V_via_right_homomorphism}
	(\Id_{\DuG} \otimes \Delta_R)(\multunit[\G])
	= \multunit[\G]_{12}\bichar_{13}.
	\end{equation}
	
	A bicharacter \(\bichar\in\U(\Contvin(\DuG)\otimes\Contvin(\DuG[H]))\) and a Hopf~\Star{}homomorphism \(f\in\Mor(\Contvin^{\univ}(\G),\Contvin^{\univ}(\DuG[H]))\) determine each other uniquely by 
	\begin{equation}
	\label{eq:bicha_univmor}
	(\Id_{\G}\otimes\Lambda_{\DuG[H]}\circ f)(\wW^{\G})=\bichar.
	\end{equation}
\end{theorem} 

The dual bicharacter~\(\Dubichar\in\U(\Contvin(\DuG[H])\otimes\Contvin(\DuG))\) should be thought of as the dual quantum group morphism  from $\G[H]$ to $\DuG$. It corresponds to  the dual right quantum group homomorphism 
\(\hat{\Delta}_{R}\in\Mor(\Contvin(\DuG[H]),\Contvin(\DuG[H])\otimes\Contvin(\G))\). Thus~\(\Delta_R\) 
and~\(\hat{\Delta}_{R}\) are in bijection as~\(\bichar\) and~\(\Dubichar\) 
are.
Finally, the dual bicharacter describes a unique Hopf~\Star{}homomorphism \(\hat{f}\in\Mor(\Contvin^{\univ}(\G[H]),\Contvin^{\univ}(\DuG))\). 
Thus~\(f\) and~\(\hat{f}\) determine each other uniquely by 
\[
(\Id_{\G}\otimes\Lambda_{\DuG[H]}\circ f)(\wW^{\G})=\bichar
=\flip ((\Id_{\G[H]}\otimes\Lambda_{\DuG}\circ\hat{f})(\wW^{\G[H]})^{*})
\]

Let~\(\gamma^\univ\) be a coaction of~\(\Contvin^{\univ}(\G)\) on a $\Cst$\nb-algebra~\(A\). 
It is said to be \emph{normal} if the associated reduced action of $\G$ on $A$,
\(\gamma\defeq(\Id_{A}\otimes\Lambda_{\G})\circ \gamma^\univ\),  is injective. By virtue of \cite{Roy-Timmermann:Max_twisted_tensor}*{Theorem A.2} or~\cite{Fischer:thesis}, an injective 
coaction of~\(\Contvin(\G)\) on a \(\Cst\)\nb-algebra~\(A\) 
lifts uniquely to a normal coaction of~\(\Contvin^{\univ}(\G)\) on~\(A\).

Yet another equivalent description of quantum group homomorphisms~\cite{Meyer-Roy-Woronowicz:Homomorphisms}*{Section 6} 
shows that for an  injective coaction \(\gamma\in \Mor(A, A\otimes \Contvin(\G))\) and a  right quantum group homomorphism 
\(\Delta_R\in\Mor(\Contvin(\G),\Contvin(\G)\otimes\Contvin(\DuG[H]))\) there is a unique injective coaction \(\delta\in \Mor (A,A\otimes \Contvin(\DuG[H]))\) such that  the following diagram commutes:
\begin{equation}
 \label{eq:right_quantum_group_homomorphism_as_fucntor}
  \begin{tikzpicture}[baseline=(current bounding box.west)] 
    \matrix(m)[cd,column sep=4.5em]{
     A&A\otimes \Contvin(\G)\\
     A\otimes\Contvin(\DuG[H])& A\otimes \Contvin(\G)\otimes \Contvin(\DuG[H])\\
     };
    \draw[cdar] (m-1-1) -- node {\(\gamma\)} (m-1-2);
    \draw[cdar] (m-1-1) -- node[swap] {\(\delta\)} (m-2-1);
    \draw[cdar] (m-1-2) -- node {\(\Id_A\otimes\Delta_R\)} (m-2-2);
    \draw[cdar] (m-2-1) -- node[swap] {\(\gamma\otimes\Id_{\DuG[H]}\)} (m-2-2);
 \end{tikzpicture}
 \end{equation}  
 Let~\(\gamma^\univ\) denote the associated normal coaction of~\(\Contvin^{\univ}(\G)\) on~\(A\) and let 
 \(f\in\Mor(\Contvin^{\univ}(\G),\Contvin^{\univ}(\DuG[H]))\) be the Hopf~\Star{}homomorphism corresponding  to~\(\Delta_{R}\). Then we say that~\(\delta\) is \emph{induced from~\(\gamma\)} by~\(\Delta_{R}\) or~\(f\). 
The  following lemma describes~\(\delta\) explicitly.

\begin{lemma}
	\label{lem:Feb9}
	The coaction \(\delta\) in~\textup{\eqref{eq:right_quantum_group_homomorphism_as_fucntor}} is given by 
	\(\delta\defeq (\Id_{A}\otimes\Lambda_{\DuG[H]}\circ f)\circ \gamma^\univ\).
\end{lemma}
\begin{proof}

By the uniqueness of $ \delta, $ it is enough to show that $ \delta $ satisfies the equation 
$$ (\Id_{A}\otimes\Delta_{R})\circ\gamma = (\gamma\otimes\Id_{\DuG[H]})\circ\delta. $$
Let us 	recall \cite{Roy-Timmermann:Max_twisted_tensor}*{Equation 2.23}: 
	\(\Delta_{R}\circ\Lambda_{\G}=(\Lambda_{\G}\otimes\Lambda_{\DuG[H]}\circ f)\circ\Comult[\G]^{\univ}\).
	Using this equation and properties of~\(\gamma\) and~\(\gamma^\univ\) we obtain
	\begin{align*}
	(\Id_{A}\otimes\Delta_{R})\circ\gamma
	= (\Id_{A}\otimes\Delta_{R}\circ\Lambda_{\G})\circ\gamma^\univ
	&= (\Id_{A}\otimes(\Lambda_{\G}\otimes\Lambda_{\DuG[H]}\circ f)\circ\Comult[{\G}]^{\univ})\circ\gamma^\univ \\
	&=  ((\Id_{A}\otimes\Lambda_{\G})\circ\gamma^\univ\otimes\Lambda_{\DuG[H]}\circ f)\circ\gamma^\univ\\
	&=(\gamma\otimes\Id_{\DuG[H]})\circ\delta .\qedhere
	\end{align*}
\end{proof}

\subsection{Twisted tensor product \texorpdfstring{$\Cst$}{C*}-algebras and the action of the Drinfeld double} \label{twistedtensprod}

In this subsection we recall the notion of quantum group twisted tensor product of \(\Cst\)\nb-algebras as developed in \cite{Meyer-Roy-Woronowicz:Twisted_tensor} and the action of the generalized Drinfeld  double on the twisted tensor product constructed in~\cite{Roy:Codoubles}. 

We start with the following data: let \(\G\) and~\(\G[H]\) be locally compact quantum groups, let \((A,\gamma_{A})\), 
 \((B,\gamma_{B})\), be a~\(\G\)\nb-\(\Cst\)\nb-algebra, respectively a~\(\G[H]\)\nb-\(\Cst\)\nb-algebra, such that~\(\gamma_{A}\) and~\(\gamma_{B}\) are injective maps and let 
\(\bichar\in\U(\Contvin(\DuG)\otimes\Contvin(\DuG[H]) ) \) be a bicharacter (viewed as a morphism from $\G$ to $\DuG[H]$). 

By~\cite{Meyer-Roy-Woronowicz:Twisted_tensor}*{Lemma 3.8}, there exists a~\(\bichar\)\nb-\emph{Heisenberg pair}, i.e.\ a Hilbert space \(\Hils\) and a pair of representations~\(\alpha\in\Mor(\Contvin(\G),\Comp(\Hils))\) 
and~\(\beta\in\Mor((\Contvin(\G[H]),\Comp(\Hils))\) that satisfy the following condition:
\begin{equation}
\label{eq:V-Heispair}
\multunit[\G]_{1\alpha}\multunit[\GH]_{2\beta}=\multunit[\GH]_{2\beta}\multunit[\G]_{1\alpha}\bichar_{12}
\qquad\text{in~\(\U(\Contvin(\DuG)\otimes\Contvin(\DuG[H])\otimes\Comp(\Hils))\).}
\end{equation}
Here~\(\multunit[\G]_{1\alpha}\defeq \bigl((\Id_{\DuG}\otimes\alpha)(\multunit[\G])\bigr)_{13}\) and 
\(\multunit[\GH]_{2\beta}\defeq \bigl((\Id_{\DuG[H]}\otimes\beta)(\multunit[\GH])\bigr)_{23}\). We recall \cite{Meyer-Roy-Woronowicz:Twisted_tensor}*{Example 3.2} to motivate the notion of $\bichar$\nb-Heisenberg pairs. Indeed, let $G=H=\R$ and consider a standard bicharacter $\bichar\in\U(\Contvin(\R)\otimes\Contvin(\R))$, i.e.\ the map  $(s,t)\to \exp(ist)$. Then the maps $\alpha$ and $\beta$ satisfying the condition~\eqref{eq:V-Heispair} can be equivalently described as a pair of continuous one-parameter unitary groups $(U_{1}(s),U_{2}(t))_{s,t\in\R}$ satisfying the canonical commutation relation in the Weyl form: $U_{1}(s)U_{2}(t)=\exp(ist)U_{2}(t)U_{1}(s)$ for all $s,t\in\R$.

Define~\(j_{A}\in\Mor(A,A\otimes B\otimes\Comp(\Hils))\) and~\(j_{B}\in\Mor(B,A\otimes B\otimes\Comp(\Hils))\) by
\begin{equation}
\label{eq:def-js}
j_{A}(a)\defeq \bigl((\Id_{A}\otimes\alpha)(\gamma_{A}(a))\bigr)_{13},
\qquad
j_{B}(b)\defeq \bigl((\Id_{B}\otimes\alpha)(\gamma_{B}(b))\bigr)_{23}
\end{equation}
for all~ \(a\in A\),  \(b\in  B\).

Then we have the following theorem.
\begin{theorem}[\cite{Meyer-Roy-Woronowicz:Twisted_tensor}*{Theorem 4.6}]
 The space	\(A\boxtimes_{\bichar}B\defeq j_{A}(A)j_{B}(B)\) is a~\(\Cst\)\nb-algebra, which does not depend (up to an isomorphism) on the choice of the~\(\bichar\)\nb-Heisenberg pair \((\alpha,\beta)\).
\end{theorem}

Let~$\bichar=1\in\U(\Contvin(\hat{G})\otimes\Contvin(\hat{H}))$ be the trivial bicharacter. The associated twisted tensor product of $A$ and $B$ is isomorphic to the minimal tensor product $A\otimes B$. Further,
if $ \Gamma  $ is a discrete group acting on a unital $\Cst$\nb-algebra $ A $ via an action $ \beta$, then  the reduced crossed product $ A \rtimes_{\beta, \red } \Gamma $ is another particular example of the above construction. We formalize this observation in the theorem below. 

\begin{proposition}[\cite{Meyer-Roy-Woronowicz:Twisted_tensor}*{Theorem 6.3}] 
	\label{18thjan1}
	Let $\Gamma$ be a discrete group, and assume that $ \gamma_A \colon= \beta\in \Mor(A, A \otimes \Contvin( \Gamma )) $ is a coaction of $\Contvin( \Gamma ) $ on a $\Cst$\nb-algebra $ A, $ that  $ ( B, \gamma_B ) = (\Cred ( \Gamma ), \hat{\Delta}_{\Gamma} ) $\textup{(}i.e.\ $\gamma_B$ is a canonical action of $\hat{\Gamma}$ on itself\textup{)} and that \(\bichar=\multunit[\Gamma]\in\U(\Cred ( \Gamma ) \otimes\Contvin( \Gamma ) ) \) is the reduced bicharacter for~\(\Gamma\). Then there exists an isomorphism $ \Psi\colon A \boxtimes_\bichar B \rightarrow  A \rtimes_{\beta,\red} \Gamma $ such that  for all $ a \in A, b \in B$, 
	\begin{equation}
	\label{eq:jan26}
	\Psi ( j_A ( a ) ) = \beta ( a ), 
	\qquad
	\Psi ( j_B ( b ) ) = 1 \otimes b.
	\end{equation}
	Here we have used the canonical faithful representations of~\(\Contvin(\Gamma)\) and~\(\Cred(\Gamma)\) in~\(\Bound(\ell^{2}(\Gamma))\) 
	in the second legs.
\end{proposition}

The \emph{generalized Drinfeld double}~\(\GDrin{\bichar}\) of~\(\G\) and~\(\G[H]\) with respect to the bicharacter~\(\bichar\) is a locally compact quantum group described  (see~\cite{Roy:Codoubles}*{Theorem 5.1(ii)}) as follows:
\begin{alignat}{2}
\label{eq:GDrin-alg}
\Contvin(\GDrin{\bichar}) &=\rho(\Contvin(\G))\theta(\Contvin(\G[H])),\\
\label{eq:GDrin-comult}
\Comult[\GDrin{\bichar}](\rho(x)\theta(y))
&=(\rho\otimes\rho)(\Comult[\G](x)) (\theta\otimes\theta)(\Comult[\GH](y))
\end{alignat}
for all~\(x\in\Contvin(\G)\) and~\(y\in\Contvin(\G[H])\). Here~\(\rho\) and \(\theta\) form a pair of faithful representations of~\(\Contvin(\G)\) and~\(\Contvin(\G[H])\) on a Hilbert space~\(\Hils_{\mathcal{D}}\) that 
satisfies \(\bichar\)\nb-\emph{Drinfeld commutation relation}:
\begin{equation}
\label{eq:Drinf-comm}
\bichar_{12}\multunit[\G]_{1\rho}\multunit[\GH]_{2\theta} 
=\multunit[\GH]_{2\theta} \multunit[\G]_{1\rho}\bichar_{12} 
\qquad\text{in~\(\U(\Contvin(\DuG)\otimes\Contvin(\DuG[H])\otimes\Comp(\Hils_{\mathcal{D}}))\),}
\end{equation}
where~\(\multunit[\G]_{1\rho}\defeq((\Id_{\DuG}\otimes\rho)(\multunit[\G]))_{13}\) and~\(\multunit[\GH]_{2\theta}\defeq((\Id_{\GH}\otimes\theta)(\multunit[\GH]))_{23}\).

In fact, \(\rho\in\Mor\bigl(\Contvin(\G),\Contvin(\GDrin{\bichar})\bigr)\) and \(\theta\in\Mor\bigl(\Contvin(\G[H]),\Contvin(\GDrin{\bichar})\bigr)\) are Hopf \Star{}homomorphisms.

Now, \cite{Roy:Codoubles}*{Lemma 6.3} shows that there is a canonical coaction~\(\gamma_{A}\bowtie_{\bichar}\gamma_{B}\in\Mor(A\boxtimes_{\bichar} B, A\boxtimes_{\bichar}B\otimes\Contvin(\GDrin{\bichar}))\) of~\(\Contvin(\GDrin{\bichar})\) on 
\(A\boxtimes_{\bichar}B\) defined by 
\begin{equation}
\label{eq:GDrin_act}
\gamma_{A}\bowtie_{\bichar}\gamma_{B}(j_{A}(a))= (j_{A}\otimes\rho)(\gamma_{A}(a)),
\qquad 
\gamma_{A}\bowtie_{\bichar}\gamma_{B}(j_{B}(b)) = (j_{B}\otimes\theta)(\gamma_{B}(b))
\end{equation}
for all~\(a\in A\), \(b\in B\).

\begin{lemma}
  \label{lemm:injec-Drinact}
  Suppose~\(\gamma_{A}\) and~\(\gamma_{B}\) are injective maps. Then the map~\(\gamma_{A}\bowtie_{\bichar}\gamma_{B}\) 
  defined by~\textup{\eqref{eq:GDrin_act}} is also injective.
 \end{lemma}
 \begin{proof}
  Without loss of generality we may assume that~\((\corep{U}^{A},\varphi_{A})\) and~\((\corep{U}^{B},\varphi_{B})\) 
  are faithful covariant representations of \((A,\gamma_{A})\) and \((B,\gamma_{B})\) on the Hilbert spaces~\(\Hils_{A}\) 
  and \(\Hils_{B}\), respectively. 
  
  There is a faithful representation~\(\Pi\colon A\boxtimes_{\bichar}B\to\Bound(\Hils_{A}\otimes\Hils_{B})\) such that 
  \(\Pi(j_{A}(a))=\varphi_{A}(a)\otimes 1_{\Hils_{B}}\) and~\(\Pi(j_{B}(b))=Z(1_{\Hils_{A}}\otimes\varphi_{B}(b))Z^{*}\), 
  where~\(Z \in \U(\Hils_A \otimes \Hils_B)\) is the unique unitary that satisfies 
  \begin{equation}
   \label{eq:def-Z}
    \corep{U}^{A}_{1\alpha}\corep{U}^{B}_{2\beta}Z_{12} 
    =\corep{U}^{B}_{2\beta}\corep{U}^{A}_{1\alpha} 
    \qquad\text{in~\(\U(\Hils_{A}\otimes\Hils_{B}\otimes\Hils)\)}
  \end{equation}
  for any~\(\bichar\)\nb-Heisenberg pair~\((\alpha,\beta)\) on~\(\Hils\) (see~\cite{Meyer-Roy-Woronowicz:Twisted_tensor}*{Theorem 4.1 and 4.2}).
  Define a pair of representations~\(\tilde{\alpha}\defeq (\alpha\otimes\rho)\circ \Comult[\G]\) and~\(\tilde{\beta}(y)\defeq 
  (\beta\otimes\theta)\circ\Comult[\GH]\) of~\(\Contvin(\G)\) and~\(\Contvin(\G[H])\) on~\(\Hils\otimes\Hils_{\mathcal{D}}\).
  By Definition~\ref{def:corepresentation}.(1) for~\(\corep{U}^{A}\) and~\(\corep{U}^{B}\) and~\eqref{eq:def-Z} 
  we have
  \[
    \corep{U}^{A}_{1\tilde{\alpha}}\corep{U}^{B}_{2\tilde{\beta}}Z_{12}
    =\corep{U}^{A}_{1\alpha}\corep{U}^{A}_{1\rho}\corep{U}^{B}_{2\beta}\corep{U}^{B}_{2\theta}Z_{12}
    =\corep{U}^{A}_{1\alpha}\corep{U}^{B}_{2\beta}\corep{U}^{A}_{1\rho}\corep{U}^{B}_{2\theta}Z_{12}   
    =\corep{U}^{B}_{2\beta}\corep{U}^{A}_{1\alpha}Z^{*}_{12}\corep{U}^{A}_{1\rho}\corep{U}^{B}_{2\theta}Z_{12}
  \]
  in~\(\U(\Hils_{A}\otimes\Hils_{B}\otimes\Hils\otimes\Hils_{\mathcal{D}})\). 
  Here~\(\alpha\) and~\(\beta\) are acting on the third leg and~\(\rho\) and~\(\theta\) are acting on the fourth leg. Similarly, 
  we have
  \begin{equation}
   \label{eq:Z-Drinf}
      \corep{U}^{B}_{2\tilde{\beta}}\corep{U}^{A}_{1\tilde{\alpha}}
   =\corep{U}^{B}_{2\beta}\corep{U}^{B}_{2\theta}\corep{U}^{A}_{1\alpha}\corep{U}^{A}_{1\rho}
  =\corep{U}^{B}_{2\beta}\corep{U}^{A}_{1\alpha}\corep{U}^{B}_{2\theta}\corep{U}^{A}_{1\rho}
  \end{equation}
  Now~\((\tilde{\alpha},\tilde{\beta})\) is a \(\bichar\)\nb-Heisenberg pair on~\(\Hils\otimes\Hils_{\mathcal{D}}\) (see the proof 
  of~\cite{Roy:Codoubles}*{Lemma 6.3}). Therefore, we have the same commutation relation~\eqref{eq:def-Z} if we replace~\(\alpha\) 
  by~\(\tilde{\alpha}\) and~\(\beta\) by~\(\tilde{\beta}\). This gives the equality 
  \begin{equation}
   \corep{U}^{A}_{1\rho}\corep{U}^{B}_{2\theta}Z_{12}=Z_{12}\corep{U}^{B}_{2\theta}\corep{U}^{A}_{1\rho} 
   \qquad\text{in ~\(\U(\Hils_{A}\otimes\Hils_{B}\otimes\Hils_{\mathcal{D}})\).}
  \end{equation}
  Using~\eqref{eq:Z-Drinf} and~\eqref{eq:GDrin_act} we compute (for $a \in A, b \in B$)
  \begin{align*}
   & \corep{U}^{A}_{1\rho}\corep{U}^{B}_{2\theta}(\Pi(j_{A}(a)j_{B}(b))\otimes 1_{\GDrin{\bichar}})(\corep{U}^{B}_{2\theta})^{*} 
   (\corep{U}^{A}_{1\rho})^{*}\\
   &= \corep{U}^{A}_{1\rho}\corep{U}^{B}_{2\theta}
   \bigl((\varphi_{A}(a)\otimes 1_{\Hils_{B}}\otimes 1_{\GDrin{\bichar}})Z_{12}(1_{\Hils_{A}}\otimes\varphi_{B}(b)\otimes 1_{\GDrin{\bichar}}) \bigr)Z_{12}^{*}(\corep{U}^{B}_{2\theta})^{*} (\corep{U}^{A}_{1\rho})^{*}\\
   &= \corep{U}^{A}_{1\rho}(\varphi_{A}(a)\otimes 1_{\Hils_{B}}\otimes 1_{\GDrin{\bichar}})\corep{U}^{B}_{2\theta}Z_{12}
   (1_{\Hils_{A}}\otimes\varphi_{B}(b)\otimes 1_{\GDrin{\bichar}})(\corep{U}^{A}_{1\rho})^{*}(\corep{U}^{B}_{2\theta})^{*} Z_{12}^{*}\\
   &= \corep{U}^{A}_{1\rho}(\varphi_{A}(a)\otimes 1_{\Hils_{B}}\otimes 1_{\GDrin{\bichar}})\corep{U}^{B}_{2\theta}Z_{12}
   (\corep{U}^{A}_{1\rho})^{*}(1_{\Hils_{A}}\otimes\varphi_{B}(b)\otimes 1_{\GDrin{\bichar}})(\corep{U}^{B}_{2\theta})^{*} Z_{12}^{*}\\
   &=\corep{U}^{A}_{1\rho}(\varphi_{A}(a)\otimes 1_{\Hils_{B}}\otimes 1_{\GDrin{\bichar}})(\corep{U}^{A}_{1\rho})^{*}Z_{12}
   \corep{U}^{B}_{2\theta}(1_{\Hils_{A}}\otimes\varphi_{B}(b)\otimes 1_{\GDrin{\bichar}})(\corep{U}^{B}_{2\theta})^{*} Z_{12}^{*}\\
   &=\bigl((\varphi_{1}\otimes\rho)(\gamma_{A}(a))\bigr)_{13}Z_{12}\bigl((\varphi_{2}\otimes\theta)(\gamma_{B}(b))\bigr)_{23}Z_{12}^{*}\\
   &=(\Pi\otimes\Id_{\GDrin{\bichar}})\bigl((j_{A}\otimes\rho)(\gamma_{A}(a))(j_{B}\otimes\theta)(\gamma_{B}(b))\bigr)
     =(\Pi\otimes\Id_{\GDrin{\bichar}})\bigl(\gamma_{A}\bowtie\gamma_{B}(j_{A}(a)j_{B}(b))\bigr)
   \end{align*}
  Since \(\Pi\) is injective we conclude that~\(\gamma_{A}\bowtie\gamma_{B}\) is also injective.
 \end{proof}

Finally note the special (trivial) case of the above construction. The generalized Drinfeld double of  locally compact groups $G$ and $H$ associated to the trivial bicharacter $\bichar =1\in\U(\Contvin(\hat{G})\otimes\Contvin(\hat{H})) $ is just the cartesian product of the initial groups: $\Contvin(\GDrin{\bichar})=\Contvin(G\times H)$. Moreover, $\Contvin(G\times H)$ canonically  coacts (component-wise) on $A\otimes B$ which is a particular case of~\eqref{eq:GDrin_act}. Furthermore for any locally compact quantum group $\G$ the generalized Drinfeld double of $\G$ and $\DuG$ with respect to the reduced bicharacter $\multunit[G]\in\U(\Contvin(\DuG)\otimes\Contvin(\G))$ coincides with the usual Drinfeld double of $\G$. 

\section{Universal $\Cst$\nb-algebras of Drinfeld doubles and their universal property}
 \label{sec:Univ-Drinf}
 Let~\(\G\) and~\(\G[H]\) be locally compact quantum groups and let~\(\bichar\in\U(\Contvin(\DuG)\otimes\Contvin(\DuG[H]))\) be   
 a bicharacter. Recall that the dual of the generalized Drinfeld double~\(\GDrin{\bichar}\) is the quantum 
 codouble 
  \(\GCodb{\bichar}\) defined by
  \begin{align*}
    \Contvin(\GCodb{\bichar}) &\defeq \Contvin(\DuG[H])\otimes\Contvin(\DuG), \\
    \Comult[\GCodb{\bichar}](\hat{y}\otimes\hat{x}) &\defeq 
    \bichar_{23} \flip_{23}\bigl(\Comult[\DuGH](\hat{y})\otimes\Comult[\DuG](\hat{x})\bigr)\bichar^{*}_{23} 
    \quad\text{  for all~\(\hat{x}\in\Contvin(\DuG)\), \(\hat{y}\in\Contvin(\DuG[H])\).}
  \end{align*}

  The goal of this section is to prove the existence of a coaction of~\(\Contvin^\univ(\GDrin{\bichar})\) on $ A \boxtimes_{\bichar} B $ (Lemma \ref{lemm:Coact-univ-Drinf}) satisfying a universal property formalized in Theorem \ref{the:UnivProp-Drinf}. Let us start by proving the existence of a  
  \(\bichar\)\nb-Drinfeld pair at the universal level.
  \begin{proposition}
  There exists a unique pair  of morphisms~\(\rho^\univ\in\Mor(\Contvin^\univ(\G),\Contvin^\univ(\GDrin{\bichar}))\) 
  and \(\theta^\univ\in\Mor(\Contvin^\univ(\G[H]),\Contvin^\univ(\GDrin{\bichar}))\) satisfying the following commutation relation:
  \begin{equation}
   \label{eq:Univ-Drinfpair}
   \bichar_{12}\wW^{\G}_{1\rho^\univ}\wW^{\G[H]}_{2\theta^\univ} 
   =\wW^{\G[H]}_{2\theta^\univ} \wW^{\G}_{1\rho^\univ}\bichar_{12}
   \qquad\text{in~\(\U(\Contvin(\DuG)\otimes\Contvin(\DuG[H])\otimes\Contvin^\univ(\GDrin{\bichar}))\).}
  \end{equation} 
  \end{proposition}
  \begin{proof}
 Recall the universal representation \(\wW^{\GDrin{\bichar}}\in\U(\Contvin(\DuG[H])\otimes\Contvin(\DuG)\otimes\Contvin^\univ(\GDrin{\bichar}))\) of~\(\GCodb{\bichar}\) in~\(\Contvin^\univ(\GDrin{\bichar})\). 
  By~\cite{Roy:Codoubles}*{Proposition 7.9}, there exist representations~\(\corep{U}^{1}\in\U(\Contvin(\DuG)\otimes\Contvin^\univ(\GDrin{\bichar}))\) and~\(\corep{U}^{2}\in\U(\Contvin(\DuG[H])\otimes\Contvin^\univ(\GDrin{\bichar}))\), respectively of~\(\DuG\) and 
  \(\DuG[H]\) in~\(\Contvin^\univ(\GDrin{\bichar})\) such that 
  \begin{equation}
   \label{eq:univ-rep-codoub}
    \wW^{\GDrin{\bichar}}=\corep{U}^{1}_{23}\corep{U}^{2}_{13} 
    \quad\text{and}\quad 
    \bichar_{12}\corep{U}^{1}_{13}\corep{U}^{2}_{23} 
    =\corep{U}^{2}_{23}\corep{U}^{1}_{13}\bichar_{12}.
  \end{equation}
  Universality of~\(\wW^{\G}\) and~\(\wW^{\G[H]}\) gives unique $\C^*$-algebra morphisms~\(\rho^\univ\in\Mor(\Contvin^\univ(\G),\Contvin^\univ(\GDrin{\bichar}))\) and \(\theta^\univ\in\Mor(\Contvin^\univ(\G[H]),\Contvin^\univ(\GDrin{\bichar}))\) 
  such that~\((\Id_{\DuG}\otimes\rho^\univ)(\wW^{\G})=\corep{U}^{1}\) and~\((\Id_{\DuG[H]}\otimes\theta^\univ)(\wW^{\G[H]})
  =\corep{U}^{2}\).
  \end{proof}
 Therefore, we have 
  \begin{align*}
  & \Contvin^\univ(\GDrin{\bichar}) =\{(\omega_{1}\otimes\omega_{2}\otimes\Id_{\GDrin{\bichar}})(\wW^{\GDrin{\bichar}})\mid \omega_{1}\in\Contvin(\DuG[H])', \omega_{2}\in\Contvin(\DuG)'\}^\CLS \\
  &=\Bigl\{\theta^\univ \bigl( (\omega_{2}\otimes\Id_{\G[H]})(\wW^{\G[H]})\bigr)  \rho^\univ \bigl( (\omega_{1}\otimes\Id_{\G})(\wW^{\G})\bigr)\mid \omega_{1}\in\Contvin(\DuG[H])', \omega_{2}\in\Contvin(\DuG)'\Bigr\}^\CLS  \\
  &=\rho^\univ(\Contvin^\univ(\G))\theta^\univ(\Contvin^\univ(\G[H])).
  \end{align*}
  Since \(\wW^{\G}, \wW^{\G[H]}, {\wW^{\GDrin{\bichar}}}\) are bicharacters (see Example~\ref{ex:bichar}), 
  we also obtain 
  \begin{align*}
     & (\Id_{\DuG[H]}\otimes\Id_{\DuG}\otimes  \Comult[\GDrin{\bichar}]^\univ)(\wW^{\GDrin{\bichar}})\\
     &=\wW^{\GDrin{\bichar}}_{123}\wW^{\GDrin{\bichar}}_{124}\\
     &= (\wW^{\G}_{1\rho^\univ})_{23}(\wW^{\G[H]}_{1\theta^\univ})_{13}(\wW^{\G}_{1\rho^\univ})_{24}(\wW^{\G[H]}_{1\theta^\univ})_{14}\\
     &=(\wW^{\G}_{1\rho^\univ})_{23}(\wW^{\G}_{1\rho^\univ})_{24}(\wW^{\G[H]}_{1\theta^\univ})_{13}(\wW^{\G[H]}_{1\theta^\univ})_{14}\\
     &=((\Id_{\DuG}\otimes(\rho^\univ\otimes\rho^\univ)\circ \Comult[\G]^\univ)(\wW^{\G}))_{234}((\Id_{\DuG[H]}\otimes(\theta^\univ\otimes\theta^\univ)\circ\Comult[\GH]^\univ)(\wW^{\G[H]}))_{134}
  \end{align*}
  Taking the slice on the first and second leg by~\(\omega_{1}\otimes\omega_{2}\) for all~\(\omega_{1}\in\Contvin(\DuG[H])'\), 
  \(\omega_{2}\in\Contvin(\DuG)'\) and using the first equality in~\eqref{eq:univ-rep-codoub} we obtain
  \[
  \Comult[\GDrin{\bichar}]^\univ(\rho^\univ(x)\theta^\univ(y)) 
  =(\rho^\univ\otimes\rho^\univ)(\Comult[\G]^\univ(x))(\theta^\univ\otimes\theta^\univ)(\Comult[\GH]^\univ(y))
  \]
 for all~\(x\in\Contvin^\univ(\G)\), \(y\in\Contvin^\univ(\G[H])\).
 \begin{definition}
  \label{def:univ-Drinfpair}
  The pair of morphisms~\((\rho^\univ,\theta^\univ)\) is called the~\emph{universal \(\bichar\)\nb-Drinfeld pair}.
 \end{definition} 
 \begin{lemma}
  \label{lem:hopf-mor}
  The maps~\(\rho^\univ\) and~\(\theta^\univ\) are Hopf~\Star{}homomorphisms. 
  Furthermore~\(\Lambda_{\GDrin{\bichar}}\circ\rho^\univ=\rho\circ\Lambda_{\G}\) and \(\Lambda_{\GDrin{\bichar}}\circ\theta^\univ=\theta\circ\Lambda_{\G[H]}\).
 \end{lemma}
 \begin{proof}
   The reduced bicharacter of~\(\GDrin{\bichar}\) is~\(\multunit[\GDrin{\bichar}]=\multunit[\GH]_{2\theta}\multunit[\G]_{1\rho}\in\U(\Contvin(\DuG[H])\otimes\Contvin(\DuG)\otimes\Contvin(\GDrin{\bichar}))\) (see~\cite{Roy:Codoubles}*{Theorem 5.1(iii)}). 
   Using the first equality in~\eqref{eq:univ-rep-codoub} we have
   \begin{align*}
    \multunit[\GDrin{\bichar}]
    =(\Id_{\DuG}\otimes\Lambda_{\GDrin{\bichar}})(\wW^{\GDrin{\bichar}})
    &= ((\Id_{\DuG}\otimes\theta)(\multunit[\GH]))_{23}(\Id_{\DuG}\otimes\rho)(\multunit[\G]))_{13}\\
    &=(\Id_{\DuG[H]}\otimes\Id_{\DuG}\otimes\Lambda_{\GDrin{\bichar}})
    (\wW^{\GDrin{\bichar}})\\
    &=((\Id_{\DuG[H]}\otimes\Lambda_{\GDrin{\bichar}}\circ\theta^\univ)(\wW^{\GH}))_{23}
         ((\Id_{\DuG}\otimes\Lambda_{\GDrin{\bichar}}\circ\rho^\univ  )(\wW^{\G}))_{13}
   \end{align*}
   in~\(\U(\Contvin(\DuG[H])\otimes\Contvin(\DuG)\otimes\Contvin(\GDrin{\bichar}))\). The last 
   equation implies
   \begin{align*}
   & ((\Id_{\DuG[H]}\otimes\Lambda_{\GDrin{\bichar}}\circ\theta^\univ)(\wW^{\GH}))^{*}_{23}
   ((\Id_{\DuG}\otimes\theta)(\multunit[\GH]))_{23}\\
   &= ((\Id_{\DuG}\otimes\Lambda_{\GDrin{\bichar}}\circ\rho^\univ  )(\wW^{\G}))_{13}
      (\Id_{\DuG}\otimes\rho)(\multunit[\G]))^{*}_{13}
   \end{align*}
  Clearly, the first leg in the term of the left hand side and the second leg in the term of the right hand side of the above equation 
  are trivial. Hence, both the terms are equal to~\(1_{\DuG[H]}\otimes 1_{\DuG}\otimes u\) for some element~\(u\in\U(\Contvin(\GDrin{\bichar}))\). In particular, we have 
  \begin{equation}
   \label{eq:bichar-equal}
     (\Id_{\DuG}\otimes\Lambda_{\GDrin{\bichar}}\circ\rho^\univ  )(\wW^{\G})=(1_{\DuG}\otimes u)
      (\Id_{\DuG}\otimes\rho)(\multunit[\G])
      \qquad\text{in~\(\U(\Contvin(\DuG)\otimes\Contvin(\GDrin{\bichar}))\).}
  \end{equation}
  Since~\(\wW^{\G}\) is the (universal) left representation of~\(\DuG\) in~\(\Contvin^\univ(\G)\), the unitary 
  \(((\Id_{\DuG}\otimes\Lambda_{\GDrin{\bichar}}\circ\rho^\univ  )(\wW^{\G})\) is also a representation of~\(\DuG\)
   in~\(\Contvin(\GDrin{\bichar})\). Also, \((\Id_{\DuG}\otimes\rho)(\multunit[\G])\) is a left representation 
  of~\(\DuG\) in~\(\Contvin(\GDrin{\bichar})\). These facts force \(u=1_{\GDrin{\bichar}}\). Now~\((\Id_{\DuG}\otimes\rho)(\multunit[\G])\) is the bicharacter induced by the Hopf~\Star{}homomorphism~\(\rho\). Then~\eqref{eq:bichar-equal} and~\eqref{eq:bicha_univmor} 
  show that~\(\rho^\univ\) is a Hopf~\Star{}homomorphism which is the unique universal 
  lift of~\(\rho\). Slicing the first leg of the both sides of~\eqref{eq:bichar-equal} with~\(\omega\in\Contvin(\DuG)'\) we get 
  \(\Lambda_{\GDrin{\bichar}}\circ\rho^\univ=\rho\circ\Lambda_{\G}\). The rest of the proof follows in a similar way.
 \end{proof}
 
 \subsection{Coaction on the twisted tensor products}
  \label{subsec:coaction-on-twisted}
   
  As a prelude to the results proved in the next section, we  show that  there is a unique coaction~\(\gamma^\univ_{A}\bowtie_{\bichar}\gamma^\univ_{B}\in\Mor
  (A\boxtimes_{\bichar}B,A\boxtimes_{\bichar}B\otimes\Contvin^\univ(\GDrin{\bichar}))\) satisfying some natural conditions.
  
 \begin{lemma}
  \label{lemm:Coact-univ-Drinf}
  Let~\((A,\gamma_{A})\) and \((B,\gamma_{B})\) be \(\G\)\nb- and \(\G[H]\)\nb-\(\Cst\)\nb-algebras such that 
  \(\gamma_{A}\) and~\(\gamma_{B}\) are injective. Let~\(\gamma^\univ_{A}\) and~\(\gamma^\univ_{B}\) be the respective 
  universal normal coactions. Then there is a unique coaction~\(\gamma^\univ_{A}\bowtie_{\bichar}\gamma^\univ_{B}\in\Mor
  (A\boxtimes_{\bichar}B,A\boxtimes_{\bichar}B\otimes\Contvin^\univ(\GDrin{\bichar}))\) such that 
  \begin{equation}
   \label{eq:uniq-ext}
   \gamma^\univ_{A}\bowtie_{\bichar}\gamma^\univ_{B}(j_{A}(a))=(j_{A}\otimes\rho^\univ)(\gamma_{A}^\univ(a)), 
   \qquad
   \gamma^\univ_{B}\bowtie_{\bichar}\gamma^\univ_{B}(j_{B}(b))=(j_{B}\otimes\theta^\univ)(\gamma_{B}^\univ(b))
  \end{equation}
  for all~\(a\in A\), \(b\in B\).
 \end{lemma}
  \begin{proof}
 Since~\(\gamma_{A}\) and~\(\gamma_{B}\) are injective,  the canonical coaction 
 \(\gamma_{A}\bowtie_{\bichar}\gamma_{B}\) defined by~\eqref{eq:GDrin_act} is injective by Lemma~\ref{lemm:injec-Drinact}. Hence, there is a unique 
 normal coaction of~\(\Contvin^\univ(\GDrin{\bichar})\), denoted 
 by~\(\gamma^\univ_{A}\bowtie_{\bichar}\gamma^\univ_{B}\), such that 
 $$(\Id_{A\boxtimes_{\bichar}B}\otimes\Lambda_{\GDrin{\bichar}}) \circ \gamma^\univ_{A}\bowtie_{\bichar}\gamma^\univ_{B}
 =\gamma_{A}\bowtie_{\bichar}\gamma_{B}.$$
 Therefore, by~\eqref{eq:GDrin_act}, for $a \in A$ we have 
 \begin{align*}
  (\Id_{A\boxtimes_{\bichar}B}\otimes\Lambda_{\GDrin{\bichar}})(\gamma^\univ_{A}\bowtie_{\bichar}\gamma^\univ_{B}(j_{A}(a)))
 &=\gamma_{A}\bowtie_{\bichar}\gamma_{B}(j_{A}(a))
 =(j_{A}\otimes\rho)(\gamma_{A}(a)) 
 \end{align*}
  where the last equality uses  Lemma~\ref{lem:hopf-mor} and the definition of induced coactions.
  
  Let~\(\Delta_{\GDrin{\bichar}}^{\red,\univ}\in\Mor(\Contvin(\GDrin{\bichar}),\Contvin(\GDrin{\bichar})\otimes\Contvin^\univ(\GDrin{\bichar}))\) denote the universal lift of~\(\Comult[\GDrin{\bichar}]\) while 
  viewed as coaction of~\(\Contvin({\GDrin{\bichar}})\) on~\(\Contvin(\GDrin{\bichar})\). By virtue of 
  \cite{Roy-Timmermann:Max_twisted_tensor}*{Proposition 4.8}, \(\Delta_{\GDrin{\bichar}}^{\red,\univ}\) 
  satisfies the following condition
  \begin{equation}
   \label{eq:Univ_lift_comult}
   (\Lambda_{\GDrin{\bichar}}\otimes\Id_{\GDrin{\bichar}})\circ\Comult[\GDrin{\bichar}]^{\univ}
   =\Delta_{\GDrin{\bichar}}^{\red,\univ}\circ\Lambda_{\GDrin{\bichar}}.
  \end{equation}
  By~\eqref{eq:GDrin-comult} and repeated application of \eqref{eq:Univ_lift_comult}, \eqref{eq:GDrin_act} and 
  Lemma~\ref{lem:hopf-mor} gives
  \begin{align*}   
  & (\gamma_{A}\bowtie_{\bichar}\gamma_{B}\otimes\Id_{\GDrin{\bichar}^\univ})
     (\gamma^\univ_{A}\bowtie_{\bichar}\gamma^\univ_{B}(j_{A}(a)))\\
  &=(\Id_{A\boxtimes B}\otimes\Lambda_{\GDrin{\bichar}}\otimes\Id_{\GDrin{\bichar}})(
      (\gamma^\univ_{A}\bowtie_{\bichar}\gamma^\univ_{B}\otimes\Id_{\GDrin{\bichar}})
     (\gamma^\univ_{A}\bowtie_{\bichar}\gamma^\univ_{B}(j_{A}(a))))\\
  &= \bigl(\Id_{A\boxtimes B}\otimes(\Lambda_{\GDrin{\bichar}^\univ}\otimes\Id_{\GDrin{\bichar}})\circ\Comult[\GDrin{\bichar}]^\univ\bigr)
      (\gamma^\univ_{A}\bowtie_{\bichar}\gamma^\univ_{B}(j_{A}(a)))\\
  &= (\Id_{A\boxtimes B}\otimes\Delta_{\GDrin{\bichar}}^{\red,\univ})
      (\gamma_{A}\bowtie_{\bichar}\gamma_{B}(j_{A}(a)))\\
  &= (\Id_{A\boxtimes B}\otimes\Delta_{\GDrin{\bichar}}^{\red,\univ})
     ((j_{A}\otimes\rho)(\gamma_{A}(a))\\
  &=(\Id_{A\boxtimes B}\otimes\Delta_{\GDrin{\bichar}}^{\red,\univ}\circ\Lambda_{\GDrin{\bichar}})
      ((j_{A}\otimes\rho^\univ)\gamma^\univ_{A}(a))\\
  &=\bigl(\Id_{A\boxtimes B}\otimes(\Lambda_{\GDrin{\bichar}}\otimes\Id_{\GDrin{\bichar}})\circ
       \Comult[\GDrin{\bichar}]^\univ\circ\rho^\univ\bigr)
       ((j_{A}\otimes\Id_{\Contvin(\G^\univ)})\gamma^\univ_{A}(a))\\
  &= \bigl(j_{A}\otimes\Lambda_{\GDrin{\bichar}}\circ\rho^\univ\otimes\rho^\univ\bigr)
       ((\Id_{A}\otimes\Comult[\G^\univ])(\gamma^\univ_{A}(a)))\\
  &= \bigl((j_{A}\otimes\rho\circ\Lambda_{\G})\circ\gamma^\univ_{A}\otimes\rho^\univ\bigr)(\gamma^\univ_{A}(a))\\
  &=((j_{A}\otimes\rho)\gamma_{A}\otimes\rho^\univ)(\gamma^\univ_{A}(a))
    =(\gamma_{A}\bowtie_{\bichar}\gamma_{B}\otimes\Id_{\GDrin{\bichar}^\univ})
       \bigl((j_{A}\otimes\rho^\univ)(\gamma^\univ_{A}(a)\bigr)).
  \end{align*}
  Now, injectivity of~\(\gamma_{A}\bowtie_{\bichar}\gamma_{B}\) gives 
  \(\gamma^\univ_{A}\bowtie_{\bichar}\gamma^\univ_{B}(j_{A}(a))=(j_{A}\otimes\rho^\univ)\gamma^\univ_{A}(a)\) 
  for all~\(a\in A\). The second part of~\eqref{eq:uniq-ext} can be shown similarly.
 \end{proof}
 
\subsection{Universal property of~\(\Contvin^\univ(\GDrin{\bichar})\)}
 \label{subsec:universal property}
 Let $(\gamma^\univ_{A},A)$ be a coaction of \(\Contvin^\univ(\G)\) on a \(\Cst\)\nb-algebra $A$; in other words $\gamma^\univ_A$ is an action of $\G$ on the universal level. 

 In the spirit of~\cite{Goswami-Roy:Faithful_act_LCQG}*{Definition 4.2}, we say 
that the action  \(\gamma^\univ_{A}\) is \emph{faithful} if the \Star{}algebra generated by the set 
 \(\{(\omega_{1}\otimes\Id_{\G^\univ})(\gamma^\univ_{A}(a))\mid \omega_{1}\in A', a \in A\}\) is 
 strictly dense in~\(\Mult(\Contvin^\univ(\G))\).  If~\(\G\) is compact then we get the usual 
 definition of faithfulness of the action~\(\gamma_{A}^\univ\): \Star{}algebra generated by the 
 set~\(\{(\omega_{1}\otimes\Id_{\G})(\gamma^\univ_{A}(a))\mid \omega_{1}\in A', a \in A\}\) 
 is norm dense in~\(\Cont^\univ(\G)\).

\begin{theorem}
 \label{the:UnivProp-Drinf}
 Let $\G$, $\G[H]$ be locally compact quantum groups and let $A, B$ be $\Cst$\nb-algebras.
 Let~\((A,\gamma^\univ_{A})\), 
 \((B,\gamma^\univ_{B})\), be an action of \(\G\) on $A$ on the universal level, respectively an action of \(\G[H]\) on $B$ on the universal level. Assume that~\(\gamma_{A}^\univ\) and~\(\gamma_{B}^\univ\) are faithful and normal. 

 Suppose that \(\G[I]\) is a locally compact quantum group and
 \begin{enumerate}
  \item there is a coaction~\(\gamma\in\Mor(A\boxtimes_{\bichar}B, A\boxtimes_{\bichar}B\otimes\Contvin^\univ(\G[I]))\) 
  of~\(\Contvin^\univ(\G[I])\) on~\(A\boxtimes_{\bichar}B\);
  \item there are Hopf~\Star{}homomorphisms~\(\rho_{1}\in\Mor(\Contvin^\univ(\G),\Contvin^\univ(\G[I]))\) and 
  \(\theta_{1}\in\Mor(\Contvin^\univ(\G[H]),\Contvin^\univ(\G[I]))\) such that 
  \[
   \gamma \circ j_{A}=(j_{A}\otimes\rho_{1})\circ \gamma_{A}^\univ, 
   \qquad 
   \gamma \circ j_{B}=(j_{B}\otimes\theta_{1})\circ\gamma_{B}^\univ.
  \]
 \end{enumerate} 
Then there is a unique Hopf~\Star{}homomorphism~\(\Psi\in\Mor(\Contvin^\univ(\GDrin{\bichar}),\Contvin^\univ(\G[I]))\) such that~\(\Psi\circ\rho^\univ=\rho_{1}\) and \(\Psi\circ\theta^\univ=\theta_{1}\). In particular, this implies that \((\Id_{A\boxtimes_{\bichar}B}\otimes\Psi)\circ\gamma^\univ_{A}\bowtie\gamma^\univ_{B}\circ j_{A}=\gamma\circ j_{A}\) and \((\Id_{A\boxtimes_{\bichar}B}\otimes\Psi)\circ\gamma^\univ_{A}\bowtie\gamma^\univ_{B}\circ j_{B}=\gamma\circ j_{B}\).
\end{theorem}
\begin{proof}
 Denote the induced coactions of~\(\Contvin(\G)\) and~\(\Contvin(\G[H])\) on~\(A\) and~\(B\) by 
 \(\gamma_{A}\) and~\(\gamma_{B}\), respectively. Since \(\gamma^\univ_{A}\) and \(\gamma^\univ_{B}\) 
 are normal, \(\gamma_{A}\) and~\(\gamma_{B}\) are injective.
 
 Let~\((\alpha,\beta)\) be a~\(\bichar\)\nb-Heisenberg pair on a Hilbert space~\(\Hils\). Then using the definition of 
 \(j_{A}\) we compute ($a \in A$)
 \begin{align*}
  \gamma(j_{A}(a)) 
    =(j_{A}\otimes\rho_{1})(\gamma_{A}^\univ(a))
  &=\bigl(((\Id_{A}\otimes\alpha)\gamma_{A}\otimes\rho_{1})(\gamma_{A}^\univ(a))\bigr)_{134}\\
  &=\bigl(((\Id_{A}\otimes\alpha\circ\Lambda_{\G})\gamma^\univ_{A}\otimes\rho_{1})(\gamma_{A}^\univ(a))\bigr)_{134}\\
  &=\bigl(((\Id_{A}\otimes\alpha\circ\Lambda_{\G}\otimes\rho_{1})(\Id_{A}\otimes\Comult[\G]^{\univ})(\gamma_{A}^\univ(a))\bigr)_{134}
 \end{align*} 
 Similarly, we obtain 
 $$\gamma(j_{B}(b)) =\bigl(((\Id_{B}\otimes\beta\circ\Lambda_{\G[H]}\otimes\theta_{1})((\Id_{B}\otimes\Comult[\GH]^{\univ})(\gamma_{B}^\univ(b)))\bigr)_{234}.$$
 Let~\((\bar{\alpha},\bar{\beta})\) be a~\(\bichar\)\nb-anti\nb-Heisenberg pair on a Hilbert space~\(\Hils_{1}\) (see~\cite{Meyer-Roy-Woronowicz:Twisted_tensor}*{Lemma 3.6 \& Lemma 3.8} for its existence):
 \begin{equation}
 \label{eq:V-antiHeispair}
 \multunit[\GH]_{2\bar{\beta}}\multunit[\G]_{1\bar{\alpha}}=\bichar_{12}\multunit[\G]_{1\bar{\alpha}}\multunit[\GH]_{2\bar{\beta}}
 \qquad\text{in~\(\U(\Contvin(\DuG)\otimes\Contvin(\DuG[H])\otimes\Comp(\Hils_{1}))\).}
\end{equation}
By~\cite{Roy-Timmermann:Max_twisted_tensor}*{Proposition 5.10}, \(j_{A}\) and~\(j_{B}\) satisfy the following commutation 
relation:
\[
  [(j_{A}\otimes\bar{\alpha})(\gamma_{A}(a)),(j_{B}\otimes\bar{\beta})(\gamma_{B}(b))]=0 
  \qquad\text{for all~\(a\in A\), \(b\in B\).}
\]
This implies that 
\[
  [(\gamma\circ j_{A}\otimes\bar{\alpha})(\gamma_{A}(a)),(\gamma\circ j_{B}\otimes\bar{\beta})(\gamma_{B}(b))]=0 
  \qquad\text{for all~\(a\in A\), \(b\in B\).}
\]
This is in turn equivalent to
\begin{align*}
& \Bigl[\bigl((\Id_{A}\otimes\alpha\circ\Lambda_{\G}\otimes\rho_{1}\otimes\Lambda_{\G}\circ\bar{\alpha})(\Id_{A}\otimes
(\Id_{\G}\otimes\Comult[\G^\univ])\circ\Comult[\G]^{\univ})(\gamma_{A}^\univ(a))\bigr)_{1345}, \\
 &
 \bigl((\Id_{B}\otimes\beta\circ\Lambda_{\GH}\otimes\theta_{1}\otimes\Lambda_{\GH}\circ\bar{\beta})(\Id_{B}\otimes(\Id_{\GH}\otimes\Comult[\GH^\univ])\circ\Comult[\GH]^{\univ})(\gamma_{B}^\univ(b))\bigr)_{2345}\Bigr]=0.
\end{align*}
Therefore, the following operators 
\begin{alignat*}{2}
 \alpha'(x) &\defeq (\alpha\circ\Lambda_{\G}\otimes\rho_{1}\otimes\Lambda_{\G}\circ\bar{\alpha})((\Id_{\G}\otimes\Comult[\G]^{\univ})(\Comult[\G]^{\univ}(x))),\\
 \beta'(y) &\defeq (\beta\circ\Lambda_{\GH}\otimes\theta_{1}\otimes\Lambda_{\GH}\circ\bar{\beta})((\Id_{\GH}\otimes\Comult[\GH]^{\univ})(\Comult[\GH]^{\univ}(y)))
\end{alignat*}
commute for all~\(x=(\omega_{1}\otimes\Id_{\G})(\gamma^\univ_{A}(a))\in\Mult(\Contvin^{\univ}(\G))\),~\(y=(\omega_{2}\otimes\Id_{\G[H]})(\gamma^\univ_{B}(b))\in\Mult(\Contvin^{\univ}(\G[H]))\), where~\(\omega_{1}\in A'\), \(\omega_{2}\in B'\), \(a\in A\), \(b\in B\). Now faithfulness of 
\(\gamma^\univ_{A}\) and~\(\gamma^\univ_{B}\) shows that the above operators  commute for all~\(x\in\Mult(\Contvin^\univ(\G))\) and 
\(y\in\Mult(\Contvin^\univ(\G[H]))\); hence in particular for all \(x\in\Contvin^\univ(\G)\) and \(y\in\Contvin^\univ(\G[H])\).

Equivalently, 
\begin{equation}
 \label{eq:Drinpair-aux}
\wW^{\G}_{1\alpha'}\wW^{\GH}_{2\beta'}=\wW^{\GH}_{2\beta'}\wW^{\G}_{1\alpha'}
\quad\text{in
\(\U(\Contvin(\DuG)\otimes\Contvin(\DuG[H])\otimes\Comp(\Hils)\otimes\Contvin^\univ(\G[I])\otimes\Comp(\Hils_{1}))\).}
\end{equation}
Using the properties of~\(\wW^{\G}\), \(\wW^{\G[H]}\), \(\Lambda_{\G}\), \(\Lambda_{\GH}\) and the equations~\eqref{eq:V-Heispair}, \eqref{eq:V-antiHeispair} we obtain:
\begin{align*}
    \wW^{\G}_{1\alpha'}\wW^{\GH}_{2\beta'}
 &= \multunit[\G]_{1\alpha}\wW^{\G}_{1\rho_{1}}\multunit[\G]_{1\bar{\alpha}}\multunit[\GH]_{2\beta}\wW^{\G[H]}_{2\theta_{1}}\multunit[\GH]_{2\bar{\beta}}\\
 &=\multunit[\G]_{1\alpha}\multunit[\GH]_{2\beta}\wW^{\G}_{1\rho_{1}}\wW^{\G[H]}_{2\theta_{1}}\multunit[\G]_{1\bar{\alpha}}\multunit[\GH]_{2\bar{\beta}}\\
 &=\multunit[\GH]_{2\beta}\multunit[\G]_{1\alpha}\bichar_{12}\wW^{\G}_{1\rho_{1}}\wW^{\G[H]}_{2\theta_{1}}\bichar_{12}^{*}\multunit[\GH]_{2\bar{\beta}}\multunit[\G]_{1\bar{\alpha}}.
\end{align*}
In the computation above the morphisms~\(\alpha\), \(\beta\) are acting on the third leg, \(\rho_{1}\), \(\theta_{1}\) are acting on the fourth leg and 
\(\bar{\alpha}\), \(\bar{\beta}\) are acting on the fifth leg, respectively.

Similarly, we get
\begin{align*}
    \wW_{2\beta'}\wW^{\G}_{1\alpha'}
 &=\multunit[\GH]_{2\beta}\multunit[\G]_{1\alpha}\wW^{\G[H]}_{2\rho_{1}}\wW^{\G}_{1\rho_{1}}\multunit[\GH]_{2\bar{\beta}}\multunit[\G]_{1\bar{\alpha}}.
 \end{align*}
 Then~\eqref{eq:Drinpair-aux} gives 
 \[
 \bichar_{12}\wW^{\G}_{1\rho_{1}}\wW^{\G[H]}_{2\theta_{1}}=\wW^{\G[H]}_{2\rho_{1}}\wW^{\G}_{1\rho_{1}} \bichar_{12}
\qquad\text{in \(\U(\Contvin(\DuG)\otimes\Contvin(\DuG[H])\otimes\Contvin^\univ(\G[I]))\).} 
\]
By~\cite{Roy:Codoubles}*{Proposition 7.9},~\(\wW^{\G}_{1\rho_{1}}\wW^{\G[H]}_{2\theta_{1}}\) is a left representation of~\(\GCodb{\bichar}\) in 
\(\Contvin^\univ(\G[I])\). Hence, there is a unique~\(\eta\in\Mor(\Contvin^\univ(\GDrin{\bichar}),\Contvin^\univ(\G[I]))\) such that 
\((\Id_{\DuG}\otimes\Id_{\DuG[H]}\otimes\eta)(\wW^{\G}_{1\rho^\univ}\wW^{\G[H]}_{2\theta^\univ})
 =\wW^{\G}_{1\rho_{1}}\wW^{\G[H]}_{2\theta_{1}}\), which implies 
\(((\Id_{\DuG}\otimes\eta)(\wW^{\G}_{1\rho^\univ}))((\Id_{\DuG[H]}\otimes\eta)(\wW^{\G[H]}_{2\theta^\univ}))
=\wW^{\G}_{1\rho_{1}}\wW^{\G[H]}_{2\theta_{1}}\). Equivalently, 
\[
(\wW^{\G}_{1\rho_{1}})^{*}((\Id_{\DuG}\otimes\eta)(\wW^{\G}_{1\rho^\univ}))=\wW^{\G[H]}_{2\theta_{1}}((\Id_{\DuG[H]}\otimes\eta)(\wW^{\G[H]}_{2\theta^\univ}))^{*}
\]
in~\(\U(\Contvin(\DuG)\otimes\Contvin(\DuG[H])\otimes\Contvin^\univ(\G[I]))\). 
Hence there is a unique~\(u\in\U(\Contvin^\univ(\G[I]))\) such that~\(((\Id_{\DuG}\otimes\eta\circ\rho^\univ)(\wW^{\G})=(\Id_{\DuG}\otimes\rho_{1})(\wW^{\G})(1\otimes u)\) and 
\((\Id_{\G[H]}\otimes\theta_{1})(\wW^{\G[H]})=(1\otimes u)((\Id_{\DuG[H]}\otimes\eta\circ\theta^\univ)(\wW^{\G[H]})\). Now 
\((\Id_{\G[H]}\otimes\theta_{1})(\wW^{\G[H]})\in\U(\Contvin(\DuG)\otimes\Contvin^\univ(\GDrin{\bichar}))\) is a left representation 
of~\(\DuG\) in~\(\Contvin^\univ(\GDrin{\bichar})\). Repeating the argument used in the proof of Lemma~\ref{lem:hopf-mor}
 we can conclude that \(u=1\) and this completes the proof.
\end{proof}

\section{Orthogonal filtrations of $C^*$-algebras and their quantum symmetries} 
 \label{sec:orthfilt}

The theory of quantum symmetry group of orthogonal filtrations first appeared in \cite{MR3066746}. Later, M.\ Thibault de Chanvalon extended the theory to Hilbert modules in the paper \cite{MR3158722}. 
In this section we recall the notions of orthogonal filtration of a unital $\Cst$\nb-algebra and their quantum symmetry groups as developed in these two papers. Then, under suitable conditions, we prove the existence of a canonical filtration of  twisted tensor products.

\begin{definition}[\cite{MR3158722}, Definition 2.4]
Let ~\(A\) be a unital $\Cst$\nb-algebra and let $ \tau_{A}$ be a faithful state on $ A$. An orthogonal filtration for the pair $ (A, \tau_{A}) $ is a sequence of finite dimensional subspaces $ \{ A_i \}_{i \geq 0} $ such that $ A_0 = \bc1_A, ~  {\rm Span} \cup_{i \geq 0} A_i $ is dense in $ A $ and $ \tau_{A} ( a^* b ) = 0 $ if $ a \in A_i, ~ b \in A_j $ and $ i \neq j. $  We will usually write $\Afilt$ for the triple $(A, \tau_{A}, \{ A_i \}_{i \geq 0})$ as above.
\end{definition}

\begin{remark} \label{14thfeb1}
Thibault de Chanvalon's definition replaced subspaces $A_i$ by suitable Hilbert modules. However, if we view a unital $ C^* $-algebra $ A $ as a right Hilbert module over itself and we take $ W = \bc 1_A $ and $ J $ to be the map $ a \mapsto a^*, $ then  the above definition indeed coincides with the one given in Definition 2.4 of \cite{MR3158722}. We should mention that  in the original formulation of \cite{MR3066746} it was additionally assumed that $ {\rm Span} \cup_{i \geq 0} A_i $ is a  $\ast$-algebra; at the same time the indices $i$ were allowed to come from an arbitrary set. We will occasionally use the latter framework without further comment.
\end{remark}

The following example will be crucially used throughout the rest of the article.

\begin{example} 
 \label{groupalgexample}
Let $\Gamma$ be a finitely generated discrete group endowed with a proper length function $l.$ Then the collection $ B^l_n := {\rm span} \{ \lambda_g \mid l ( g ) = n \} $, $n \geq 0$, forms a filtration for the pair $ (\Cred(\Gamma), \tau_{\Gamma} ) $ where $\tau_{\Gamma}$ is the canonical trace on $\Cred(\Gamma)$.
\end{example}

\subsection{The quantum symmetry group of an orthogonal filtration}
\begin{theorem}[\cite{MR3066746}, \cite{MR3158722}]
\label{banicaskalskitheorem}
Let $ \{ A_i \}_{i \geq 0} $ be an orthogonal filtration for a pair $ ( A, \tau_{A} ) $ as above.    Let $ {\bf \clc} (\Afilt) $ be the category with objects as pairs $ (G, \alpha ) $ where $G$ is a compact quantum group, $ \alpha $ is an action of $\G$ on $ A $ such that $ \alpha ( A_i ) \subseteq A_i \otimes_{{\rm alg}}  \C (G) $ for each $i\geq 0$, and the morphisms being CQG morphisms intertwining the respective actions. Then there exists a universal initial object in the category $ {\bf \clc} (\Afilt) $ called the quantum symmetry group 
of the filtration $\Afilt$ and denoted by $ \QISO (\Afilt) $. Moreover the action of $\QISO (\Afilt)$ on $ A $ is faithful 
\textup{(}see Subsection~\textup{\ref{subsec:universal property}} for the definition of faithfulness\textup{)}.  
\end{theorem}

\begin{remark} 
As mentioned in Remark \ref{14thfeb1}, the definition of  an orthogonal filtration in  \cite{MR3066746} included an additional condition, namely that  $ {\rm Span} \{ A_i : i \geq 0 \}$ is  a $\ast$-algebra. However, M.\ Thibault de Chanvalon showed (\cite{MR3158722})  that the existence of the quantum symmetry group of an orthogonal filtration can be proved without assuming this extra condition.
\end{remark}

Note that we assume throughout that the actions in our category are defined on the reduced level; in fact the construction of the quantum symmetry group in \cite{MR3066746} gives naturally an action on the universal level (which then induces the reduced action). A certain care needs then to be taken when one interprets the intertwining relation with respect to the CQG morphisms (acting on the universal level), but this can be always dealt with, for example by exploiting the purely algebraic picture of the actions (see Lemma \ref{25thjan1}).

\begin{remark} \label{prop20thmay}
Given a pair $(\G,\alpha)\in \clc(\Afilt)$ we automatically deduce that the coaction $ \alpha\in \Mor(A, A \otimes \C(\G))$ is injective.  This is because $ \alpha $  preserves the faithful state $ \tau_{A}$ as observed in \cite{MR3066746}, hence it is injective by Lemma~\ref{lemm:faithful-state}.
\end{remark}

\begin{remark} 
 \label{20thmayremark}
  For a finitely generated countable group $ \Gamma $ and a fixed word-length function $ l, $ consider the orthogonal filtration $\Bfilt:= (\Cred(\Gamma), \tau_\Gamma, \{ B^l_n \}_{n \geq 0}  ) $ of Example \ref{groupalgexample}. Then it can be easily seen that $(\hat{\Gamma}, \hat{\Delta})$ is an object of the category $ \clc (\Bfilt)$; in particular we have a morphism from $\widehat{\Gamma}$ to $\QISO(\Bfilt)$, represented by a Hopf~$^*$\nb-homomorphism $\pi_\Gamma\in \Mor( \C^\univ(\QISO(\Bfilt)), \C^*(\Gamma))$.  Moreover, in \cite{adamjyotishgroup}, it was proved that for $ \Gamma = \mathbb{Z}_n $ ($ n \in \N, n \neq 4 $) $\C (\QISO (\Cfilt)) \cong \Cst(\Gamma) \oplus \Cst(\Gamma)$.
\end{remark}

As an immediate application of Theorem \ref{banicaskalskitheorem}, one can make the following observations.

\begin{lemma} \label{25thjan1}
Let $\Afilt$ be as above and let $(G, \alpha  )$ be an object in the category $ \mathcal{C} (\Afilt).$ Then we have the following:
\begin{enumerate}
\item if $ \{ a_{ij} \mid j = 1,2, \ldots, {\rm dim} ~ ( A_i ) \} $ is a basis of $ A_i,  $ then there exist elements $ q^i_{kj} \in \textup{Pol}(\G)$ \textup{(}$i\geq 0, k,j=1,\ldots, \textup{dim} (A_i)\textup{)}$ such that
\[
\alpha ( a_{ij} ) = \sum_k a_{ik} \otimes q^i_{kj}
\qquad\text{for all~$ j,k = 1,2, \ldots, {\rm dim} ~ ( A_i ) $.}
\]
\item The action $ \alpha $ is faithful if and only if the $\Cst$\nb-algebra generated by $ \{  q^i_{kj}\mid i \geq 0, j,k = 1,2, \ldots, {\rm dim} ~ ( A_i )  \} $ is equal to $\Cont(G).$ 
\item If $ \alpha $ is a faithful action, then the canonical morphism \textup{(}in $ \mathcal{C} ( \Afilt)$\textup{)} from $ \Cont^\univ ( \QISO (\Afilt )   ) $ to $ \Cont (G)$ is surjective.
\end{enumerate}
\end{lemma}

\subsection{Orthogonal filtration of a twisted tensor product} \label{orthfiltsubsect}

Throughout this subsection we will work with the following notation: $ \Afilt:=( A, \tau_A, \{ A_i \}_{i \geq 0} ) $ and $\Bfilt:= ( B, \tau_B, \{ B_j \}_{j \geq 0} ) $ will denote  orthogonal filtrations of unital $\Cst$\nb-algebras $ A $ and $ B,$ $ \gamma_A $ and $ \gamma_B $ will denote the canonical actions of $\QISO(\Afilt)$ on $A$ and $\QISO(\Bfilt)$ on $B$, respectively, while \(\bichar\in\U(\Contvin(\widehat{\QISO(\Afilt )}) 
\otimes\Contvin(\widehat{\QISO(\Bfilt)})\) will be a fixed bicharacter.

Let~\((\alpha,\beta)\) be a~\(\bichar\)\nb-Heisenberg pair on~\(\Hils\). We will work with a realization of~\(A\boxtimes_{\bichar}B\) inside \(A\otimes B\otimes\Bound(\Hils)\) defined via 
embeddings~\(j_{A}\) and~\(j_{B}\) described by~\eqref{eq:def-js}.

Since \(\gamma_{A}\) preserves~\(\tau_{A}\) and~\(\gamma_{B}\) preserves~\(\tau_{B}\), we can apply  \cite{Meyer-Roy-Woronowicz:Twisted_tensor}*{Lemma 5.5} for completely positive maps to
define a functional~\(\tau_{A}\boxtimes_{\bichar}\tau_{B}\colon A\boxtimes_{\bichar}B\to\bc\)  
by 
\begin{equation}
 \label{eq:30thjaneq}  
\tau_{A}\boxtimes_{\bichar}\tau_{B}(j_{A}(a)j_{B}(b))=\tau_{A}(a)\tau_{B}(b) 
\qquad\text{for all \(~a\in A\), \(b\in B\).} 
\end{equation}
\begin{proposition} 
 \label{filtrationonbox}
The functional \(\tau_{A}\boxtimes_{\bichar}\tau_{B}\) is a faithful state on~\(A\boxtimes_{\bichar}B\) and the triple~\(\ABfilt:=(A\boxtimes_{\bichar}B,\tau_{A}\boxtimes_{\bichar}\tau_{B},  \{ j_{A}(A_{i})j_{B}(B_{j}) \}_{i,j\geq 0} )\) is an orthogonal filtration of~\(A \boxtimes_\bichar B\).
\end{proposition}
\begin{proof}
 Define~\(\tau'\colon A\otimes B\otimes\Bound(\Hils)\to\Bound(\Hils)\) by 
\(\tau'\defeq\tau_{A}\otimes\tau_{B}\otimes\Id_{\Hils}\). Then $\tau'$ is a state because~\(\tau_{A}\) and~\(\tau_{B}\) are states. 
By~\cite{Meyer-Roy-Woronowicz:Twisted_tensor}*{Lemma 5.5}, \(\tau_{A}\boxtimes_{\bichar}\tau_{B}\defeq\tau'|_{A\boxtimes B}\). Since \(\tau'\) is  faithful, so is its restriction~\(\tau_{A}\boxtimes_{\bichar}\tau_{B}\). Therefore, \(\tau_{A}\boxtimes_{\bichar}\tau_{B}\) is a faithful state on~\(A\boxtimes_{\bichar}B\).

Let $ S\defeq {\rm Span} \{ j_{A} ( A_i ) j_{B} ( B_j ) : i,j \geq 0 \}$. Since the density of $ S $ in $ A \boxtimes_\bichar B $ is clear by the definition of $ A \boxtimes_\bichar B, $ we only need to prove that~\(\{j_{A}(A_{i})j_{B}(B_{j}\}_{i,j\geq 0}\) is orthogonal with respect to~\(\tau_{A}\boxtimes_{\bichar}\tau_{B}\). Indeed, for all~\(a_{i}\in A_{i}\), \(b_{j}\in B_{j}\), \(a_{k}\in A_{k}\) and~\(b_{l}\in B_{l}\), we have:
\begin{align*}
 c :&= \tau_{A}\boxtimes_{\bichar}\tau_{B}\bigl((j_{A}(a_{i})j_{B}(b_{j}))^{*} j_{A}(a_{k})j_{B}(b_{l})\bigr)\\
 &=\tau_{A}\boxtimes_{\bichar}\tau_{B}\bigl(j_{B}(b_{j}^{*})j_{A}(a_{i}^{*}a_{k})j_{B}(b_{l})\bigr)\\
 &=(\tau_{A}\otimes\tau_{B}\otimes\Id_{\Hils})
 \Bigl(\bigl((\Id_{B}\otimes\beta)(\gamma_{B}(b_{j}^{*}))\bigr)_{23} \bigl((\Id_{A}\otimes\alpha)(\gamma_{A}(a_{i}^{*}a_{k}))\bigr)_{13}\\
 & \;\;\;\;\; \bigl((\Id_{B}\otimes\beta)(\gamma_{B}(b_{l}))\bigr)_{23} \Bigr)\\
 &=(\tau_{B}\otimes\Id_{\Hils}) 
 \Bigl(\bigl (\Id_{B}\otimes\beta)(\gamma_{B}(b_{j}^{*}))\bigr) \bigl((\tau_{A}\otimes\alpha)(\gamma_{A}(a_{i}^{*}a_{k}))\bigr)
 \bigl((\Id_{B}\otimes\beta)(\gamma_{B}(b_{l}))\bigr) \Bigr)\\
  &=\tau_{A}(a_{i}^{*}a_{k})(\tau_{B}\otimes\beta)(\gamma_{B}(b_{j}^{*}b_{l}))
     =\tau_{A}(a_{i}^{*}a_{k})\tau_{B}(b_{j}^{*}b_{l})1_{\mathcal{B}(\Hils)}.
  \end{align*}
 Therefore, if $ ( i, j ) \neq ( k, l ), $ then~\(c=0\) since~\(\tau_{A}(a_{i}^{*}a_{k})=0\) if~\(i\neq k\) and~\(\tau_{B}(b_{j}^{*}b_{l})=0\)  if~\(j\neq l\). This proves that \(\{j_{A}(A_{i})j_{B}(B_{j})\}_{i,j\geq 0}\) yields an orthogonal filtration with respect to 
\(\tau_{A}\boxtimes\tau_{B}\). 
\end{proof}

 Let then~\((\G,\gamma_{A}')\) and \((\G[H],\gamma_{B}')\) be objects in $ {\bf \clc} ( \Afilt) $ and~$ {\bf \clc} (\Bfilt) $, respectively. Suppose~\(\bichar_{1}\in\U(\Contvin(\DuG)\times\Contvin(\DuG[H]))\) is a bicharacter. Then universality of~\(\QISO(\Afilt)\) and~\(\QISO(\Bfilt)\) gives the existence Hopf \Star{}homomorphisms~\(f_{1}\colon\Cont^\univ(\QISO(\Afilt))\to\Cont(\G)\) 
 and~\(f_{2}\colon\Cont^\univ(\QISO(\Bfilt))\to\Cont(\G[H])\) such that~\((\Id_{A}\otimes f_{1})\circ \gamma^\univ_{A}=\gamma_{A}'\) and 
 \((\Id_{B}\otimes f_{2})\circ \gamma^\univ_{B}=\gamma_{B}'\). These admit universal lifts and by Theorem~\ref{the:equiv_homs}
 induce dual Hopf~\Star{}homomorphisms 
 \[
    \hat{f}_{1}\colon\Contvin(\DuG)\to\Contvin(\widehat{\QISO(\Afilt)})\
    \qquad\text{and}\qquad
    \hat{f}_{2}\colon\Contvin(\DuG[H])\to\Contvin(\widehat{\QISO(\Bfilt)}).
\] 
The latter maps allow us to define a bicharacter $\bichar \in \U(\Contvin(\widehat{\QISO(\Afilt)})\otimes\Contvin(\widehat{\QISO(\Bfilt)})$ by the formula~\(\bichar= (\hat{f}_{1}\otimes\hat{f}_{2})(\bichar_{1})\).
\begin{corollary}
 \label{cor:feb6th}
 In the situation above, there is a faithful state~\(\tau_{A}\boxtimes_{\bichar_{1}}\tau_{B}\) 
 on~\(A\boxtimes_{\bichar_{1}}B\) such that the triplet~\((A\boxtimes_{\bichar_{1}}B,\tau_{A}\boxtimes_{\bichar_{1}}\tau_{B}, 
 \{j'_{A}(A_{i})j'_{B}(B_{j})\}_{i,j\geq 0})\) is an orthogonal filtration of~\(A\boxtimes_{\bichar_{1}}B\), where~\(j'_{A}\) and 
 \(j'_{B}\) are embeddings of~\(A\) and \(B\) into~\(A\boxtimes_{\bichar_{1}}B\).
\end{corollary}
\begin{proof}
Clearly, \(\gamma'_{A}\) and~\(\gamma'_{B}\) are injective coactions because they preserve~\(\tau_{A}\) and~\(\tau_{B}\) respectively. Now Lemma~\ref{lem:Feb9} shows that \(\gamma'_{A}\) and~\(\gamma'_{B}\) are induced by the Hopf~\Star{}homomorphism 
\(f_{1}\) and~\(f_{2}\), respectively. Therefore,  by~\cite{Meyer-Roy-Woronowicz:Twisted_tensor}*{Theorem 5.2}, \(A\boxtimes_{\bichar}B\) and 
\(A\boxtimes_{\bichar_{1}}B\) are \emph{equivalent}: there is an isomorphism~\(\Theta\in\Mor(A\boxtimes_{\bichar}B, A\boxtimes_{\bichar_{1}}B)\) such that 
\begin{equation}
 \label{eq:cross-iso}
  \Theta \circ j_{A}= j'_{A}, 
  \qquad
  \Theta \circ j_{B}=j'_{B}.
\end{equation}
Then~\(\tau_{A}\boxtimes_{\bichar_{1}}\tau_{B}\defeq (\tau_{A}\boxtimes_{\bichar}\tau_{B})\circ\Theta^{-1}\) defines a faithful state 
on~\(A\boxtimes_{\bichar_{1}}B\) and the double-indexed family~\(\{j'_{A}(A_{i})j'_{B}(B_{j})\}_{i,j\geq 0}\) defines an orthogonal filtration of~\(A\boxtimes_{\bichar_{1}}B\) 
with respect to~\(\tau_{A}\boxtimes_{\bichar_{1}}\tau_{B}\).
\end{proof}

\section{Quantum symmetries of twisted tensor product} 
\label{sec:sec5}

Let~$ \Afilt \defeq( A, \tau_A, \{ A_i \}_{i \geq 0} ) $ and $\Bfilt \defeq ( B, \tau_B, \{ B_j \}_{j \geq 0} ) $ be orthogonal filtrations of unital $\Cst$\nb-algebras $ A $ and $ B$. Let $ \gamma_A $ and $ \gamma_B $  denote the actions of  $ \QISO(\Afilt )  $ and $\QISO(\Bfilt )  $ on $ A $ and $  B $, respectively. Let~\((\G,\gamma_{A}')\) and \((\G[H],\gamma_{B}')\) be objects in $ {\bf \clc} ( \Afilt) $ and~$ {\bf \clc} (\Bfilt) $ respectively, and  suppose we have a bicharacter~\(\bichar_{1}\in\U(\Contvin(\DuG)\otimes\Contvin(\DuG[H]))\).  Let $\bichar \in \U(\Contvin(\widehat{\QISO(\Afilt)})\otimes\Contvin(\widehat{\QISO(\Bfilt)})$ be the associated bicharacter as in Corollary \textup{\ref{cor:feb6th}}.  By Proposition \ref{filtrationonbox}, $ \{ j_A ( A_i ) j_B ( B_j )    \}_{i,j \geq 0} $ is an orthogonal filtration of $ A \boxtimes_\bichar B $ with respect to the state $ \tau_A \boxtimes_\bichar \tau_B .$ The resulting  triple will be denoted by  $\ABfilt$.
Finally let~\(\GDrin{\bichar}\) denote the Drinfeld double of~\( \QISO(\Afilt) \) and~\( \QISO(\Bfilt)  \) with respect to the bicharacter~\(\bichar\). The aim of this section is to prove Theorem \ref{cor:equiv-cross-iso} which states that  \(\QISO(\ABCfilt)\) is isomorphic to~\(\GDrin{\bichar}\). It turns out that this conclusion can be easily derived from the following theorem:
\begin{theorem} 
  \label{maintheorem} Let $\Afilt$, $\Bfilt$ be orthogonal filtrations, and fix a reduced bicharacter $\bichar\in \U ( \Contvin(\widehat{\QISO( \Afilt )}) \otimes \Contvin(\widehat{\QISO (\Bfilt )}))$.
  The quantum symmetry group \(\QISO( \ABfilt)\), whose existence is guaranteed by \textup{Proposition~\ref{filtrationonbox}}, is isomorphic to $\GDrin{\bichar}$, the Drinfeld double of~\( \QISO(\Afilt) \) and~\( \QISO(\Bfilt)  \) with respect to the bicharacter~\(\bichar\). 
\end{theorem}

For the rest of the section, the symbol \(\bichar\) will denote a fixed reduced bicharacter belonging to $ \U ( \Contvin(\widehat{\QISO( \Afilt )}) \otimes \Contvin(\widehat{\QISO (\Bfilt )})).$ 
As a preparation for proving Theorem \ref{maintheorem}, we will first prove some auxiliary results.
Remark~\ref{prop20thmay} shows that the coactions~\(\gamma_{A}\) and~\(\gamma_{B}\) are injective. By the same argument the actions~\(\gamma^\univ_{A}\) and~\(\gamma^\univ_{B}\) defined on the universal level are normal. Thus, by Lemma~\ref{lemm:Coact-univ-Drinf}, there is a coaction~\(\gamma_{A}^\univ\bowtie_{\bichar}\gamma_{B}^\univ\) of~\(\Contvin^\univ(\GDrin{\bichar})\) on~\(A\boxtimes_{\bichar}B\) satisfying \eqref{eq:uniq-ext}. This allows us to show the following fact.

\begin{lemma}
  \label{17thjan1} The pair 
 \((\GDrin{\bichar},\gamma_{A}^\univ\bowtie_{\bichar}\gamma_{B}^\univ)\) is an object in the category 
 \(\clc (\ABfilt)\). Moreover, \(\gamma_{A}^\univ\bowtie_{\bichar}\gamma_{B}^\univ\) is a faithful coaction.
\end{lemma}  
\begin{proof}
    Let $i\geq 0$, and let $ \{ v_{i,x}\mid x = 1,2, \ldots, {\rm dim} ( A_i ) \} $ be a basis of $ A_i. $ By Lemma \ref{25thjan1}, we have elements $ \{ q^i_{kl}\mid k,l = 1,2, \ldots, {\rm dim} ( A_i ) \} \in \Pol(  \QISO (\Afilt) )\subset \C^\univ(\QISO(\Afilt)  )$ such that for each $k=1,\ldots, {\rm dim} ( A_i ) $ 
		$$ \gamma^\univ_A ( v_{i,k} ) = \sum_{l=1}^{\dim ( A_i )} v_{i,l} \otimes q^i_{lk}. $$
		Moreover, by virtue of \eqref{eq:uniq-ext}, 
		$$ \gamma_{A}^\univ\bowtie_{\bichar}\gamma_{B}^\univ ( j_A ( v_{i,k} ) ) = \sum_{l=1}^{\dim ( A_i )} j_A ( v_{i,l} ) \otimes \rho^\univ (  q^i_{lk}  ). $$
Similarly, if $j \geq 0$ and $ \{ w_{j,m} \mid m = 1,2, \ldots, {\rm dim} ( B_j ) \} $ is a basis of $ B_j, $ then we have elements $ \{ r^j_{mn}\mid m,n = 1,2, \ldots, {\rm dim} ( B_j ) \} \in \Cont^\univ( \QISO (\Bfilt ) ) $ such that for $m = 1,2, \ldots, {\rm dim} ( B_j )$
	$$	\gamma_{A}^\univ\bowtie_{\bichar}\gamma_{B}^\univ ( j_B ( w_{j,m} ) ) = \sum_{n=1}^{\dim ( B_j )} j_B ( w_{j,n} ) \otimes \theta^\univ ( r^j_{nm} ). $$
	Therefore, we can conclude that
	\begin{equation} \label{25thjan2}
	 \gamma_{A}^\univ\bowtie_{\bichar}\gamma_{B}^\univ ( j_A ( v_{i,k} ) j_B ( w_{j,m} ) ) = \sum_{l,n} j_A ( v_{i,l} ) j_B ( w_{j,n} ) \otimes \rho^\univ ( q^i_{lk} ) \theta^\univ ( r^j_{nm} ).
	\end{equation}
	Thus the map $ \gamma_{A}^\univ\bowtie_{\bichar}\gamma_{B}^\univ $ preserves the subspace $ j_A ( A_i ) j_B ( B_j ) $ for each $i,j \geq 0$.  This proves the first assertion.
	
Now we prove that the action $\gamma_{A}^\univ\bowtie_{\bichar}\gamma_{B}^\univ $ is faithful. By  Lemma \ref{25thjan1} (2) and by the equality \eqref{25thjan2} it is enough to show that the $\Cst$\nb-algebra generated by the set $ \{ \rho^\univ (  q^i_{l,k} ) \theta^\univ (  r^j_{n,m}  )\mid i,j \geq 0, l,k=1,\ldots,{\rm dim} ( A_i ), n, m = 1, \ldots, {\rm dim} ( B_j )   \} $ is equal to $\Cont^\univ(\GDrin{\bichar})$.
	
	Since $ \gamma^\univ_A $ and $ \gamma^\univ_B $ are faithful coactions, we have:
	 $$ C^* \{ q^i_{lk}\mid i \geq 0, l,k = 1, \ldots, {\rm dim} ( A_i ) \} = \Cont^\univ( \QISO (\Afilt)  ),$$
	 $$ C^* \{ r^j_{nm}\mid j\geq 0, n,m=1, \ldots, {\rm dim} ( B_j ) \} = \Cont^\univ( \QISO (\Bfilt )  ). $$
As $\Cont^\univ(\GDrin{\bichar})= \rho^\univ (  \Cont^\univ( \QISO ( \Afilt )  )  ) \theta^\univ (  \Cont^\univ( \QISO ( \Bfilt ) )  ),$ the proof is completed. 
\end{proof}

Next we prove the following auxiliary result. 
\begin{lemma}
 \label{lemma16thjan1}
 Let~\(\gamma^\univ\) denote the coaction of~\(\Cont^\univ(\QISO(\ABfilt))\) on~\(A\boxtimes_{\bichar}B\). Then 
 \begin{alignat}{2}
  \label{univ-coact-slice_A}
  (\tau_{A}\boxtimes_{\bichar}\Id_{B}\otimes\Id_{\QISO(\ABfilt)})(\gamma^\univ(j_{A}(a)j_{B}(b)))
  &= \tau_{A}(a)\gamma^\univ(j_{B}(b)),\\
  \label{univ-coact-slice_B}  
   (\Id_{A}\boxtimes_{\bichar}\tau_{B}\otimes\Id_{\QISO(\ABfilt)})(\gamma^\univ(j_{A}(a)j_{B}(b)))
  &=\tau_{B}(b)\gamma^\univ(j_{A}(a)) 
 \end{alignat}
 for all~\(a\in A\) and~\(b\in B\).
\end{lemma}
\begin{proof}
 By~\cite{Meyer-Roy-Woronowicz:Twisted_tensor}*{Lemma 5.5}, concerning equivariant completely positive maps, we have 
\begin{equation}
 \label{eq:20janaux1}
   \tau_{A}\boxtimes_{\bichar}\Id_{B}(j_{A}(a)j_{B}(b))=\tau_{A}(a)j_{B}(b),
   \qquad
   \Id_{A}\boxtimes_{\bichar}\tau_{B}((j_{A}(a)j_{B}(b))=\tau_{B}(b)j_{A}(a)
\end{equation}
for all~\(a\in A\), \(b\in B\).

By Lemma~\ref{17thjan1} and  Lemma \ref{25thjan1} (3) it follows that there is a unique surjective Hopf \Star{}homomorphism \(q\colon \Cont^\univ(\QISO(\ABfilt))\to\Cont^\univ(\GDrin{\bichar})\) such 
that 
\begin{alignat}{2}
   \label{eq:18janaux4}
        (\Id_{A\boxtimes_{\bichar}B}\otimes q)\circ \gamma^\univ \circ  j_{A} 
  &=\gamma_{A}^\univ\bowtie\gamma_{B}^\univ \circ j_{A}
  &=(j_{A}\otimes\rho^\univ) \circ \gamma_{A}^\univ,\\
  \label{eq:18janaux5} 
  (\Id_{A\boxtimes_{\bichar}B}\otimes q)\circ \gamma^\univ \circ j_{B}
  &=\gamma_{A}^\univ\bowtie\gamma_{B}^\univ \circ j_{B}
  &=(j_{B}\otimes\theta^\univ)\circ \gamma_{B}^\univ.
\end{alignat}
Notice that the morphism \(\gamma_{A}^\univ\bowtie_{\bichar}\gamma_{B}^\univ\) is injective as it preserves the faithful 
state~\(\tau_{A}\boxtimes_{\bichar}\tau_{B}\). Using~\eqref{eq:20janaux1}, \eqref{eq:18janaux4} and \eqref{eq:18janaux5} we get
\begin{align*}
  &     (\tau_{A}\boxtimes_{\bichar}\Id_{B}\otimes q)(\gamma^\univ(j_{A}(a)j_{B}(b))) \\
 &=(\tau_{A}\boxtimes_{\bichar}\Id_{B}\otimes\Id_{\QISO(\ABfilt)})
 \bigl((j_{A}\otimes\rho^\univ)(\gamma_{A}^\univ(a))(j_{B}\otimes\theta^\univ)(\gamma_{B}^\univ(b))\bigr)\\
 &=(\tau_{A}\otimes\rho^\univ)(\gamma_{A}^\univ(a))(j_{B}\otimes\theta^\univ)(\gamma_{B}^\univ(b))\\
 &=\tau_{A}(a)(j_{B}\otimes\theta^\univ)(\gamma_{B}^\univ(b))
   =(\Id_{A\boxtimes_{\bichar}B}\otimes q)\bigl(\tau_{A}(a)\gamma^\univ( j_{B}(b))\bigr). 
\end{align*}
Note that we have shown above that 
\[(\Id_{A\boxtimes_{\bichar}B}\otimes q) \bigl( (\tau_{A}\boxtimes_{\bichar}\Id_{B}\otimes \Id)(\gamma^\univ(j_{A}(a)j_{B}(b)))\bigr) =  (\Id_{A\boxtimes_{\bichar}B}\otimes q)\bigl(\tau_{A}(a)\gamma^\univ( j_{B}(b))\bigr)\]

The equation
$ ( {\rm id} \otimes q  ) \circ \gamma^\univ = \gamma_{A}^\univ\bowtie\gamma_{B}^\univ $ implies that $  {\rm id} \otimes q  $ is injective on $ {\rm Ran} ( \gamma^\univ ). $ Thus it follows that
\eqref{univ-coact-slice_A} holds, if only we can show that $(\tau_{A}\boxtimes_{\bichar}\Id_{B}\otimes \Id)(\gamma^\univ(j_{A}(a)j_{B}(b)))$ belongs to the image of $\gamma^\univ$. For that (by density) we may assume that $a \in A_i$, $b \in B_j$ for some $i,j \geq 0$. Since by definition the coaction   $\gamma^\univ$ preserves the subspaces $ j_A ( A_i ) j_B ( B_j ) $ for all $i,j\geq 0,$ it has to preserve the subspaces $ j_A ( A_i ) $ and $ j_B ( B_j ) $ and hence  $\gamma^\univ(j_A(a)) =\sum_{k=1}^n j_A(a_k) \otimes x_k$, $\gamma^\univ(j_B(b)) =\sum_{k=1}^n j_B(b_k) \otimes y_k$ for some $n \in \N$, $a_1,\ldots, a_n \in A_i$, $b_1, \ldots,b_n \in B_j$ and $x_1, \ldots, x_n, y_1, \ldots, y_n \in  \Cont^\univ(\QISO(\ABfilt))$. Thus 
\begin{align*} (\tau_{A}\boxtimes_{\bichar}&\Id_{B}\otimes \Id)(\gamma^\univ(j_{A}(a)j_{B}(b))) = \sum_{j,k=1}^n (\tau_{A}\boxtimes_{\bichar}\Id_{B}\otimes \Id)
(j_A(a_j)j_B(b_k) \otimes x_jy_k) \\&= \sum_{j,k=1}^n j_A(\tau_{A}(a_j))j_B(b_k) \otimes x_jy_k = \bigl(\sum_{j=1}^n j_A(\tau_{A}(a_j))j_B(1) \otimes x_j \bigr)\gamma^\univ(j_{B}(b)) 
\\&= (\tau_{A}\boxtimes_{\bichar}\tau_{B}\otimes \Id) (\gamma^\univ(j_{A}(a)j_{B}(1))) \gamma^\univ(j_{B}(b))  
\\&= \bigl(\tau_{A}\boxtimes_{\bichar}\tau_{B} (j_{A}(a)j_{B}(1))\bigr) \gamma^\univ(j_{B}(b))=\tau_{A}(a) \gamma^\univ(j_{B}(b)),
\end{align*}
where in the second equality we used \cite{Meyer-Roy-Woronowicz:Twisted_tensor}*{Lemma 5.5}. This shows the desired containment and completes the proof of \eqref{univ-coact-slice_A}.

Similarly we can show that~\eqref{univ-coact-slice_B} holds.
\end{proof}

We are now ready to prove Theorem~\textup{\ref{maintheorem}}.

\begin{proof}[Proof of Theorem~\textup{\ref{maintheorem}}]
From the Podle\'s condition for~\(\gamma^\univ\) we get
 \[
 \gamma^\univ(A\boxtimes B)(1_{A\boxtimes_{\bichar}B}\otimes\Cont^\univ(\QISO(\ABfilt)))=A\boxtimes B\otimes\Cont^\univ(\QISO(\ABfilt)).
\] 
Applying~\(\tau_{A}\boxtimes_{\bichar}\Id_{B}\otimes\Id_{\QISO(\ABfilt)}\) to the 
both sides of the above equality and using~\eqref{univ-coact-slice_A} and \eqref{eq:20janaux1} gives 
\[
 \gamma^\univ(j_{B}(B))(1_{j_{B}(B)}\otimes\Cont^\univ(\QISO(\ABfilt)))= j_{B}(B)\otimes\Cont^\univ(\QISO(\ABfilt)).
\]
Thus,~\(\gamma^\univ(j_{B}(B))\subseteq j_{B}(B)\otimes\Cont^\univ(\QISO(\ABfilt))\). Therefore, 
\(\tilde{\gamma}_{B}\defeq \bigl(j_{B}^{-1}\otimes\Id_{\QISO(\ABfilt)}\bigr)\circ\gamma^\univ\circ j_{B}\) defines 
 a coaction of~\(\Cont^\univ(\QISO(\ABfilt))\) on 
\(B\).
Moreover, \((\QISO(\ABfilt),\tilde{\gamma}_{B})\) is an object in $ \clc (\Bfilt)$ and so by Theorem~\ref{banicaskalskitheorem}, there is a Hopf~\Star{}homomorphism~\(\theta_{1}\colon \Cont^\univ(\QISO(\Bfilt))\to\Cont^\univ(\QISO(\ABfilt))\) such that 
\((\Id_{B}\otimes\theta_{1})\circ \gamma^\univ_{B}=\tilde{\gamma}_{B}\). This yields the following equality: 
\[
   (\Id_{B}\otimes\theta_{1})\circ \gamma^\univ_{B}=(j_{B}^{-1}\otimes\Id_{\QISO(\ABfilt)})\circ \gamma^\univ \circ j_{B}.
\]
Hence,
\begin{equation}
 \label{eq:18janaux1}
(j_{B}\otimes\theta_{1})\circ \gamma^\univ_{B}=\gamma^\univ \circ j_{B}.\end{equation}
Similarly, we can show that there is a coaction~\(\tilde{\gamma}_{A}\) of~\(\Cont^\univ(\QISO(\ABfilt))\) on~\(A\) and a Hopf~\Star{}homomorphism \(\rho_{1}\colon\Cont^\univ(\QISO(\Afilt))\to\Cont^\univ(\QISO(\ABfilt))\) such that 
\begin{equation}
  \label{eq:18janaux2}
 (j_{A}\otimes\rho_{1}) \circ \gamma^\univ_{A}=\gamma^\univ \circ j_{A}.
\end{equation}
By the universal property of~\(\Cont^\univ(\GDrin{\bichar})\) proved in Theorem~\ref{the:UnivProp-Drinf}, there is a unique 
Hopf~\Star{}homomorphism \(\Psi\colon \Cont^\univ(\GDrin{\bichar})\to \Cont^\univ(\QISO(\ABfilt))\) 
such that~\(\Psi\circ\rho^\univ=\rho_{1}\), \(\Psi\circ\theta^\univ=\theta_{1}\) and
\begin{alignat}{2}
   \label{eq:18janaux3}
(\Id_{A\boxtimes_{\bichar}B}\otimes\Psi)\circ\gamma^\univ_{A}\bowtie\gamma^\univ_{B}\circ j_{A} 
&=\gamma^\univ\circ j_{A}\\
   \label{eq:18janaux3}
(\Id_{A\boxtimes_{\bichar}B}\otimes\Psi)\circ \gamma^\univ_{A}\bowtie\gamma^\univ_{B}\circ j_{B}
&=\gamma^\univ\circ j_{B}
\end{alignat}

Using~\eqref{eq:18janaux2}, \eqref{eq:18janaux3} and~\eqref{eq:18janaux4} we have
\begin{align*}
(\Id_{A\boxtimes_{\bichar}B}\otimes \Psi\circ q)\circ\gamma^\univ \circ j_{A} 
= (j_{A}\otimes\Psi\circ\rho^\univ)\circ\gamma_{A}^\univ
= (j_{A}\otimes\rho_{1})\circ \gamma_{A}^\univ
=\gamma^\univ\circ j_{A}.
\end{align*}
Similarly, we can show that
\((\Id_{A\boxtimes_{\bichar}B}\otimes \Psi\circ q)\circ\gamma^\univ \circ j_{B} 
=\gamma^\univ\circ j_{B}\). Therefore, for all~\(\omega 
\in (A\boxtimes_{\bichar}B)'\) we have 
\[
(\Psi\circ q)((\omega\otimes\Id_{\QISO(\ABfilt)})(\gamma^\univ(x)))=(\omega\otimes\Id_{\QISO(\ABfilt)})(\gamma^
\univ(x))
\]
for all~\(x\in A\boxtimes_{\bichar}B\). Faithfulness of~\(\gamma^\univ\) gives 
\(\Psi\circ q(c)=c\) for all~\(c\in\Cont^\univ(\QISO(\ABfilt))\).

A similar computation gives 
\[
 (\Id_{A\boxtimes_{\bichar}B}\otimes q\circ \Psi)(\gamma^\univ_{A}\bowtie\gamma^\univ_{B}(j_{A}(a)j_{B}(b)))
 =\gamma^\univ_{A}\bowtie\gamma^\univ_{B}(j_{A}(a)j_{B}(b))
\]
for all~\(a\in A\), \(b\in B\). Finally, by talking slices with~\(\omega\in (A\boxtimes_{\bichar}B)' \) on the first leg of the the both sides in the above equation  and using faithfulness of~\(\gamma^\univ_{A}\bowtie\gamma^\univ_{B}\) we obtain
\((q\circ\Psi)(d)=d\) for all~\(d\in\Cont^\univ(\GDrin{\bichar})\).
\end{proof}

 Thus we are in a position to prove the  main result of this article:

\begin{theorem}
 \label{cor:equiv-cross-iso}
Let~\((\G,\gamma_{A}')\) and \((\G[H],\gamma_{B}')\) be objects in $ {\bf \clc} ( \Afilt) $ and~$ {\bf \clc} (\Bfilt) $ respectively, and  suppose we have a bicharacter~\(\bichar_{1}\in\U(\Contvin(\DuG)\otimes\Contvin(\DuG[H]))\).  Let $\bichar \in \U(\Contvin(\widehat{\QISO(\Afilt)})\otimes\Contvin(\widehat{\QISO(\Bfilt)})$ be the associated bicharacter as in \textup{Corollary~\ref{cor:feb6th}}. Then the quantum symmetry group \(\QISO(\ABCfilt)\) is isomorphic to~\(\GDrin{\bichar}\).
\end{theorem}
\begin{proof}
The result is an immediate consequence of Theorem~\textup{\ref{maintheorem}}, Corollary \textup{\ref{cor:feb6th}} and its proof.
 \end{proof}
In particular, we can 
choose~\(\bichar=1\in\U(\Contvin(\widehat{\QISO(\Afilt)})\otimes\Contvin(\widehat{\QISO(\Bfilt)})\). Then~\(A\boxtimes_{\bichar}B\cong A\otimes B\). Also, by virtue of \cite{Roy:Codoubles}*{Example 5.10} the reduced Drinfeld double of~\(\QISO(\Afilt)\) and 
\(\QISO(\Bfilt)\) with respect to~\(\bichar\) is~\(\QISO(\Afilt)\otimes\QISO(\Bfilt)\). 
Thus, denoting the filtration of $A \otimes B$ coming from Proposition \ref{filtrationonbox} by $\Afilt \otimes \Bfilt$ and using the standard Cartesian product construction for compact quantum groups (so that $\C^\univ(\G\times \GH)=\C^\univ(\G) \otimes_{{\rm max}} \C^\univ(\GH)$) we obtain the following corollary.

\begin{corollary}
 \label{cor:triv-bichar}
 The quantum symmetry group \(\QISO(\Afilt\otimes\Bfilt)\) is 
 isomorphic to \(\QISO(\Afilt)\times\QISO(\Bfilt)\).
\end{corollary}

In the next example, we apply Theorem \ref{cor:equiv-cross-iso} to describe the quantum symmetry group of a class of Rieffel deformations of unital $ \textup{C}^{*} $-algebras by actions of compact groups (\cite{rieffel}, \cite{kasprzak}). These are examples which are not necessarily of the crossed product type; the next section is devoted to the examples arising as reduced crossed products.

\begin{example}
Let $ A $ and $ 
B $ be unital $ \textup{C}^{*} $-algebras equipped with orthogonal filtrations. 
Assume that $ G $ and $ H $ are compact abelian groups acting respectively on $A$ and on $B$ in the filtration preserving way (so that they are objects of respective categories). Moreover, let $ \chi: \hat{G} \times \hat{H} \rightarrow \mathbb{T} $ be a bicharacter. The coactions $ \alpha_A: A \rightarrow A \otimes \Cont ( G ) $ and $ \alpha_B: B \rightarrow B \otimes \Cont ( H ) $ define a canonical coaction $ \gamma $ of $ \Cont ( K ):= \Cont ( G \times H ) $ on $ E:= A \otimes B. $ Furthermore $ \chi $ defines a bicharacter $ \psi $ on $ \hat{K} $ via the formula
	$$ \psi: \hat{K} \times \hat{K} \rightarrow \mathbb{T}, ~ \psi (  ( g_1, h_1   ), ~ ( g_2, h_2 )   ) = \chi ( g_2, h_1  )^{-1}, \;\;\; g_1,g_2, h_1, h_2 \in \hat{K}. $$
	Since $ \psi $ is a bicharacter, it defines a $2$-cocycle  on the group $ \hat{K}. $ The Rieffel deformation of the data $ ( E, \gamma, \psi  ) $ yields a new unital $ \textup{C}^{*}   $-algebra $ E_\psi. $
	
	By Theorem 6.2 of \cite{Meyer-Roy-Woronowicz:Twisted_tensor}, $ E_\psi $ is isomorphic to $ A \boxtimes_\psi B. $  Therefore, we can apply Theorem \ref{cor:equiv-cross-iso} to compute the quantum symmetry group of $ E_\psi. $ Concrete examples can be obtained in the following way, using the notions appearing in the next section: take $ A $ and $ B $ to be any two $ \textup{C}^{*} $ algebras appearing in Examples \ref{3rdjuly}, \ref{18thdec2}, \ref{18thdec3}, \ref{18thdec4} and $ G = H = \mathbb{T}^n. $ Then the homomorphism $ f_{{\bf \lambda}} $ of Proposition \ref{18thdec} defines the required bicharacter.
\end{example}

\section{The case of the reduced crossed products} 
 \label{sec:sec6}

In this section we will apply the general results obtained before for quantum symmetry groups of twisted tensor products to the case of crossed products by discrete group actions.

\subsection{Quantum symmetries of  reduced crossed products}

Throughout this subsection, we will adopt the following notations and conventions.  The triple $ \Afilt\defeq ( A, \tau_A, \{ A_i \}_{i \geq 0}  ) $ will denote  an orthogonal filtration of  a unital $\Cst$\nb-algebra $ A. $ Further $ \Gamma $ will denote a discrete countable group with a neutral element $e$ and a proper length function $ l: \Gamma \to \mathbb{N}_0 $, so that the $\Cst$\nb-algebra $ B\defeq \Cred(\Gamma) $ has the orthogonal filtration $\Bfilt \defeq ( \Cred(\Gamma), \tau_\Gamma, \{ B^l_n \}_{n \geq 0}    ) $ as in Example \textup{\ref{groupalgexample}}.   Symbols $ \gamma^\univ_A $ and $ \gamma_A $ will denote  respectively the universal and reduced version of the action of $\QISO (\Afilt) $ on $A$, and  \(\multunit[\Gamma]\) will denote the reduced bicharacter associated to $\Gamma$ (see  Theorem~\ref{18thjan1}). Given any action $\beta$ of $\Gamma$ on $A$ (classically viewed as  a homomorphism from $\Gamma$ to $\textup{Aut}(A)$, but here interpreted as a coaction of $\Contvin(\Gamma)$, that is a morphism $\beta\in \Mor(A, A \otimes \Contvin(\Gamma))$ satisfying the equation~\eqref{eq:right_action}), we will denote the resulting reduced crossed product, contained in $\Mult(A\otimes \Comp(\ell^2(\Gamma)))$, as $A \rtimes_{\beta, \red} \Gamma$. As customary, we will write then $a \lambda_g$ for   $ \beta ( a ) ( 1 \otimes \lambda_g ) $, where $a\in A, g \in \Gamma$. Moreover
we will write $ \tau:=  \tau_A \circ \tau^\prime\in S( A \rtimes_{\beta, \red} \Gamma) $, where $ \tau^\prime $ is the canonical conditional expectation from $ A \rtimes_{\beta, \red} \Gamma $ onto $ A $ defined by the continuous linear extension of the prescription $ \tau^{\prime} ( \sum_g a_g \lambda_g  ) = a_e$. 
Finally, given the data as above define for each $i,j \geq 0$
\begin{equation} A_{ij} \defeq {\rm span} \{ a_i \lambda_{g_j} \mid a_i \in A_i, l ( g_j ) = j   \}. \label{Aijdeft}\end{equation}

Note first the following easy lemma, extending Theorem~\ref{18thjan1}.

\begin{lemma} \label{31stjan2}
Let $\Afilt$, $\Bfilt$ be as above and fix an action $\beta$ of $\Gamma$ on $A$.
	The isomorphism~$ \Psi\colon A \boxtimes_\bichar B \rightarrow  A \rtimes_{\beta,\red} \Gamma $, discussed in \textup{Theorem~\ref{18thjan1}} has the following properties: 
	\[
	\tau = (\tau_{A}\boxtimes_{\multunit[\Gamma]}\tau_{\Gamma})\circ \Psi,\] and  for each $i, j \geq 0$ we have
	\[\Psi(j_{A}(A_{i})j_{B}(B_{j}^{l})) = A_{ij}.\] 
\end{lemma}

\begin{proof}  
Easy computation.\qedhere
\end{proof}

The above lemma shows that we have a natural candidate for an orthogonal filtration of the reduced crossed product. This raises two natural questions: first, when does the family $\{A_{ij}:i,j\geq 0\}$ indeed form an orthogonal filtration, and second, when can we determine the respective quantum symmetry group. 
The next theorem involves a condition which identifies the family $\{A_{ij}:i,j\geq 0\}$ as an orthogonal filtration arising via the construction described in Subsection \ref{orthfiltsubsect}, and further allows us to apply Theorem~\ref{maintheorem} to compute the quantum symmetry group in question. Later, in Proposition \ref{prop15thmarch}, we will see another, more general situation, under which $\{A_{ij}:i,j\geq 0\}$ still forms an orthogonal filtration. Then of course the second question will have to be addressed separately.

\begin{theorem} 
\label{maintheoremcrossed}
Let $\Afilt$, $\Bfilt$ be as above. Suppose that the map  $ \pi\in\Mor(\Cont^\univ(\QISO ( \Afilt )),\Contvin( \Gamma ))$ is a Hopf~\Star{}homomorphism, describing  quantum group morphism from $\Gamma$ to $\QISO(\Afilt)$, and define the action $ \beta\in \Mor(A,A \otimes \Contvin( \Gamma ) ) $ as 
\begin{equation}
 \label{compatibility} 
   \beta\defeq ( {\Id} \otimes \pi )\circ \gamma^\univ_A.
\end{equation}
Then the triplet $( A \rtimes_{\beta,\red} \Gamma, \tau, ( A_{ij} )_{i,j \geq 0})$ is an orthogonal filtration,  denoted by $\crfilt$.

Recall the  Hopf~\Star{}homomorphism $ \pi_\Gamma\in\Mor(\Cont(\QISO ( \Bfilt )),\C^*( \Gamma ))$ mentioned in \textup{Remark \ref{20thmayremark}} and define a bicharacter $\bichar \in\U(\Contvin(\widehat{\QISO(\Afilt)})\otimes\Contvin(\widehat{\QISO(\Bfilt)}))$ as \[\bichar\defeq(\widehat{\pi}\otimes\widehat{\pi_\Gamma})(\WW^{\Gamma}),\]
where $\widehat{\pi}, \widehat{\pi_\Gamma}$ denote the respective dual morphisms.

Then $ \QISO ( \crfilt) \cong \GDrin{\bichar},$ where $\GDrin{\bichar}$ is the Drinfeld double of $\QISO(\Afilt)$ and $\QISO(\Bfilt)$ determined by $\bichar$.

\end{theorem}

\begin{proof}
The proof proceeds via identifying the sets $A_{ij}$ with those constructed via Proposition \ref{filtrationonbox}.

By virtue of Lemma~\ref{lem:Feb9}, the coactions \(\beta\) and \(\DuComult[\Gamma]\) are induced by~\(\pi\) and~\(\pi_{\Gamma}\), respectively. 
Note that~\(\beta\) and~\(\DuComult[\Gamma]\) are injective. Thus we can use \cite{Meyer-Roy-Woronowicz:Twisted_tensor}*{Theorem 5.2} to deduce there is an isomorphism~\(\Theta\colon A\boxtimes_{\bichar}\Cred(\Gamma)\to A\boxtimes_{\multunit[\Gamma]}\Cred(\Gamma)\) such that 
\[
\Theta (j'_{A})=j_{A}, \qquad
\Theta (j'_{B})=j_{B},
\]
where~\(j'_{A}\) and~\(j'_{B}\) are embeddings of~\(A\) and~\(\Cred(\Gamma)\) into~\(A\boxtimes_{\bichar}\Cred(\Gamma)\). An argument similar to that used in the proof of Corollary~\ref{cor:feb6th} shows that the isomorphism~\(\Theta\) maps~\(\tau_{A}\boxtimes_{\bichar}\tau_{\Gamma}\) and~\(\{j'_{A}(A_{i})j'_{B}(B_{j})\}_{i,j\geq 0}\) 
to~\(\tau_{A}\boxtimes_{\multunit[\Gamma]}\tau_{\Gamma}\) and~\(\{j_{A}(A_{i})j_{B}(B_{j})\}_{i,j\geq 0}\), respectively. Then, using 
Lemma~\ref{31stjan2}, we obtain that \(\Psi\circ\Theta\) maps~\(\tau_{A}\boxtimes_{\bichar}\tau_{\Gamma}\) and~\(\{j'_{A}(A_{i})j'_{B}(B_{j})\}_{i,j\geq 0}\)  to \(\tau\) and~\(\{A_{ij}\}_{i,j\geq 0}\), respectively. The map \(\Psi\circ\Theta\) is an isomorphism and 
the triplet \((A\boxtimes_{\bichar}B,\tau_A\boxtimes_{\bichar}\tau_\Gamma, \{j'_{A}(A_{i})j'_{B}(B_{j})\}_{i,j\geq 0})\) is an orthogonal filtration 
by Proposition~\ref{filtrationonbox}. Hence, the triplet $ (  A \rtimes_{\beta,\red} \Gamma, \tau, ( A_{ij} )_{i,j \geq 0} )$ is also an 
orthogonal filtration of $ A \rtimes_{\beta,\red} \Gamma $.

Then Theorem~\ref{maintheorem} allows us to conclude the proof.
\end{proof}
We quickly note that the theorem applies of course to the case of the trivial action.

\begin{corollary}
  Let  $\Afilt$, $\Bfilt$  be the filtrations introduced above. Let $\beta$ denote the trivial action of $\Gamma$ on $A.$ Then $\QISO ( \Afilt \otimes \Bfilt ) $  is isomorphic to $\QISO ( \Afilt )  \times  \QISO ( \Bfilt  )  .$
\end{corollary}
\begin{proof}
It suffices to recall that for a trivial action $ \beta, $ $ A \rtimes_{\beta,\red} \Gamma $ is isomorphic with $ A \otimes \Cred(\Gamma )$ and apply 
Corollary~\ref{cor:triv-bichar}.
\end{proof}

\begin{remark}
Note that the action defined by the formula \eqref{compatibility} preserves the state $\tau_A$. As we will see below, this preservation alone guarantees that the family $\{A_{ij}:i,j \geq 0\}$ forms an orthogonal filtration. The example presented in the last part of this section will show however that  if we only assume that $\beta$ preserves $\tau_A$, then in general the quantum symmetry group need not be of the form discussed above. More specifically we will prove that the quantum symmetry group $\QISO(\crfilt)$ need not be a Drinfeld double of $\QISO(\Afilt)$ and $\QISO(\Bfilt)$  with respect to any bicharacter.
\end{remark}

\subsection{Examples}

In this subsection we present several examples illustrating the scope of 
Theorem \ref{maintheoremcrossed} and also discuss the cases in which it does not apply.
 We begin with the following observation, presenting a general situation where one can apply Theorem \ref{maintheoremcrossed}.

\begin{proposition} \label{18thdec}
Suppose that $ \Afilt \defeq( A, \tau_A, ( A_i )_{i \geq 0} ) $ is an orthogonal filtration and denote as usual by $ \gamma^\univ_A $ the coaction of $ \Cont^\univ( \QISO (\Afilt ))$ on $ A$. Suppose that   $n \in \N$ and we have a quantum group morphism from $\mathbb{T}^n$ to $\QISO(\Afilt)$, described by a  Hopf~\Star{}homomorphism  $ \pi\in \Mor(\Cont^\univ ( \QISO ( \Afilt )), \Cont( \mathbb{T}^n))  $. Further  let $\fatlam\defeq(\lambda_1, \lambda_2, \ldots, \lambda_n) \in \mathbb{T}^n $ and define a homomorphism $ f_{\fatlam}\colon \mathbb{Z}^n \rightarrow \mathbb{T}^n $ by the formula $ f_{\fatlam} ( m_1, m_2, \ldots, m_n ) = \lambda^{m_1}_1 \cdots \lambda^{m_n}_n $ and let $ f_{\fatlam}^*\colon \Cont( \mathbb{T}^n ) \rightarrow \Contvin( \mathbb{Z}^n )  $ denote the associated  Hopf~\Star{}homomorphism. 

Then the formula  $\beta \defeq ( {\rm id} \otimes f_{\fatlam}^*\circ \pi  )\circ \gamma^\univ_A  $
defines an action $ \beta $ of $\Gamma:= \mathbb{Z}^n $ on $ A $ and the $\Cst$\nb-algebra $ A \rtimes_{\beta, \red} \mathbb{Z}^n $ satisfies the conditions of \textup{Theorem \ref{maintheoremcrossed}}.  
\end{proposition}

Proposition \ref{18thdec} can be directly applied to the following class of examples.

\begin{example} \label{3rdjuly}
Let  $ M $ be a compact Riemannian manifold. Assume  that $ \mathbb{T}^n $ is a subgroup of the maximal torus of the isometry group $ {\rm ISO} ( M ) $ for some $n \geq 1. $ Consider the orthogonal filtration $ ( E, \tau, \{ V_i \}_i, J, W ) $ on $ \C ( M ) $ coming from the Hodge-Dirac operator $ d + d^* $ (see  Example 2.5.(1) of   \cite{MR3158722}). It is well known that (Remark 2.13 of \cite {MR3158722} )  $ {\rm ISO} ( M ) $ is an object of the category  $ \mathcal{C} ( E, \tau, \{ V_i \}_i, J, W ) $ and thus we have a quantum group morphism from $ \mathbb{T}^n $ to $ {\rm QISO}  ( E, \tau, \{ V_i \}_i, J, W ). $ Now we can apply Proposition \ref{18thdec}.

In particular, if we take $ M = \mathbb{T}^n, $ the resulting crossed product is the $2n$-dimensional noncommutative torus. 
\end{example}

\begin{example} \label{18thdec2}
For $ q \in ( 0, 1 ) $ and $ G $ a compact semisimple Lie group, let $A:= \Cont( G_q ) $ denote the reduced version of the $ q $-deformation of $ G. $ It is well known that $G_q$ is coamenable.   Consider the (reduced, ergodic) action of $G_q$ on itself, i.e.\ the coproduct $ \Delta: A \rightarrow A \otimes A. $ Then the faithful Haar state $ \tau_A $ of $ A $ is the unique invariant state for the action $ \Delta. $ For an irreducible representation $\pi$ of $G_q$ denote by $A_\pi$   the linear span of its matrix coefficients.  By~\cite{MR3066746}*{Theorem 3.6}, we have an orthogonal filtration  $\Afilt\defeq ( A, \tau_A, \{ A_\pi \}_{\pi\in \textup{Irr}(G_q)} ) $ such that $ ( \Cont ( G_q ), \Delta ) $ is an object of $ \mathcal{C} (\Afilt) .$ Therefore, we have a Hopf~$^*$\nb-homomorphism  $ f: \Cont^\univ ( \QISO (\Afilt )  ) \rightarrow \Cont ( G_q ). $    Let us recall that the maximal toral subgroup  $ \mathbb{T}^n $ of $ G $ is still a quantum subgroup of $G_q$, so that we have a Hopf~$^*$\nb-homomorphism $ g: \Cont ( G_q ) \rightarrow \Cont ( \mathbb{T}^n ). $  Thus we obtain a Hopf~$^*$\nb-homomorphism from $ \Cont^\univ  ( \QISO (\Afilt ) ) $ to $ \Cont ( \mathbb{T}^n )  $ and we end up in the framework of Proposition  \ref{18thdec}.
\end{example}

\begin{example} \label{18thdec3}
The situation described above can be generalized to quantum homogeneous spaces of $G_q$ (we continue using the same notations as above). Let $\GH$ be a quantum subgroup of $G_q$, given by the surjective Hopf~$^*$\nb-homomorphism $ \chi: \Cont ( G_q ) \rightarrow \Cont (H)  $.  Let $ \C(G_q/H)  := \{ a \in \Cont ( G_q ) \mid  ( \chi \otimes {\rm id} ) (\Delta ( a )) = 1 \otimes a\}. $ We then have a reduced ergodic coaction $ \Delta\mid_{\C(G_q/H)}: \C(G_q/H) \rightarrow \C(G_q/H) \otimes \Cont ( G_q ),$ and an orthogonal filtration of $ \C(G_q/H)  $ resulting from an application of Theorem 3.6 of \cite{MR3066746}. The analogous argument to that in Example \ref{18thdec2} shows that the conditions of Proposition \ref{18thdec} hold also here. 
\end{example}

\begin{example} \label{18thdec4}
 This example deals with the orthogonal filtrations of  Cuntz algebras constructed in Proposition 4.5 in \cite{arnabpreprint}.   Let $N \in \N$, and let $ \mathcal{O}_N $ be the associated Cuntz algebra. We will denote the canonical generators of $ \mathcal{O}_N $ 
 by the symbols $ S_1, \ldots, S_N. $ For a multi-index $ \fatmu = ( \mu_1, \mu_2, \cdots, \mu_k ) $, where $k \in \N$ and $ \mu_i \in \{ 1, 2, \ldots, N \}, $ we write $ S_\fatmu $ for $  S_{\mu_1} \cdots S_{\mu_k} $ and $ \left| \fatmu \right| $ for $ \sum_{i=1}^k \mu_i. $ We have the gauge action $ \gamma $ of $\T$ on $ \mathcal{O}_N $ defined by the equation $ \gamma_z ( S_i ) = z S_i $ where $ z \in \mathbb{T}, i = 1, \ldots, N. $ Then the fixed point algebra $ \mathcal{O}^\gamma_N $ has a natural filtration $ \{ W_k \}_{k\geq 0}$. The authors of \cite{arnabpreprint} construct  further an orthogonal filtration $ \{ V_{k,m} \}_{k,m\geq 0}   $ of $ \mathcal{O}_N $ with respect to the canonical state $ \omega. $ It can be easily seen that the filtration is given by the following formulas $(k\geq 0$):
 $$ V_{k,0} = W_k, ~ V_{k,m} = {\rm Span} ~ \{S_\mu x: x \in W_k, |\mu|=m\}\;\;\; {\rm for} ~ m ~ > ~ 0, $$
$$  V_{k,m} = {\rm Span} ~ \{x S_\mu^* : x \in W_k, |\mu|=-m\} \;\;\; {\rm for}~ m ~ < ~ 0.$$

Suppose we have an action of $\mathbb{T}^n$ on  $ \mathcal{O}_N $ which acts  on each $ S_i $ merely by scalar multiplication. We claim that Proposition \ref{18thdec} applies to such actions. Indeed, it is easy to see that  $ \mathbb{T}^n $ acts by quantum symmetries on $ \mathcal{O}_N.$ Therefore, we have a quantum group morphism from $\T^n$ to $\QISO (\widetilde{\mathcal{O}_N}) $ and we can apply Proposition \ref{18thdec} to obtain an action of $\Z^n$ on $\mathcal{O}_N$ satisfying the conditions in Theorem \ref{maintheoremcrossed}.

The first example of such a group action is of course the  gauge action $ \gamma $ of $\T$ defined above. 
More  generally, following Katsura's prescription in Definition 2.1 of  \cite{MR2016248}, we have actions of $ \mathbb{T}^n $ on $ \mathcal{O}_N $ defined by
  $$ \beta^\chi_z ( S_i ) =  \chi  ( z ) S_i, \;\;\; i=1,\ldots,N,  z  \in  \mathbb{T}^n,  $$
$ \chi $ being a fixed character of $\mathbb{T}^n$ (so an element of $ \mathbb{Z}^n $).

Then we can apply Proposition \ref{18thdec} as mentioned above.
\end{example}

\subsection{The example of the Bunce-Deddens algebra}

This subsection deals with the quantum symmetries of the Bunce-Deddens algebra. Let us recall that the Bunce-Deddens algebra is isomorphic to the crossed product $ A \rtimes_{\beta} \mathbb{Z} $, where $ A $ is the commutative AF algebra of continuous functions on the middle-third Cantor set and $ \beta $ is the odometer action on $ A. $ More precisely, $A$ is an AF algebra arising as the limit of the unital embeddings
\[\mathbb{C}^2 \longrightarrow \mathbb{C}^2 \otimes \mathbb{C}^2 \longrightarrow \mathbb{C}^2 \otimes \mathbb{C}^2 \otimes\mathbb{C}^2 \longrightarrow \,\cdots.\]

Let us recall a multi-index notation for a basis of $ \mathcal{A}_n $ (the $n$-th element of the above sequence), as introduced in \cite{MR2728589}.
For each $n \in \mathbb{N},$  $\Jnd_n$ will denote the set
$\{i_1i_2\cdots i_n: i_j\in\{1,2\} \textrm{ for }j=1,\ldots,n\}$. Multi-indices in $\Jnd:= \bigcup_{n \in \mathbb{N}} \Jnd_n$ will be denoted by capital letters $I,J,\ldots$ and we let the canonical basis of the algebra $\mathcal{A}_n$ built of minimal projections be indexed by elements of $\Jnd_n$. Hence, the basis vectors of $ \mathcal{A}_n$ will be denoted by $ e_I, $ where $ I  $ belongs to $ \Jnd_n. $ 

 Then the natural embeddings $ i_n:\mathcal{A}_n \to\mathcal{A}_{n+1}$ can be  described by the formula
\begin{equation} \label{cantorset0} i_n ( e_I ) =  e_{I1} + e_{I2},\;\;\; I \in \Jnd_n, \end{equation}
where we use the standard concatenation of multi-indices.

The quantum isometry group (in the sense of \cite{goswami2}) of the $\Cst$\nb-algebra $ A $ was studied in \cite{MR2728589}. For more details, we refer to Chapter 5 of \cite{goswamibook2}.  Indeed, we can fix a spectral triple on $A$  coming from the family constructed by Christensen and Ivan (\cite{Chrivan}).  It turns out (Theorem 3.1 of \cite{MR2728589}) that the relevant quantum isometry group of $A$ (with respect to that triple), which we will denote $\mathbb{S}_{\infty}$, is the projective limit of the quantum isometry groups $\mathbb{S}_n$ of the finite dimensional commutative $\Cst$\nb-algebras  $\mathcal{A}_n$, equipped with suitable spectral triples (or, equivalently, viewed as algebras of functions on respective finite Bratteli diagrams, see \cite{MR2728589}).
More precisely,  the algebras $ \C^\univ (  \mathbb{S}_n  ) $ are inductively defined by the formulas
  $$ \C^{\univ} ( \mathbb{S}_1 ) = \C ( \mathbb{Z}_2 ), \;\;\;  \C^\univ ( \mathbb{S}_{n + 1} ) =  \C^\univ ( \mathbb{S}_n ) \ast  \C^\univ ( \mathbb{S}_n ) \oplus  \C^\univ ( \mathbb{S}_n ) \ast  \C^\univ ( \mathbb{S}_n ), \;\;\; n \in \N.  $$
Let us introduce some more notation, following \cite{MR2728589}. The algebra
	 $ \C^\univ ( \mathbb{S}_{n + 1} ) $ is generated by elements $ \{  a_{Kl,Ij}: K,I \in \Jnd_n, l,j = 1,2 \}, $ and the coaction $ \gamma^\univ_{n + 1}  $ of $ \C^\univ ( \mathbb{S}_{n + 1} ) $ on $ \mathcal{A}_{n + 1}  $ is given by the following formula:
	 \begin{equation} \label{cantorset1}
	    \gamma^\univ_{n + 1} ( e_{Ij} ) = \sum_{K \in \Jnd_n, l = 1,2} e_{Kl} \otimes a_{Kl,Ij}, \;\;\; I ~ \in ~ \Jnd_n, ~ j ~ \in ~ \{ 1,2 \}.
	  \end{equation}
The sequence $ (  \C^\univ ( \mathbb{S}_n ) )_{n \in \N} $ defines an inductive system with connecting Hopf~$^*$\nb-homomorphisms $ \phi_{n,m} : \C^\univ ( \mathbb{S}_n ) \rightarrow \C^\univ ( \mathbb{S}_{m} )  $ (if $ n \leq m $) satisfying the equation
   $$ \phi_{nk} = \phi_{mk} \circ  \phi_{nm},\;\;\;  n < m < k. $$
Further one can check that for all $n \in \N$ we have 
	 \begin{equation} \label{cantorset2} ( i_n \otimes \phi_{n,n + 1}   ) \circ \gamma^\univ_n = \gamma^\univ_{n + 1} \circ i_n. \end{equation}
	Combining \eqref{cantorset0} and \eqref{cantorset2}, we see that for  all $n \in \N$, $I,J \in \Jnd_n,$
	  \begin{equation} \label{cantorset2a} \phi_{n, n + 1} ( a_{J,I} ) = a_{J1,I1} + a_{J1,I2} = a_{J2,I1} + a_{J2,I2}.   \end{equation}
	Thus, we have a compact quantum group $ \mathbb{S}_\infty $ arising as a projective  limit of this system, a compact quantum group morphism  from $ \mathbb{S}_\infty $ to $  \mathbb{S}_n$ represented by a Hopf $^*$\nb-homomorphism $ \phi_{n,\infty}\in \Mor(\C^{\univ}(\mathbb{S}_n), \C^{\univ}(\mathbb{S}_\infty))  $ and a canonical coaction $ \gamma^\univ \in \Mor(A,A \otimes \C^\univ ( \mathbb{S}_\infty  )). $

Let  $\omega_{\xi}$ denote the canonical trace on $A$. As explained in Subsection 3.3 of \cite{MR3066746}, $ A $ can be equipped with an orthogonal filtration $ ( A, \omega_\xi, \{ \mathcal{A}_n \setminus \mathcal{A}_{n - 1}\}_{n\in \N} ). $ 
	As a result, by Theorem 3.2 of \cite{MR3066746},  the quantum symmetry group of $ ( A, \omega_\xi, \{ \mathcal{A}_n \setminus \mathcal{A}_{n - 1}\}_n ) $  is isomorphic to $  \mathbb{S}_\infty.$ 
	
Now let us recall the odometer action $ \beta $ on $ A. $  In fact the action $ \beta $ arises as an inductive limit of actions $ \beta_n $ of $ \Gamma_n:= \mathbb{Z}_{2^{n}} $  on $ \mathcal{A}_n. $ Let $ {\bf I}_n $ denote the element $ ( 1 1 1 \cdots 1 ) \in \Jnd_n. $ 
	
We first define inductively homomorphisms  $\sigma_n: \Gamma_n \to {\rm Aut} ( \Jnd_n )$, $n \in \N$. We do it  as follows, for simplicity writing $\sigma_n^i$ for $\sigma_n(i)$: first
the map	$ \sigma_1: \Jnd_1 \rightarrow \Jnd_1 $ is defined as
	 $$ \sigma_1 ( 1 ) = 2, \sigma_1 ( 2 ) = 1.  $$
	In the inductive step, for each $ j = 1,2 $ and $ 0 \leq i \leq 2^n-1, $ the map $ \sigma^i_{ n + 1} : \Jnd_{n + 1} \rightarrow \Jnd_{n + 1} $ is defined by
	\begin{equation} \label{cantorsetminus1}  \sigma^i_{n + 1} (  \sigma^k_n ( {\bf I}_n ) j ) = \sigma^{k + i}_n ( {\bf I}_n ) j \;\;\; {\rm if} \;\;\; 0 \leq k + i \leq 2^n - 1, \; j=1,2,\end{equation} 
\begin{equation} \label{cantorsetminus1b}\sigma^i_{n + 1} ( \sigma^k_n ( {\bf I}_n ) j ) = \sigma^{k + i}_n ( {\bf I}_n ) ( j + 1 ) \;\;\; {\rm if} \;\;\; k + i \geq 2^n, \;j=1,2,  \end{equation}
where $ j + 1 $ is defined mod $ 2 $	and we have used the concatenation of indices.
	It is clear that  $ \Jnd_n = \{ \sigma^i_n  ( {\bf I}_n ): i = 0, 1,2, \ldots, 2^{n} - 1 \}. $

Let $ \{ \delta_j: j = 0,1,2 \ldots, 2^{n - 1}  \}  $ denote the standard basis of the finite-dimensional  (commutative) algebra $ \C ( \Gamma_n ). $
	Then for each $n \in \N$ the coaction $ \beta_n\in \Mor( \mathcal{A}_n,\mathcal{A}_n \otimes \C ( \Gamma_n )) $ is defined as
\begin{equation} \label{cantorset3}
		   \beta_n ( e_{ \sigma^i_n ( {\bf I}_n  )} ) = \sum_{j=0}^{2^n-1}  e_{ \sigma^j_n ( {\bf I}_n  )} \otimes \delta_{j -i }.
		   \end{equation}
We then easily check that we have an inductive system $ ( \C ( \Gamma_n ), \beta_n )_{n \in \N}  $ of coactions, with connecting morphisms $ \psi_{n,m}\in \Mor ( \C ( \Gamma_n ),\C ( \Gamma_m))$ (for $ n < m $) such that for all $n,m, k \in \N$, $n < m < k  $,
$ \psi_{n,k} = \psi_{m,k} \circ \psi_{n,m}$, and a Hopf $^{*}$\nb-homomorphism $ \psi_{n,\infty} \in \Mor (\C ( \Gamma_n), \C_0 ( \mathbb{Z} ) ). $
In particular for $n \in \N$  the map $  \psi_{n, n + 1} $ is given by
\begin{equation} \label{cantorset28thmarch} \psi_{n, n +1} ( \delta^{\Gamma_n}_j ) = \delta^{\Gamma_{n + 1}}_j +  \delta^{\Gamma_{n + 1}}_{ 2^n + j}, \;\;\;j =0,\ldots, 2^n-1. \end{equation}
Thus finally we can define the odometer action using the universal property of the inductive limit, so that for all $n \in \N$
$$ \beta \circ \phi_{n, \infty}:= ( {\rm id} \otimes \psi_{n,\infty}  ) \circ \beta_n.$$
			
In what follows, we will replace the symbols $ \sigma^i_n $ and ${\bf I}_n$ by $ \sigma^i  $ and $\bf I$, respectively, unless there is any risk of confusion.

Thus we have an orthogonal filtration $ \Afilt:=( A, \omega_\xi, \{ \mathcal{A}_n \setminus \mathcal{A}_{n - 1}\}_{n\in \N} ) $ of a unital $\Cst$\nb-algebra $ A $ and an action $ \beta $ of $ \mathbb{Z}  $ on $ A.  $
In order to  apply Theorem \ref{maintheoremcrossed}, we need a Hopf $\ast$-homomorphism $ \pi \in {\rm Mor} ( \C^\univ (  \QISO (\Afilt)  ),  C_0 ( \mathbb{Z} )) $ such that  $ \beta = ( {\rm id} \otimes \pi    ) \circ \gamma^\univ_A. $
			
In order to define $ \pi, $ we begin by defining a quantum group homomorphism from $  \C^\univ ( \mathbb{S}_n ) $ to $ \C ( \Gamma_n ). $
			
\begin{lemma} \label{cantorsetlemma1}
	Let $n \in \N$.
 Define $ \pi_n\in \Mor(\C^\univ ( \mathbb{S}_n ),\C ( \Gamma_n ) )$  by the formula 
$$\pi_n ( a_{\sigma^i ({\bf I} ), \sigma^j ({\bf I}) }  ) = \delta_{i - j},\;\;\;  i,j=1,\ldots,2^n. $$
Then $ \pi_n $ is a Hopf $^{*}$\nb-homomorphism and
\begin{equation} \label{cantorset28thmarch2} ( {\rm id} \otimes \psi_{n, \infty} \circ \pi_n )\circ  \gamma^\univ_n = \beta \circ \phi_{n, \infty}. \end{equation}
\end{lemma}
					
\begin{proof} 
To prove that $ \pi_n $ is a Hopf $^*$-homomorphism, we need to recall the universal property of the quantum group  $ \mathbb{S}_n. $ Let us begin by recalling that both $ \mathbb{S}_{n + 1} $ and $ \Gamma_{n + 1} $ are quantum subgroups of the quantum permutation group $S_{2^{n + 1}}^+. $ Indeed, this follows from \cite{wang}  since both $ \mathbb{S}_{n + 1} $ and $ \Gamma_{n + 1}  $ are compact quantum groups acting on $ A_{n + 1} \approx \mathbb{C}^{2^{n + 1}}. $ In particular, the elements $ \{ a_{IJ}: I, J \in \Jnd_{n + 1}   \}  $ satisfy the magic unitary conditions of the canonical generators of $\C( S_{2^{n + 1}}^+). $ The only extra conditions on $ a_{IJ}  $ are dictated by the following equalities (see equation (2.1) of \cite{MR2728589}):
\begin{equation}
 \label{eq:30Aug18}
 \gamma^\univ_{n + 1} (  i_n (  A_n  )  ) \subseteq i_n ( A_n ) \otimes \C^\univ ( \mathbb{S}_{n + 1}  ).
 \end{equation}
Therefore, if $ \delta $ is an action of a compact quantum group $H $ on $ A_{n + 1}  $ satisfying the condition
$$ \delta (  i_n (  A_n  )  ) \subseteq i_n ( A_n ) \otimes \C^\univ (  H ), $$
then we have a quantum group morphism from $ H $ to $ \mathbb{S}_n $ given by a map $\Phi\in \Mor(\C^{\univ}(\mathbb{S}_n), \C^{\univ}(H))$ such that $ ( {\rm id} \otimes  \Phi  )\circ  \gamma^\univ_{n + 1} =\delta. $ We claim that the action $ \beta_{n + 1} $ of $ \Gamma_{n + 1} $ on $ A_{n + 1} $ satisfies the displayed condition~\eqref{eq:30Aug18}. Indeed, by \eqref{cantorset0}, \eqref{cantorsetminus1} and \eqref{cantorsetminus1b}, for all $ 0 \leq i \leq 2^n - 1 $ we have
				\begin{eqnarray*}
				 \beta_{n + 1} (  i_n ( e_{\sigma^i ( {\bf I}_n  ) }  )  )&=& \beta_{n + 1} (  e_{\sigma^i ({\bf I}  ) 1 }  )  + \beta_{n + 1} ( e_{\sigma^i ( {\bf I}  ) 2 } )\\
		&=& \sum^{2^n - 1}_{j = 0} [ e_{\sigma^j ({\bf I} ) 1 } \otimes \delta_{j - i}  +  e_{\sigma^j ( {\bf I} ) 2 } \otimes \delta_{j - i + 2^n}   +  e_{\sigma^j ({\bf I}  ) 2 } \otimes \delta_{j - i}  + e_{\sigma^j ({\bf I} ) 1 } \otimes \delta_{j - i + 2^n}  ]\\	
			&=& \sum^{2^n - 1}_{j = 0} ( e_{\sigma^j ( {\bf I} ) 1 } +   e_{\sigma^j ( {\bf I}  ) 2 }  ) \otimes ( \delta_{j - i}  +  \delta_{j - i + 2^n} )\\
			&=& \sum^{2^n - 1}_{j = 0} i_n ( e_{\sigma^j ({\bf I} ) } ) \otimes  ( \delta_{j - i}  +  \delta_{j - i + 2^n} ),
				\end{eqnarray*}
				which proves our claim.

					Finally, since $ ( {\rm id} \otimes \pi_n )\circ  \gamma^\univ_n = \beta_n $ holds, we have
					$$ ( {\rm id} \otimes \psi_{n,\infty} \circ \pi_n )\circ  \gamma^\univ_n = ( {\rm id} \otimes \psi_{n,\infty} ) \circ \beta_n = \beta \circ \phi_{n, \infty}. \qedhere $$
	\end{proof}
We naturally need also some compatibility conditions for the morphisms introduced in the last lemma.
	\begin{lemma}  \label{cantorsetlemma2} For each $n \in \N$ the following equality holds:
	 $$ \psi_{n,\infty} \circ \pi_n = \psi_{n + 1, \infty} \circ \pi_{n + 1} \circ \phi_{n, n + 1}. $$
	\end{lemma}
	\begin{proof}
	 We fix $n \in \N$, $j,k \in \{0, \ldots, 2^n-1\}$ and compute
\begin{alignat*}{2}
 &  (\psi_{n + 1, \infty} \circ \pi_{n + 1} \circ \phi_{n, n + 1}) ( a_{\sigma^j (  {\bf I} ), \sigma^k ( {\bf I}) }    )\\
&= (\psi_{n + 1, \infty} \circ \pi_{n + 1}) ( a_{\sigma^j (  {\bf I} ) 1, \sigma^k ( {\bf I} ) 1 } +  a_{\sigma^j ( {\bf I} ) 1, \sigma^k ( {\bf I}) 2 })
& ({\rm by} ~ \eqref{cantorset2a})\\
&= \psi_{n + 1, \infty} ( \delta_{( j - k )}  + \delta_{( 2^{n} + j - k )})
& ({\rm by} ~ \eqref{cantorsetminus1},  \eqref{cantorsetminus1b}~ {\rm and} ~ {\rm Lemma} ~ \ref{cantorsetlemma1} ~  )\\
&= (\psi_{n + 1, \infty} \circ \beta_{n, n +1}) ( \delta_{j - k} )
&  ( {\rm by} ~ \eqref{cantorset28thmarch}    )\\
&= \psi_{n,\infty} ( \delta_{j - k} ) & \\
&= (\psi_{n,\infty} \circ \pi_n) ( a_{\sigma^j (  {\bf I}  ), \sigma^k ( {\bf I} ) }   )
& ({\rm by} ~ {\rm Lemma} ~ \ref{cantorsetlemma1}).\qedhere
\end{alignat*}
\end{proof}
	
We are ready for the final statement, which implies that the approach of Theorem \ref{maintheoremcrossed} can be used to produce an orthogonal filtration of the Bunce-Deddens algebra.
	\begin{proposition}
	There exists a Hopf~$^*$\nb-homomorphism $ \eta \in \Mor (\C^\univ ( \mathbb{S}_\infty ),\C_0 (  \mathbb{Z} )  ) $ such that
		$$ ( {\rm id} \otimes \eta   )\circ \gamma^\univ = \beta. $$
	\end{proposition}

\begin{proof}
	  For each $n \in \N$ put $ \eta_n:= \psi_{n,\infty} \circ \pi_n: \Mor(\C^\univ ( \mathbb{S}_n ), \C_0 ( \mathbb{Z} ) ) .$ 
Then the	existence of the $\Cst$\nb-homomorphism $ \eta $ follows from standard properties of inductive limits once we apply Lemma \ref{cantorsetlemma2}. Since each $ \eta_n $ is a Hopf~$^*$\nb-homomorphism, it can be easily checked that so is $ \eta $. Finally, the equation $ ( {\rm id} \otimes \eta   ) \circ \gamma^\univ  = \beta $ follows from \eqref{cantorset28thmarch2}.
\end{proof}

\subsection{A counterexample} \label{counterexample}

Finally, as announced earlier, we show that the conditions of Theorem \ref{maintheoremcrossed} can be weakened if we are only interested in the existence of an orthogonal filtration, but are actually necessary to obtain the form of the quantum symmetry group  appearing in that theorem.  Let us start by pointing out that a reduced crossed product can have an orthogonal filtration  under more general conditions than  those assumed in Theorem \ref{maintheoremcrossed}.

\begin{proposition} 
 \label{prop15thmarch}
 Let $\Afilt \defeq ( A, \tau_A, ( A_i )_{i \geq 0} ) $ be an orthogonal filtration. Suppose $ \Gamma $ is a finitely generated discrete group having an action $ \beta $ on $A$  such that  $ \tau_A ( \beta_g ( a ) ) = \tau_A ( a )  $ for all $ a  $ in $ A $, $g \in \Gamma$.  Then the triplet $\crfilt \defeq ( A \rtimes_{\beta, \red} \Gamma, \tau,  ( A_{ij} )_{i,j \geq 0}   ) $, with $A_{ij}$ defined via \eqref{Aijdeft} defines an orthogonal filtration of the $\Cst$\nb-algebra $A \rtimes_{\beta, \red} \Gamma.$
\end{proposition}

\begin{proof}
First observe that $A_{00}=\bc.1$ and ${\rm Span} (\cup_{i,j\geq 0} A_{ij})$ is a dense *-subspace of the $\Cst$\nb-algebra  $ A \rtimes_{\beta, \red} \Gamma$. Moreover, it can be easily checked that $\tau$ is faithful on $A \rtimes_{\beta, \red} \Gamma$   using the faithfulness of $\tau_{A}$ and the 
canonical faithful conditional expectation $\tau':A \rtimes_{\beta, \red} \Gamma \to A$.

Now let $i,j,p,q\geq 0$, $(i,j)\neq(p,q)$ and consider $a \in A_{ij}$ and $b \in A_{pq}$. 
We need to prove that   $\tau(a^*b)=0$. By linearity it suffices to assume that $ a = a_i  \lambda_{\gamma_1} $ and $ b = b_p \lambda_{\gamma_{2}} $ with $ a_i \in A_i,  b_p \in A_p, ~ l ( \gamma_{1} ) = j $ and $ l ( \gamma_{2} ) = q.$ We then obtain
   \begin{align*}
    \tau ( a^* b ) 
&=  \tau (  \lambda_{\gamma^{ - 1 }_{1}} a^*_i b_p \lambda_{\gamma_{2}}) =   \tau (  \beta_{\gamma^{ - 1 }_{1}} (  a^*_{i} b_{p}  ) \lambda_{\gamma^{ - 1 }_{1} \gamma_{2} }  )\\
&= \delta_{\gamma_1, \gamma_2}  \tau_{A}  ( \beta_{\gamma^{ - 1 }_{1}} (  a^*_{i} b_{p}  )   )
= \delta_{\gamma_1, \gamma_2}   \tau_{A}  ( a^*_{i} b_{p}  )
= \delta_{jq} \delta_{ip} \delta_{\gamma_1, \gamma_2} \tau_{A} ( a^*_{i} b_{p}  ) = 0.
\end{align*}
 Thus the triplet $ ( A \rtimes_{\beta, \red} \Gamma, \tau,  ( A_{ij} )_{i,j \geq 0}   ) $ defines an orthogonal filtration of the $\Cst$\nb-algebra $A \rtimes_{\beta, \red} \Gamma.$ 
\end{proof}

\begin{remark}
 A sufficient condition for the condition $ \tau_{A} ( \beta_g ( a ) ) = \tau_{A} ( a )  $ to hold is that $ \beta_{g} ( A_i ) \subseteq A_i $  for all~\(g\in\Gamma\), $ i \geq 0. $ However, this is not a necessary condition as the next example shows.
\end{remark}

  For the rest of the section we will consider an example where $A= \Cst ( \Z_9 )$, and $\Gamma = \Z_3$. The elements of each of these cyclic groups will be denoted by  $ \overline{0}, \overline{1}, \overline{2}$, and so on. We fix the (symmetric) generating sets on $\Z_3$ and $\Z_9$, respectively $ \{ \overline{1}, \overline{2} \} $ and $ \{ \overline{1}, \overline{8} \} $, so that each of the $\Cst$\nb-algebras in question is equipped with the orthogonal filtration given by the word-length function associated with the corresponding generating set, as in Example \ref{groupalgexample}. Let $\phi$ be an automorphism of $\Z_9$ of order $3$, given by the formula $\phi(n)=4n$ for $n \in \Z_9$. It induces an action  of $\Z_3$ on $A=\C^*(\Z_9)$, described by the morphism $\beta \in \Mor(A, A \otimes \C(\Z_3))$ via the usual formula ($n \in \Z_9)$:
\begin{equation} \label{formulabeta}
\beta(\lambda_n) = \lambda_n \otimes \delta_{\overline{0}} + \lambda_{\phi(n)} \otimes \delta_{\overline{1}} + \lambda_{\phi^2(n)} \otimes \delta_{\overline{2}}.  \end{equation}
It is easy to verify that $\beta$ preserves the trace $\tau$ (so that Proposition \ref{prop15thmarch} applies), and at the same time considering say $n=1$ we see that $\beta$ does not preserve the individual subspaces in the filtration we defined on $\C^*(\Z_9)$.

\begin{proposition}
Consider the orthogonal filtration $\crfilt$ on the algebra $\Cst(\mathbb{Z}_9) \rtimes_{\beta} \Z_3$ defined by the family $\{A_{ij}:i,j \geq 0\}$ \textup{(as in Proposition \ref{prop15thmarch})}. Then $ \QISO (\crfilt) $ is not isomorphic to the generalized Drinfeld's double of $ \QISO ( \Afilt ) $ and  $ \QISO ( \Bfilt) $ with respect to any bicharacter.  In particular there is no  Hopf~\Star{}homomorphism $ \pi \in \Mor (\C(\QISO(\Afilt),\Cont( \Z_3 )) $ such that $ ( {\rm id} \otimes \pi ) \circ \gamma^\univ_A = \beta. $
\end{proposition}

\begin{proof}
As discussed before, the action $\beta$ comes from an action of $\Z_3$ on $\Z_9$ by automorphisms, which we denote by the same letter.  We will use the identification  $\Cst(\mathbb{Z}_9) \rtimes_{\beta} \Z_3 \cong \Cst(\mathbb{Z}_9 \rtimes_{\beta} \Z_3 )$, under which  the state $ \tau $ can be identified with the canonical tracial state on $ \Cst(\mathbb{Z}_9 \rtimes_{\beta} \Z_3 ).$ We can use the standard (symmetric) generating set  to define the word-length function $l:\mathbb{Z}_9 \rtimes_{\beta} \Z_3 \to \N_0$ and for $n \geq 0$ put  $U_n = {\rm Span} \{ \lambda_t:  t \in \mathbb{Z}_9 \rtimes \Z_3, l(t) = n \}.$ Then $ ( \Cst(\mathbb{Z}_9 \rtimes_{\beta} \Z_3 ), \tau_A, \{ U_n \}_{n\geq 0} ) $ is yet another orthogonal  filtration of $\Cst(\mathbb{Z}_9) \rtimes_{\beta} \Z_3$.  
 Moreover, $\{A_{i,j} \}_{i,j\geq 0}$ is a sub-filtration of $\{ U_n \}_{n\geq0}$ in the sense of \cite{MR3066746} so that by Corollary 2.11 of that paper, $\QISO (\crfilt)$ is a quantum subgroup of  $\QISO (\Cst(\mathbb{Z}_9 \rtimes_{\beta} \Z_3), \tau, \{ U_n \}_{n\geq0} ).$   Now by the first computation in Section 5 of \cite{arnabsingleauthor}, $\C(\QISO (\Cst(\mathbb{Z}_9 \rtimes_{\beta} \Z_3) , \tau, \{ U_n \}_{n\geq 0} ))$ is isomorphic to $ \Cst(\mathbb{Z}_9 \rtimes_{\beta} \Z_3) \oplus \Cst(\mathbb{Z}_9 \rtimes_{\beta} \Z_3 )$, so it  has the  vector space dimension equal $27 + 27 = 54.$ Therefore,  the vector space dimension of $\C(\QISO (\crfilt))$ is no greater than $54.$

On the other hand the Hopf  $\Cst$\nb-algebra of each generalized Drinfeld's double of $ \QISO ( \Afilt) ) $ and $ \QISO (\widetilde{\Cst (\Z_3)}) $  as a vector space is isomorphic to the tensor product of $\C( \QISO (\widetilde{\Cst ( \mathbb{Z}_9 )} )) $ and $\C( \QISO (\widetilde{\Cst(\Z_3)}))$. Since  by Remark \ref{20thmayremark}, $ \C(\QISO ( \widetilde{\Cst ( \mathbb{Z}_n )} )) \cong \C^*(\Z_n) \oplus \Cst (\Z_n)  $ for $ n \neq 4, $ the   vector space dimension of $\C( \QISO ( \widetilde{\Cst ( \mathbb{Z}_9 )} )) \otimes \C(\QISO (\widetilde{ \Cst (\Z_3)})  ) $   equals $ ( 9 + 9 ) ( 3 + 3 ) = 108.$ This completes the proof of the main part of the proposition.

The last statement follows now from Theorem \ref{maintheoremcrossed} (but can be also shown directly).

\end{proof}

\subsection{Further perspectives}
Finally we outline two possible extensions of the results of previous subsections, namely, the cases of twisted crossed products and (twisted) crossed products by discrete quantum groups.

Let us recall the notion of twisted crossed products (see \cite{BNS} and references therein). 
		Let $ \Gamma $ be a discrete group and $\Omega : \Gamma \times\Gamma \rightarrow S^1 $ be a $2$-cocycle on $ \Gamma, $ i.e, a map that satisfies the  equation: 
		$$ \Omega ( g, h ) \Omega ( gh, k ) =  \Omega ( g, h k ) \Omega ( h, k ), \;\;\;g,h,k \in \Gamma. $$
		
		Let us  then write $ \tilde{\Omega} ( h, k ) = \Omega ( k^{-1}, h^{-1} ) $, $h,k \in \Gamma$.
		For each $g \in \Gamma$ define   $ \lambda^\Omega_g $ and $ \rho^{\tilde{\Omega}}_g $ to be the `twisted' left and right shift operators on $ \ell^2 ( \Gamma ) $, given  by
		$$ \lambda^\Omega_g = \tilde{\Omega} ( g^{-1}, \cdot )\lambda_g, ~ \rho^{\tilde{\Omega}}_g = \tilde{\Omega} ( \cdot, g ) \rho_g, $$
		where $\lambda_g$ and $\rho_g$ are the usual shift unitaries acting on  $ \ell^2 ( \Gamma ) $.
		It follows that $ \lambda^\Omega_g \lambda^\Omega_h = \Omega ( g, h ) \lambda^\Omega_{gh}$, $  \rho^{\tilde{\Omega}}_g \rho^{\tilde{\Omega}}_h = \tilde{\Omega} ( g, h ) \rho^{\tilde{\Omega}}_{gh}$ for any $g,h \in \Gamma$. 
		
		The twisted group $\Cst$\nb-algebra $\Cred( \Gamma, \Omega ) $ is defined as the closed linear span of $ \{ \rho^{\tilde{\Omega}}_g\mid g \in \Gamma  \} $ in $ \Bound( \ell^2 ( \Gamma ) ). $  The $ \C^* $-algebra $\Cred( \Gamma, \Omega ) $ is equipped with a canonical coaction of $\Cred( \Gamma )$, i.e.\ the morphism $\delta \in \Mor(\Cred( \Gamma, \Omega ), \Cred( \Gamma, \Omega )\otimes \Cred( \Gamma ))$ given by the formula
		\begin{equation} 
		\label{2ndmarch2018}  
		\delta ( \rho^\Omega_g )\defeq\flip(\multunit[*] (1\otimes \rho^\Omega_g)  \multunit ) = \rho^\Omega_g \otimes \rho_g,  \;\;\; g \in \Gamma,
		\end{equation}
		where the operator $ \multunit \in\U( \ell^2 ( \Gamma \times \Gamma ) ) $ is defined by~$ \multunit \xi ( g, h ) = \xi ( gh, h)$ (with $g,h \in \Gamma$) and $ \sigma $ is the usual flip.
		Now if $ A $ is a unital $\Cst$\nb-algebra and $ \beta\in \Mor(A,A \otimes\Contvin( \Gamma )) $ is a coaction, then the twisted crossed  product $ A \rtimes_{\beta,\red,\Omega} \Gamma $ is defined as the closed linear span of $ {\rm Span} \{  \beta ( A ) ( 1 \otimes \C^*_\red ( \Gamma, \Omega ))   \} $ in $ \Mult ( A \otimes\Comp( \ell^2 ( \Gamma  ) ) )$. 
		
		Following the same line of argument as that in the proof of Proposition~\ref{18thjan1}, we obtain  a \(\Cst\)\nb-algebra isomorphism
		$$\Psi\colon A\boxtimes_{\multunit}\Cred(\Gamma, \omega )\to A\rtimes_{\beta,\red,\Omega}\Gamma $$
		such that for all $a \in A$ and $g \in \Gamma$ we have
		$$ \Psi(j_{A}(a))=\beta(a),~ \Psi(j_{\Cred(\Gamma, \omega )} (\rho^{\Omega}_{g}))=1\otimes\rho^{\Omega}_{g}.$$
		
		Let $ \Afilt\defeq ( A, \tau_A, \{ A_i \}_{i \geq 0}  ) $ denote  an orthogonal filtration of  a unital $\Cst$\nb-algebra $ A$. Let $ \Gamma $ be a discrete group with a proper length function $l.$ Define $ B^{l, \Omega}_n = {\rm Span} \{  \rho^{\tilde{\Omega}}_g: l ( g ) = n  \}. $ Then  \(\Bfilt\defeq ( \Cred(\Gamma, \Omega), \tau_\Gamma, \{B^{l, \Omega}_n\}_{n\in \mathbb{N}}  ) \) is an orthogonal filtration. Suppose~\(\beta\) satisfies~\eqref{compatibility}. Then by the arguments analogous to those of Section~\ref{sec:sec5} we can prove that~$ \QISO ( \crfilt) \cong \GDrin{\bichar},$ where $\GDrin{\bichar}$ is the Drinfeld double of $\QISO(\Afilt)$ and $\QISO(\Bfilt)$ determined by $\bichar$.

Let us finish the article by mentioning that all the results of Section \ref{sec:sec6}, and in particular  Theorem~\ref{maintheoremcrossed}, remain true if we consider actions of finitely generated discrete quantum groups (and length functions on such quantum groups) instead of classical discrete groups. 
Moreover, the results on twisted crossed products also go through for discrete quantum group actions.

\end{document}